\documentclass[reqno,10pt]{amsart}

\usepackage[a4paper,left=25mm,right=25mm,top=30mm,bottom=30mm,marginpar=25mm]{geometry} 
\usepackage{amsmath}
\usepackage{amssymb}
\usepackage{amsthm}
\usepackage{amscd}
\usepackage[ansinew]{inputenc}
\usepackage{color}
\usepackage[final]{graphicx}
\usepackage{tikz}
\usepackage{bbm}
\usepackage{comment}
\usepackage{pgfplots}
\usetikzlibrary{intersections, pgfplots.fillbetween}

\usetikzlibrary{patterns} 
\usetikzlibrary{positioning}
\usetikzlibrary{calc}

\usetikzlibrary{intersections}

\usepackage{float}
 \usepackage[colorlinks=true, pdfstartview=FitV, 
 linkcolor=olive
 ,             citecolor=olive
 , urlcolor=olive
 ]{hyperref}

\renewcommand{\epsilon}{\varepsilon}

\numberwithin{equation}{section}

\newtheoremstyle{thmlemcorr}{10pt}{10pt}{\itshape}{}{\bfseries}{.}{10pt}{{\thmname{#1}\thmnumber{ #2}\thmnote{ (#3)}}}
\newtheoremstyle{thmlemcorr*}{10pt}{10pt}{\itshape}{}{\bfseries}{.}\newline{{\thmname{#1}\thmnumber{ #2}\thmnote{ (#3)}}}
\newtheoremstyle{defi}{10pt}{10pt}{\itshape}{}{\bfseries}{.}{10pt}{{\thmname{#1}\thmnumber{ #2}\thmnote{ (#3)}}}
\newtheoremstyle{remexample}{10pt}{10pt}{}{}{\bfseries}{.}{10pt}{{\thmname{#1}\thmnumber{ #2}\thmnote{ (#3)}}}
\newtheoremstyle{ass}{10pt}{10pt}{}{}{\bfseries}{.}{10pt}{{\thmname{#1}\thmnumber{ A#2}\thmnote{ (#3)}}}

\theoremstyle{thmlemcorr}
\newtheorem{theorem}{Theorem}
\numberwithin{theorem}{section}
\newtheorem{lemma}[theorem]{Lemma}
\newtheorem{corollary}[theorem]{Corollary}
\newtheorem{proposition}[theorem]{Proposition}

\theoremstyle{thmlemcorr*}
\newtheorem{theorem*}{Theorem}
\newtheorem{lemma*}[theorem]{Lemma}
\newtheorem{corollary*}[theorem]{Corollary}
\newtheorem{proposition*}[theorem]{Proposition}
\newtheorem{problem*}[theorem]{Problem}
\newtheorem{conjecture*}[theorem]{Conjecture}

\theoremstyle{defi}
\newtheorem{definition}[theorem]{Definition}

\theoremstyle{remexample}
\newtheorem{remark}[theorem]{Remark}

\theoremstyle{ass}

\newcommand{\om}{\omega}
\newcommand{\Om}{\Omega}

\newcommand{\interiorR}{\mathring R}

\newcommand{\Acal}{\mathcal{A}}

\newcommand{\Dcal}{\mathcal{D}}

\newcommand{\Fcal}{\mathcal{F}}

\newcommand{\Gcal}{\mathcal{G}}

\newcommand{\Ical}{\mathcal{I}}

\newcommand{\Lcal}{\mathcal{L}}

\newcommand{\Ocal}{\mathcal{O}}

\newcommand{\SF}{\mathcal{F}^{\mathrm{st}}}

\newcommand{\Ucal}{\mathcal{U}}
\newcommand{\Vcal}{\mathcal{V}}

\newcommand{\NN}{\mathbb{N}}

\DeclareMathOperator{\diverg}{div}
\DeclareMathOperator{\curl}{curl}
\DeclareMathOperator{\dist}{dist}
\DeclareMathOperator{\rank}{rank}

\DeclareMathOperator{\supp}{supp}

\newcommand{\sfrac}[2]{#1/#2}

\newcommand{\N}{\mathbb{N}}
\newcommand{\R}{\mathbb{R}}
\newcommand{\RR}{\mathbb{R}}

\newcommand{\Q}{\mathbb{Q}}
\newcommand{\Z}{\mathbb{Z}}

\newcommand{\lsc}{\mathrm{sc}}
\newcommand{\qc}{\mathrm{qc}}
\newcommand{\lip}{\mathrm{Lip}}
\newcommand{\loc}{\mathrm{loc}}
\newcommand{\ho}{\mathrm{hom}}

\newcommand{\per}{\mathrm{per}}

\newcommand{\weakly}{\rightharpoonup}
\newcommand{\weaklystar}{\overset{*}\rightharpoonup}

\newcommand{\eps}{\epsilon}
\newcommand{\epsi}{\epsilon}

\newcommand{\BBB}{\color{black}}

\allowdisplaybreaks

 \definecolor{Korange}{rgb}{0.945,0.561,0}
  \definecolor{Kgreen}{rgb}{0.3,0.5,0.3}

 \newcommand{\LM}[2]{\hbox{\vrule width.4pt \vbox to#1pt{\vfill
\hrule width#2pt height.4pt}}}
\newcommand{\LLL}{{\mathchoice {\>\LM{7}{5}\>}{\>\LM{7}{5}\>}{\,\LM{5}{3.5}\,}{\,\LM{3.35}{2.5}\,}}}
\title[\(\Gamma\)-Convergence and stochastic homogenization in the \(\Acal\)-free setting]{\(\Gamma\)-convergence and stochastic homogenization\\ for functionals in
the \(\Acal\)-free setting}

\author[Gianni Dal Maso]{Gianni Dal Maso}
\address[Gianni Dal Maso]{Gianni Dal Maso\\ SISSA, Via Bonomea 265, 34136 Trieste, Italy}
\email{dalmaso@sissa.it}

\author[Rita Ferreira]{Rita Ferreira}
\address[Rita Ferreira]{King Abdullah University of Science and Technology (KAUST), CEMSE Division, Thuwal 23955-6900, Saudi Arabia}
\email{rita.ferreira@kaust.edu.sa}

\author[Irene Fonseca]{Irene Fonseca}
\address[Irene Fonseca]{Irene Fonseca\\ Carnegie Mellon University, 5000 Forbes Avenue, Pittsburgh,
PA 15213}
\email{fonseca@andrew.cmu.edu}

\usepackage[normalem]{ulem}

\begin{document}

%%%%%%%%%%%%%%%%%%%%%% ABSTRACT %%%%%%%%%%%%%%%%%%%%%%%%%%%%%%%%%
 
\maketitle

 \begin{abstract}  
 \vspace{-12pt}   
We obtain a compactness result for \(\Gamma\)-convergence of integral functionals defined on \(\mathcal A\)-free vector fields.
This is used to study homogenization problems for these functionals without periodicity assumptions.
More precisely, we prove that the homogenized integrand can  be obtained by taking limits of minimum values of suitable 
minimization problems on large cubes, when the side length of these cubes tends to \(+\infty\), assuming that these limit values do not depend on the center of the cube. 
Under the usual stochastic periodicity assumptions, this result is then used to solve the stochastic homogenization problem by means of the subadditive ergodic theorem. 
                     
\vspace{8pt}

 

 %\noindent\textsc{Date:} \today.
 \end{abstract}

\tableofcontents

%%%%%%%%%%%%%%%%%%%%%%%%% INTRODUCTION %%%%%%%%%%%%%%%%%%%%%%
\section{Introduction}\label{sect:intro}
Many problems in continuum mechanics and electromagnetism lead
to the study of vector fields \(u\in L^p(D;\R^d)\) that satisfy
a differential constraint of the form 
\begin{equation}\label{diff constr}
\sum_{i=1}^{N} A^i\partial_i u=0\quad\text{in } D,
\end{equation}
where  \(D\subset \R^N\) is a bounded open set and \(A^i\) are
\(l\times d\)-matrices that fulfil the constant-rank property
(see \eqref{eq:rank}).  These vector fields are called \(\Acal\)-free.
The theory of compensated compactness,  developed in  \cite{Mur,
Tar1, Tar2, Tar3, Tar4,
Tar5}, provides powerful tools for their
analysis and has recently been extended to \(\Acal\)-free measures
\cite{Arr, DePh-Rin}. 
 
 When \(f\) is a Carath\'eodory function satisfying the usual
\(p\)-growth condition (see \eqref{eq:pgrowth}),
 the study of the minimization of the integral functional
\begin{equation}\label{functional F intro}
F(u,D):=\int_D f(x,u(x))\,dx
\end{equation}
among all functions \(u\in L^p(D;\R^d)\) that  satisfy the differential
constraint \eqref{diff constr} leads to the notion of \(\Acal\)-quasiconvexity
introduced in \cite{FoMu99} (see Definition \ref{def:Aqcx}), and
inspired by   a slightly different definition found in \cite{Dac}.
This condition is necessary and sufficient for  lower semicontinuity
of \eqref{functional F intro}  with respect to the weak topology of \(L^p(D;\R^d)\) under the
constraint \eqref{diff constr}.  Further results on \(\Acal\)-quasiconvex
functionals can be found in
\cite{Bai-Che-Mat-San, Arr-DePh-Rin} for the   linear growth
case, in \cite{Con-Gme} in the context of partial regularity
of minimizers, in \cite{Bai-Mat-San} in connection with Young
measures, in \cite{Dac-Fon, Dem} when two operators are present,
in \cite{Fon-Leo-Mul} with different exponents in the bounds
 of \(f\) from below and from above, in \cite{Kou-Vik} in connection
with {G}\aa rding inequalities, in \cite{KKKP} for the case of
boundary \(\Acal\)-quasiconvexity, in \cite{Pro} for an extended-valued
function \(f\), in \cite{Rai} for potentials for  \(\Acal\)-quasiconvexity,
and in \cite{Rai02} for the study of relaxation via \(\Acal\)-quasiconvex
       envelopes. \BBB

When \(x\mapsto f(x,\xi)\) is periodic in \(x\), the limit behavior
as \(\eps\to0^+\)
of the minimizers of the functionals
\begin{equation}\label{Feps intro}
F_\eps(u,D):=\int_D f(\tfrac{x}{\eps},u(x))\,dx
\end{equation}
is studied in \cite{BrFoLe00} using \(\Gamma\)-convergence (see
also \cite{Mat-Mor-San} for the  \(p=1 \) case). More precisely,
under an additional \(p\)-Lipschitz 
condition (see \eqref{eq:pLip22}), the family of functionals
\((F_\eps(\cdot,D))_{\eps>0}\), subject to
the differential constraint \eqref{diff constr},
\(\Gamma\)-converges as \(\eps\to 0^+\), with respect to the
weak topology of \(L^p(D;\R^d)\), to a homogenized functional
of the form
\begin{equation}\label{Fhom intro}
F_{\ho}(u,D):=\int_D f_{\ho}(u(x))\,dx,
\end{equation}
subject to the same differential constraint. The homogenized integrand
\(f_{\ho}\) is obtained from \(f\) by solving some auxiliary
minimum problems for \(F(\cdot,Q_r)\) on cubes  \(Q_r\) whose
side length \(r\) tends to infinity. The periodic case was  further developed in
\cite{DaFo16,DaFo16b}, where the authors established 
periodic homogenization results for integral energies under
periodically oscillating or space-dependent differential constraints.

\medskip\noindent \textbf{Aim and main compactness theorem.}
The aim of this paper is to study, more  generally, \(\Gamma\)-convergence
of sequences
 \((F_k)_{k\in \NN}\) of functionals of the form 
\begin{equation}\label{functional Fk intro}
F_k(u,D):=\int_D f_k(x,u(x))\,dx
\end{equation}
subject to the differential constraint \eqref{diff constr}.
 We assume that the integrands \(f_k\)
satisfy
\(p\)-growth and \(p\)-Lipschitz conditions with constants independent
of \(k\),  and 
we prove  a compactness result (see Corollary~\ref{cor:Afree
compactness}): there exist a subsequence, which we do not relabel,
and a functional \(F\) of the form \eqref{functional F intro}
such that for every bounded open set \(D\subset\R^N\), the sequence
of functionals \((F_k(\cdot,D))_{k\in\NN}\), subject to
the differential constraint \eqref{diff constr},
\(\Gamma\)-converges to \(F(\cdot,D)\) with respect to the weak
topology of \(L^p(D;\R^d)\).
Under different hypotheses, \(\Gamma\)-convergence 
results for functionals of the form \eqref{functional Fk intro}
were studied in \cite{Kre-Rin} in the context of dimension reduction
problems.

\medskip\noindent \textbf{Strategy of the proof.}  Following
an idea introduced in \cite{AnDaZe14},  we  prove this
result
 by  first studying  the \(\Gamma\)-convergence
of the sequence of functionals \eqref{functional Fk intro}
without differential constraint  but  with respect to
a topology in \(L^p(D;\R^d)\) that takes into account the 
convergence of
%
%\begin{equation*}
\(
\sum_{i=1}^{N} A^i\partial_i u
\)
%\end{equation*}
%
in \(W^{-1,p}(D;\R^l)\) (see Theorem \ref{thm:Gamma-unconst}).

The proof is based on the usual localization technique 
for \(\Gamma\)-convergence and on a new integral representation
result (see Theorem \ref{thm:unconstrained main}).

To obtain from this result the \(\Gamma\)-convergence in the
\(\Acal\)-free setting, we use a modification procedure
introduced in \cite{FoMu99}, which allows us to replace a sequence
\((u_k)_{k\in\NN}\) with
\(
\sum_{i=1}^{N} A^i\partial_i u_k \to 0
\)
in \(W^{-1,p}(D;\R^l)\)
by a sequence \((v_k)_{k\in\NN}\) with the same limit in the
weak topology of
 \(L^p(D;\R^d)\) and such that
\(
\sum_{i=1}^{N} A^i\partial_i v_k= 0
\)
in \(D\)  for every \(k\), with a negligible modification of
the values of \(F_k\) (see 
Lemmas \ref{lem:equiint} and \ref{lem:pequiint}).

\medskip\noindent
\textbf{Characterising the \(\Gamma\)-limit integrand.} 
 When \(\xi\mapsto f(x,\xi)\) is \(\Acal\)-quasiconvex for a.e.~\(x\in\R^N\), 
 we reconstruct, for a.e.~\(x\in\R^N\) and every \(\xi\in\RR^d\),   the  value of  \(f(x,\xi)\)  via
 the infima of some 
auxiliary minimum problems for \(F(\cdot,Q_\rho(x))/\rho^N\)
on cubes with center \(x\) 
and side length \(\rho\), taking the limit as \(\rho\to 0^+\)
(see Theorem \ref{thm:formula}).
This allows us to characterize the integrand of the \(\Gamma\)-limit
of a sequence 
\eqref{functional Fk intro} by taking limits, as \(k\to\infty\),
of 
the infima of these auxiliary minimum problems for \(F_k(\cdot,Q_\rho(x))/\rho^N\)
(see Theorem \ref{thm:f from M}).  We further  prove
a technical variant of this result (see Theorem \ref{thm:f from
M eta})
in which the 
infima of the auxiliary minimum problems satisfy a subadditivity
condition, preparing the ground for
stochastic applications.

\medskip\noindent
\textbf{Homogenisation without periodicity.}  The preceding
 characterization of the integrand of the \(\Gamma\)-limit
is then used to study the 
homogenization problem for \eqref{Feps intro} without periodicity
assumptions. After a change of variables, the
functionals
\(F_\eps(\cdot,Q_\rho(x))/\rho^N\) are transformed into
\(F(\cdot,Q_{\rho/\eps}(x/\eps))/(\rho/\eps)^N\). Therefore, the
previous results 
show that the family of functionals \((F_\eps(\cdot,D))_{\eps>0}\),
subject to
the differential constraint \eqref{diff constr},
\(\Gamma\)-converges as \(\eps\to 0^+\), with respect to the
weak topology of \(L^p(D;\R^d)\), to the functional \eqref{Fhom
intro} subject to the same differential constraint, assuming only that
there exists
the limit, as \(r\to+\infty\), of the infima of these auxiliary
minimum problems for 
\(F(\cdot,Q_{r}(rx))/r^N\), and that this limit, which defines
\(f_{\ho}\), does not depend on \(x\)
(see Theorem \ref{thm:homogenization eta}).
 These conditions are satisfied not only when \(x\mapsto f(x,\xi)\) is periodic on 
\(\R^{N}\) for every \(\xi\in \R^{d}\), but also  for suitable perturbations of the periodic
case (see Proposition~\ref{periodic + compact support}).

\medskip\noindent
\textbf{Stochastic homogenisation.}  When \(f=f(\omega,x,\xi)\)
depends also on a variable \(\omega\) running on a probability
space, 
under the standard assumptions of stochastic homogenization (see
Definition \ref{sf}), we obtain that
the limits that define \(f_{\ho}(\omega,\xi)\) exist almost surely;
hence,
the family of functionals \((F_\eps(\omega,\cdot,D))_{\eps>0}\),
subject to
the differential constraint \eqref{diff constr},
\(\Gamma\)-converges as \(\eps\to 0^+\), with respect to the
weak topology of \(L^p(D;\R^d)\), to the functional
\begin{equation*}
F_{\ho}(\omega, u,D):=\int_D f_{\ho}(\omega, u(x))\,dx
\end{equation*}
subject to the same differential constraint (see Theorem \ref{prop:
stocf}). Finally, if the stochastic process is ergodic, then
it can be shown that the homogenized integrand does not depend on \(\omega\).

\medskip
Overall, we provide a unified \(\Gamma\)-convergence framework
for integral functionals with
\(\Acal\)-free constraints, covering deterministic, non-periodic,
and stochastic settings in a single
theory.

\section{Notation and preliminaries}\label{sect:notation}

Throughout this paper, \(N,d,l\in\NN\) and \(p, q\in(1,+\infty)\) are fixed,
with \(\frac1p+\frac1q=1\).
We define \(\Ocal(\RR^N)\) to be the collection of all bounded open  subsets of \(\RR^N\) and, for  \(D\in \Ocal(\RR^N)\), we define \(\Ocal(D):=\{D\cap U\colon U \in  \Ocal(\RR^N)\} \) to be the collection of all open subsets of \(D\). For every \(x\in\R^N\) and \(\rho>0\), we consider
the open cube \(Q_\rho(x)\)  with center \(x\), side length \(\rho\), and sides parallel to the coordinate axes.

We use the standard notation for Lebesgue and Sobolev spaces. In particular, for \(m\in\NN\) and \(D\in \Ocal(\RR^N)\), \(W^{-1,p}(D;\RR^m)\)
is the dual of \(W^{1,q}_{0}(D;\RR^m)\).  We also consider the dual of \(W^{1,q}(D;\RR^m)\),
which we  denote by  \(\widetilde W^{-1,p}(D;\RR^m)\). In both cases,
the duality product will be denoted by \(\langle\cdot,\cdot\rangle\).
The norm used in \(W^{1,q}_{0}(D;\RR^m)\) is \(\|u\|_{W^{1,q}_{0}(D;\RR^m)}:=\|\nabla u\|_{L^{q}(D;\RR^{m\times N})}\),
 and the norm used in \(W^{1,q}(D;\RR^m)\) is \(\|u\|_{W^{1,q}(D;\RR^m)}:=(\|u\|^{q}_{L^{q}(D;\RR^{m})}+\|\nabla u\|^{q}_{L^{q}(D;\RR^{m\times N})})^{1/q}\).
The norms in \(W^{-1,p}(D;\RR^m)\) and  \(\widetilde W^{-1,p}(D;\RR^m)\) are the corresponding dual norms.

Several of our  results involve integral functionals whose fields
are subject to   linear partial differential constraints with constant coefficients, known as the \(\Acal\)-free setting as mentioned in the Introduction, which we now make precise. For \(i\in\{1,\dots,N\}\), let \(A^i\in \RR^{l\times d}\) be \(l\times
d\) real-valued matrices such that
there exists \(r\in\NN\) satisfying
\begin{equation}\label{eq:rank}
\rank \left( \sum_{i=1}^N A^i w_i\right) = r
\end{equation}
for all \(w\in\RR^N\setminus\{0\}\). This condition is referred to as the constant-rank property.

For every \(u\in L^p_{\loc}(\R^{N};\RR^d)\), let \(\Acal u\) be the \(\RR^{l}\)-valued distribution on \(\R^{N}\) defined by
\begin{equation}\label{eq:Au on RN}
\langle\Acal u,\psi\rangle:=-\sum_{i=1}^{N}\int_{\RR^{N}}(A^iu)\cdot \partial_i\psi\,
dx\quad\text{for every }\psi\in C^{\infty}_{c}(\R^{N};\RR^{l}).
\end{equation}

Given \(D\in \Ocal(\RR^N)\),
we consider the  operators 
 \[\Acal_D\colon L^p(D;\RR^d)\to W^{-1,p}(D;\RR^{l})\quad  \text{and} \quad\widetilde \Acal_D\colon L^p(D;\RR^d)\to \widetilde W^{-1,p}(D;\RR^{l})\]  that map each \(u\in L^p(D;\RR^d)\) into the 
elements  \(\Acal_D u\) of \(W^{-1,p}(D;\RR^{l})\)  and \(\widetilde\Acal_D u\)  of \(\widetilde W^{-1,p}(D;\RR^{l})\) defined, respectively, by
\begin{eqnarray}\label{eq:Au}
&&\langle\Acal_D u,\psi\rangle:=-\sum_{i=1}^{N}\int_{D}(A^iu)\cdot \partial_i\psi\,
dx\quad\text{for every }\psi\in W^{1,q}_{0}(D;\RR^{l}),
\\
&&\langle\widetilde\Acal_D u,\psi\rangle:=-\sum_{i=1}^{N}\int_{D}(A^iu)\cdot \partial_i\psi\,
dx\quad\text{for every }\psi\in  W^{1,q}(D;\RR^{l}).\label{eq:widetildeAu}
\end{eqnarray}

In particular, \(\ker\Acal_D\) is the set of all \(u\in L^p(D;\RR^d)\) such
that
\begin{equation*}
\sum_{i=1}^{N}\int_{D}(A^iu)\cdot \partial_i\psi\, dx=0\quad\text{for every
}\psi\in W^{1,q}_{0}(D;\RR^{l}),
\end{equation*}
in which case we either write \(\Acal_D u =0\) or \(u\in \ker\Acal_D \).  Moreover,  we have   that
\begin{equation}
\label{eq:normDiv}
\begin{aligned}
\Vert \Acal_D u \Vert_{W^{-1,p}(D;\RR^{l})} = 
\Vert\diverg (Au)\Vert_{W^{-1,p}(D;\RR^{l})},
\end{aligned}
\end{equation}
where \(Au(x):=(A^1u(x),...,A^Nu(x))\in \RR^{l\times N}\) for
\(x\in D\).

The wave cone (or characteristic cone) \(\Lambda\) associated with the operator \(\Acal\) is the subset of \(\R^d\) defined by
\begin{equation}\label{eq:wave}
\Lambda:=\bigcup_{w\in\RR^N\setminus\{0\}} \ker\left( \sum_{i=1}^N A^i w_i\right).
\end{equation}

As extensively studied in \cite{FoMu99}, the weak lower semicontinuity of functionals defined on \(\ker\Acal_D\)  is intimately related to the notion of \(\Acal\)-quasiconvexity, which we recall next. 

\begin{definition}[\(\Acal\)-quasiconvex functions]\label{def:Aqcx}
Let \(Q\subset \RR^N\) be a cube and  \(g \colon \mathbb R^d \to \mathbb R\)  a locally bounded Borel function.
We say that  \(g\) is \(\Acal\)-quasiconvex if
\begin{equation*}
%\label{eq:defAqcx}
\begin{aligned}
g(\xi) \leq \frac1{|Q|}\int_{Q} g(\xi + w(x))\,dx
\end{aligned}
\end{equation*}
 for every \(\xi\in\RR^d\) and
\(w \in C^\infty(\mathbb R^N;\mathbb R^d)\), with  \(w\) \(Q\)-periodic, \(\Acal w=0\) in \(\R^{N}\), and
\(\int_{Q}
w(y)dy = 0\). 
\end{definition}

\begin{remark}\label{rmk:Aqcx} By a change of variables, we see that the previous definition does not depend on the choice of the cube \(Q\). In the unconstrained case, when \(\Acal = 0\), 
\(\Acal\)-quasiconvexity reduces to convexity, while when \(\Acal=\curl\), 
\(\Acal\)-quasiconvexity reduces to quasiconvexity in the sense of Morrey (cf.~\cite{FoMu99}).
By virtue of Jensen's inequality, every convex function is \(\Acal\)-quasiconvex.
\end{remark}

For every \(D\in\Ocal(\RR^N)\), let \(\Vert\cdot\Vert^{\Acal}_D\) be the
norm on \(L^p(D;\RR^d)\) involving  the operator \(\Acal_D\) and defined
by
\begin{equation*}%\label{eq:newnorm}
\Vert u\Vert_D^\Acal:=\Vert u\Vert_{W^{-1,p}(D;\RR^d)}+ \Vert \Acal_D
u\Vert_{ W^{-1,p}(D;\RR^l)}.
\end{equation*}
Note that \(L^p(D;\RR^d)\) endowed with this norm is separable.

\begin{remark}\label{rmk:Lambda-convex}
If \(g\) is \(\Acal\)-quasiconvex and satisfies the \(p\)-growth condition
\begin{equation}\label{eq:pgrowth0}
\frac1{c_0}|\xi|^p - c_0 \leq g(\xi)\leq c_0(1+ |\xi|^{p})\quad\hbox{for  every }\xi\in\R^d
\end{equation}
for  some constant \(c_0\ge1\), then \(g\) is  \(\Lambda\)-convex on \(\R^d\), i.e., 
\begin{equation*}
g(\theta \xi_1+ (1-\theta)\xi_2) \le \theta g(\xi_1) + (1-\theta)g(\xi_2)
\end{equation*}
for every \(\xi_1\), \(\xi_2\in \R^d\) with \(\xi_1-\xi_2\in\Lambda\) and every \(\theta\in[0,1]\) (see, e.g., the proof of \cite[Lemma 2.19]{Arr-DePh-Rin}).
If the vector space \(\mathrm{span}(\Lambda)\) generated by the wave cone \(\Lambda\) coincides with \(\R^d\), then 
\(\Lambda\)-convexity and \eqref{eq:pgrowth0} imply that there exists a constant \(\underline c_1>0\), depending only on \(c_1\), \(p\), and \(d\), such that
the \(p\)-Lipschitz condition
\begin{equation}\label{eq:pLip222}
|g(\xi_1) - g(\xi_2)| \leq  \underline c_1 (1+ |\xi_1|^{p-1} +  |\xi_2|^{p-1})|\xi_1 - \xi_2|
\end{equation}
holds for every \(\xi_1\), \(\xi_2\in \R^d\) (see, e.g., \cite[Lemma 2.3]{KirKris}).

The equality \(\mathrm{span}(\Lambda)=\R^d\) is satisfied in many interesting cases, for instance when \(\Acal=\curl\) or \(\Acal=\diverg\) (see  \cite[Remark 3.3]{FoMu99} for the precise definitions).
\end{remark}

\begin{remark}\label{rmk:constants}
	For functions $g\colon \R^d\to [0,+\infty)$, the  \(p\)-Lipschitz condition can be expressed in many equivalent ways.  Indeed, if $g$ satisfies  \eqref{eq:pLip222}, then we can find a constant $\bar c_1>0$, 
depending only on \(\underline c_1\) and \(p\), such that
	\begin{equation*}
		|g(\xi_1) - g(\xi_2)| \leq \bar c_1\big(1+ (|\xi_1|\land |\xi_2|)^{p-1} +  |\xi_1-\xi_2|^{p-1}\big)|\xi_1 - \xi_2|
	\end{equation*}
holds for every \(\xi_1, \xi_2\in\RR^d\), where \(a\land b:=\min\{a,b\}\). If $g$ satisfies also \eqref{eq:pgrowth0}, the previous inequality implies that there exists a constant \(\hat c_1\ge 1\), depending only on \(c_0\), \( \bar c_1\), and \(p\), such that
 \begin{equation}\label{eq:pLipg}
 	|g(\xi_1) - g(\xi_2)| \leq \hat c_1\big(1+ (g(\xi_1)\land  g(\xi_2))^{\frac{p-1}p} + |\xi_1-\xi_2|^{p-1}\big)|\xi_1
 	- \xi_2|
 \end{equation}
 holds for every \(\xi_1, \xi_2\in\RR^d\).
 
 Conversely, if $g$ satisfies  \eqref{eq:pgrowth0}  and \eqref{eq:pLipg}, we can find a constant \(\check c_1>0\), depending only on \(c_0\), \(\hat c_1\), and \(p\), such that
 \eqref{eq:pLip222} holds with $\underline c_1$ replaced by $\check c_1$.
 
 In the rest of the paper, we prefer to express the \(p\)-Lipschitz condition in the rather unusual form 
 \eqref{eq:pLipg}, because in this inequality the constant is stable under \(\Gamma\)-convergence, while this is not true for the more usual form
 \eqref{eq:pLip222}.
\end{remark}

Next, we introduce the class of integrands that we consider in our analysis. 
Throughout the paper $c_0$ and $c_1$ are fixed constants with $c_0\geq1$ and $c_1\geq  \hat c_1\geq 1$, where $\hat c_1$ is the constant in \eqref{eq:pLipg} corresponding to the constant $\bar c_1$ obtained from $\underline c_1$ in \eqref{eq:pLip222}. 

\begin{definition}\label{def:Frak}
 Let \(\Fcal\) be
the collection of all Carath\'eodory functions  \(f\colon \RR^{N} \times \RR^{d}
\to [0,+\infty)\)  satisfying
the following \(p\)-growth condition:
\begin{equation}\label{eq:pgrowth}
\frac1{c_0}|\xi|^p - c_0 \leq f(x,\xi)\leq c_0(1+ |\xi|^{p})
\end{equation}
for  every \(x\in \RR^N\) and \(\xi\in\RR^d\).
 Let \(\Fcal_{\lip}\) be the collection of all \(f\in\Fcal\) such that
\begin{equation}\label{eq:pLip}
|f(x,\xi_1) - f(x,\xi_2)| \leq c_1\big(1+ (f(x,\xi_1)\land  f(x,\xi_2))^{\frac{p-1}p} + |\xi_1-\xi_2|^{p-1}\big)|\xi_1
- \xi_2|,
\end{equation}
for every \(x\in \RR^N\) and \(\xi_1, \xi_2\in\RR^d\).
Finally, let  \(\Fcal_{\qc}\)  be the collection of all functions \(f\in \Fcal\) such that
\(\xi\mapsto f(x,\xi)\) is \(\Acal\)-quasiconvex for %a.e.\ 
every  \(x\in\R^N\).
 \end{definition}
 
We observe that there is no loss of generality in the constraint on the constant $c_1$ in the preceding definition because if \eqref{eq:pLip} holds for a constant $c_1$, then it also holds for any constant $c\geq c_1$. 

Under our assumptions on \(c_0\), \(c_1\), and \(p\), it can
be shown that the function \(f(x,\xi):=|\xi|^p\) belongs to \(\Fcal_{\qc}\).

\begin{remark}\label{rem: exchangind pLip}
	By exchanging the roles of \(\xi_1\) and \(\xi_2\), we see that \eqref{eq:pLip} is equivalent to
	\begin{equation}\label{eq:pLip asymmetric}
		f(x,\xi_1) \le  f(x,\xi_2)+ 
		c_1\big(1+ f(x,\xi_2)^{\frac{p-1}p} + |\xi_1-\xi_2|^{p-1}\big)|\xi_1
		- \xi_2|
	\end{equation}
	for  every \(x\in \RR^N\) and every \( \xi_1, \xi_2\in\RR^d\).
\end{remark}

\begin{remark}%\label{rem: simpler pLip} 
	As discussed in Remark~\ref{rmk:constants},  there exists a constant \( \check c_1>0\), depending only on \(c_0\), \(c_1\), and \(p\), such that \eqref{eq:pgrowth} and \eqref{eq:pLip} imply 
\begin{equation}\label{eq:pLip22}
|f(x,\xi_1) - f(x,\xi_2)| \leq \check c_1  (1+ |\xi_1|^{p-1} +  |\xi_2|^{p-1})|\xi_1 - \xi_2|
\end{equation}
for  every \(x\in \RR^N\) and every \(\xi_1, \xi_2\in\RR^d\).
  \end{remark}
  %}

 \begin{remark}\label{rem: qc implies Lip} By Remarks~\ref{rmk:Lambda-convex}  
 and \ref{rmk:constants},
in the case where  the vector space \(\mathrm{span}(\Lambda)\) generated by the wave cone \(\Lambda\) coincides with \(\R^d\),  we have that every \(f\in\Fcal_{\qc}\)
satisfies
\eqref{eq:pLip} with \(\hat c_1\) in place of \(c_1\). Since we are assuming throughout that \(\hat c_1\le c_1\), we conclude that 
\(\Fcal_{\qc}\subset \Fcal_{\lip}\) in this case.
\end{remark}

We now introduce the collection of functionals that correspond to integrands in \(\Fcal\).
 
\begin{definition}\label{def:Ffrak}Let \(\{L^p,\Ocal\}\) be the set of pairs \((u,D)\) with  
\(u\in L^p(D,\RR^d)\) and  \(D\in \Ocal(\RR^N)\), and let \(\Ical\) be
the collection of all functionals \(F\colon\{L^p,\Ocal\}\to [0,+\infty)\)
satisfying the following properties for every \(D\in \Ocal(\RR^N)\) and  \(u\in L^p(D;\R^d)\):
\begin{itemize}
\item[(a)]\(\displaystyle \tfrac1{c_0}\Vert u\Vert_{L^p(D;\RR^d)}^p - c_0|D| \le F(u,D) \leq c_0 \big(|D|+ \Vert
u\Vert_{L^p(D;\RR^d)}^{p} \big),\)
\smallskip

\item[(b)] the set function \(B\mapsto F(u,B)\) defined
for every \(B\in \Ocal(D)\) can be extended to a nonnegative
measure defined on all Borel subsets of
\(D\).
%\label{eq:measure on D}
\end{itemize}
Here, and henceforth, if  \(u\in L^p(D;\RR^d)\) and  \(B\in \Ocal(D)\), we  simply write \(F(u,B)\) instead of
\(F(u|_B,B)\).

Let \(\Ical_{\lip}\) be the collection of all functionals \(F\in\Ical\) such that 
\begin{itemize}%\label{eq:pLip F}
\item[(c)] \(\displaystyle |F(u,D) - F(v,D)| \leq  c_1\big(|D|^{\frac{p-1}p} + (F(u,D)\land F(v,D))^{\frac{p-1}p}
+\Vert u-v\Vert_{L^p(D;\R^d)}^{p-1}\big)
\Vert u-v\Vert_{L^p(D;\R^d)}
\)
\end{itemize}
for every \(D\in \Ocal(\RR^N)\) and \(u\), \(v\in L^p(D;\R^d)\).

Finally, let \(\Ical_{\lsc}\) be the collection of all functionals \(F\in\Ical\) such that for every \(D\in\Ocal(\R^N)\), the
functional \(u\mapsto F(u,D)\) is lower semicontinuous for the topology induced on
\(L^p(D;\RR^d)\) by the norm \(\Vert\cdot\Vert^{\Acal}_D\). 
 \end{definition}
 
 \begin{remark}\label{weak Lp versus W-1p}
 Let \(F\in\Ical\).
 By the lower bound in  (a) of Definition \ref{def:Ffrak}, if \(D\in\Ocal(\R^N)\) and \((u_k)_{k\in \NN}\) is a sequence in \(L^p(D;\R^d)\) converging to
 \(u\in L^p(D;\R^d)\) in  the topology induced  by the norm \(\Vert\cdot\Vert^{\Acal}_D\) and such that
 \((F(u_k,D))_{k\in \NN}\) is bounded, then
 \(u_k\weakly u\) weakly in \(L^p(D;\R^d)\).
Therefore, \(F\in \Ical_{\lsc}\) if and only if we have
for every \(D\in\Ocal(\R^N)\) that
\begin{equation*}%\label{seq lsc F}
F(u,D)\le \liminf_{k\to\infty}F(u_k,D)
\end{equation*}
for every \(u\in L^p(D;\R^d)\) and every sequence  \((u_k)_{k\in \NN}\subset L^p(D;\R^d)\) such that \(u_k\weakly u\)
weakly in \(L^p(D;\R^d)\) and \(\Acal_Du_k\to \Acal u\) strongly in \(W^{-1,p}(D;\R^l)\).  In particular, this inequality holds whenever \(u_k\to u\)
strongly in \(L^p(D;\R^d)\). Hence, for every \(F\in \Ical_{\lsc}\) and \(D\in\Ocal(\R^N)\) the functional \(F(\cdot,D)\)
is lower semicontinuous in the strong topology of \(L^p(D;\R^d)\). 
 \end{remark}

\begin{remark}\label{rem:Fcal gives Ffrac}
If \(f\in\Fcal\), then the functional \(F\) defined by
\begin{equation}\label{eq:Ffree}
F(u,D):=\int_D f(x, u(x))\, dx, \quad\text{for every }D\in\Ocal(\R^N)\text{ and }u\in L^p(D;\R^d),
\end{equation}
belongs to \(\Ical\).  If \(f\in\Fcal_{\lip}\), then \(F\in\Ical_{\lip}\).  Indeed, %\eqref{eq:pLip F}
 condition (c) of Definition \ref{def:Ffrak}  follows from \eqref{eq:pLip} by using  H\"older's inequality. If \(f\in\Fcal_{\qc}\), then
 \(F\in\Ical_{\lsc}\) by \cite[Theorem 3.7]{FoMu99} and Remark \ref{weak Lp versus W-1p}.
 \end{remark}

\begin{remark}\label{rem: simpler pLip F}
By analogy with  Remark~\ref{rmk:constants}, there exists a constant \( \check c_1\), depending only on \(c_0\), \(c_1\), and \(p\), such that if \(F\in\Ical_{\lip}\), 
then
\begin{equation}\label{eq:pLip220}
|F(u,D) - F(v,D)| \leq  \check c_1\big(|D|^{\frac{p-1}p} +\Vert u\Vert_{L^p(D;\R^d)}^{p-1}
+\Vert v\Vert_{L^p(D;\R^d)}^{p-1}\big)
\Vert u-v\Vert_{L^p(D;\R^d)}
\end{equation}
for  every \(D\in \Ocal(\R^N)\)  and every \(u, v\in L^p(D;\R^d)\). Conversely, if   the growth conditions (a) of Definition~\ref{def:Ffrak}  %\eqref{eq:pgrowth F} 
and the Lipschitz condition  \eqref{eq:pLip220} hold, then
there  exists a constant \(\hat c_1\), depending only on \(c_0\),  \( \check c_1\), and \(p\), such that
\begin{equation*}
 |F(u,D) - F(v,D)| \leq   \hat c_1 \big(|D|^{\frac{p-1}p} + (F(u,D)\land F(v,D))^{\frac{p-1}p}
+\Vert u-v\Vert_{L^p(D;\R^d)}^{p-1}\big)
\Vert u-v\Vert_{L^p(D;\R^d)},
\end{equation*}
which implies   that the Lipschitz condition (c) of Definition \ref{def:Ffrak} 
 %\eqref{eq:pLip F} 
 holds with \(c_1\) replaced by \(\hat c_1\). As
before, our preference for  condition (c) of Definition \ref{def:Ffrak} 
%\eqref{eq:pLip F} 
is due to the fact that, unlike \eqref{eq:pLip220}, it is stable under \(\Gamma\)-convergence.
 \end{remark}

\begin{remark}\label{rem: exchangindg pLipF}
Similarly to Remark~\ref{rem: exchangind pLip}, we see by exchanging the roles of \(u\) and \(v\) that %\eqref{eq:pLip F}
 the Lipschitz condition (c) of Definition \ref{def:Ffrak} 
 is equivalent to  the inequality 
\begin{equation}\label{eq:pLipF asymmetric}
F(u,D) \le  F(v,D) + c_1\big(|D|^{\frac{p-1}p} + F(v,D)^{\frac{p-1}p}
+\Vert u-v\Vert_{L^p(D;\R^d)}^{p-1}\big)
\Vert u-v\Vert_{L^p(D;\R^d)}
\end{equation}
for  every \(D\in \Ocal(\R^N)\)  and every \(u, v\in L^p(D;\R^d)\).
\end{remark}

 %

%\begin{remark}\label{rmk:geq0}
%If \(f\in \Fcal\) for some \(c_0>0\), then \(f_{c_0}:= f +\ c_0\) is a nonnegative Carath\'eodory 
%function in \(\Fcal_{2c_0}\). This observation, together with the fact that we work on
% bounded subsets of \(\RR^N\), allows us to restrict our analysis  to nonnegative Carath\'eodory 
%integrands, without loss of generality. 
%\end{remark}

\section{\(\Gamma\)-convergence in the unconstrained setting}\label{sect:unconstrained-Gamma}

In this section, we prove a compactness result for  the family
of functionals  \(\Ical\)  introduced in Definition~\ref{def:Ffrak}   with respect to 
\(\Gamma\)-convergence without considering the constraint \(\Acal_D u=0\).
\begin{theorem}\label{thm:Gamma-unconst}  Consider the  three 
families of functionals \(\Ical\), \(\Ical_{\lip}\),  and \(\Ical_{\lsc}\) introduced in Definition~\ref{def:Ffrak}.  Let \((F_k)_{k\in\NN}\)  be a sequence in \(\Ical\).
Then, there exist a subsequence, which we do not relabel,
and a functional  \(F\in\Ical_{\lsc}\) such that for every \(D\in\Ocal(\R^N)\), the sequence
\((F_k(\cdot,D))_{k\in\NN}\)
\(\Gamma\)-converges to \(F(\cdot,D) \vphantom{ \frac1{c_0}\Vert u\Vert_{L^p(D;\RR^d)}^p }\)  with respect to
the
topology induced by \(\Vert\cdot\Vert_D^\Acal\) on \(L^p(D;\RR^d)\).

 If, in addition, \(F_k\in\Ical_{\lip}\) for every \(k\in\NN\), then \(F\in\Ical_{\lip}\).
\end{theorem}

The proof of this theorem relies on the following technical result, which will be used also in the proof of
Proposition \ref{M and Mceta}.

\begin{lemma}\label{lemma: construction og wk}
Let \((F_k)_{k\in\NN}\)  be a sequence in \(\Ical\),
let  \(D_1, D_2, B \in \Ocal (\R^N)\), with \(D_1\subset\subset
D_2\), let \(u\in L^p (D_2\cup B;\RR^d) \), and  let   \((u_k)_{k\in\NN}\subset L^p (D_2;\RR^d)\) and \((v_k)_{k\in\NN}\subset L^p (B;\RR^d)\) be two sequences such that
\begin{equation}
\label{eq:recseq2}
\begin{aligned}
&\lim_{k\to\infty} \Vert u_k -u\Vert^{\Acal}_{D_2}=0 \quad\text{and}\quad \limsup_{k\to\infty}
F_k(u_k, D_2) <+\infty,\\
&\lim_{k\to\infty} \Vert v_k -u\Vert^{\Acal}_B=0 \quad\text{and}\quad  \limsup_{k\to\infty}
F_k(v_k, B) <+\infty.
\end{aligned}
\end{equation}
Then, setting \(D:=D_1\cup B\), there exists a sequence \((w_k)_{k\in\NN}\subset L^p (D;\RR^d)\) such that 
\begin{align}
&\displaystyle w_k=u_k \text{ in } D_1 \quad\text{and}\quad w_k=v_k \text{ in } B\setminus D_2,
\vphantom{ \Vert_{D}^{\Acal}  \lim_{k\to\infty} }
\label{wk = uk and vk}
\\
&\displaystyle
w_k\rightharpoonup u \text{ weakly in }L^p (D;\RR^d) \quad\text{and}\quad   \lim_{k\to\infty} \Vert  w_k -  u\Vert_{D}^{\Acal} =0,
\label{wk A converges to u}
\\
&\displaystyle  \limsup_{k\to\infty} F_k(w_k,D)\leq \limsup_{k\to\infty} \big(F_k(u_k, D_2) +F_k(v_k, B)\big).
\label{Fkwk leq}
\end{align}
\end{lemma}
\begin{proof}
By \eqref{eq:recseq2}, we may assume without loss of generality
that  \(\sup_k F_k(u_k,
D_2)<+\infty\) and \(\sup_k F_k(v_k, B)<+\infty \).
Consequently, by  (a) in Definition~\ref{def:Ffrak},  
 \((u_k)_{k\in\NN}\) is bounded in \(L^p (D_2;\RR^d)\)
 and \((v_k)_{k\in\NN}\) is bounded in \(L^p (B;\RR^d)\). In
turn, % Together with \eqref{eq:recseq2} this gives that
% \(u_k\rightharpoonup u\) weakly in 
%  \(L^p (D_2;\RR^d)\)
% and \(v_k\rightharpoonup u\) weakly in \(L^p (B;\RR^d)\).
this yields
a uniform bound for the total variations of the sequence \((\nu_k)_{k\in\NN}\) of
Radon measures on \(\R^N\) defined by
\begin{equation*}%\label{eq:seqmeas}
\nu_k(E):= \int_{E\cap D_2 \cap B}
\big(1+ |u_k(x)|^p + |v_k(x)|^p\big) \, dx\quad \text{for each Borel set }
E \subset \R^N.
\end{equation*}
Thus, extracting a subsequence,
which we do not relabel, 
there exists a Radon measure \(\nu\)  on \(\R^N\) such that
\begin{equation}
\label{eq:convmeas}
\nu_k \weaklystar \nu \quad \text{weakly\(^*\) in the sense of measures on }\R^N.
\end{equation}

Next, we set  \(\tau:=\dist(D_1, \partial D_2)>0\) and, for \(0<t<\tau\),
we define 
\begin{equation*}
D^t:= \left\{x\in \R^N\colon \dist(x, D_1) < t \right\}\subset\subset D_2.
\end{equation*}
Since \(\nu\) is a
finite measure, we can select \(\eta\in(0,\tfrac\tau2)\) such that
 \(\nu (\partial D^{\eta})=0\) (see \cite[Proposition~1.15]{FoLe07}).
For every \(m\in\NN\) with \(0<\frac1m<\frac\tau2\), we define the sets
\begin{equation*}
L_m:= (D^{\eta+\frac1m}\setminus \overline{ D^{\eta}}) \cap B \subset\subset D_2.
\end{equation*}
We further   consider a cut-off function
\(\theta_m \in C^\infty_c (D^{\eta+\frac1m};[0,1])\) with
\(\theta_m = 1\) on \({\overline{D^{\eta}}}\), and 
%and \(\Vert\nabla\theta_m\Vert_{L^\infty(D_2;\R^N)}
%\leq 2m\).
 define 
\begin{equation*}
\begin{aligned}
w_k^m(x):= \theta_m(x) u_k(x) + (1- \theta_m(x))
v_k(x) \quad\text{for } x\in D=D_1\cup B.
\end{aligned}
\end{equation*}

It is clear that \eqref{wk = uk and vk} holds for any choice of \(m\).

Let \(\psi \in W^{1,q}_0(D;\RR^l)\). Then, 
 using the fact that
\(\theta_m \psi\), \( \partial_i\theta_m \psi \in W^{1,q}_0(D_2;\RR^l)\) and \((1-\theta_m) \psi\), \( \partial_i(1-\theta_m )\psi= -\partial_i\theta_m \psi \in W^{1,q}_0(B;\RR^l)\), we can find a constant, \(c>0\),
depending on \(\Acal\) and \(m\)
but not on 
  \(k\), such that
\begin{equation*}
\begin{aligned}
&-\sum_{i=1}^N\int_{D} A^{i}(w_k^m - u)\cdot\partial_i \psi\, dx =-\sum_{i=1}^N\int_{D} A^{i}\big(\theta_m (u_k - u) +(1-\theta_m) (v_k - u)  \big)\cdot\partial_i \psi\,
dx \\
&\quad= -\sum_{i=1}^N \bigg(\int_{D_2} A^{i}(u_k - u)\cdot\partial_i (\theta_m \psi)\,
dx - \int_{D_2} A^{i}(u_k - u)\cdot(\partial_i
\theta_m \psi)\,
dx \bigg)\\
&\quad\quad -\sum_{i=1}^N \bigg(\int_{B} A^{i}(v_k - u)\cdot\partial_i
((1-\theta_m )\psi)\,
dx + \int_{B} A^{i}(v_k - u)\cdot(\partial_i
\theta_m \psi)\,
dx \bigg)\\
&\quad \leq c\Big( \Vert
\Acal_{D_2}
( u_k - u)\Vert_{ W^{-1,p}(D;\RR^l)} +\Vert  u_k - u\Vert_{W^{-1,p}(D_2;\RR^d)}\\&\quad\quad+\Vert
\Acal_{B}
( v_k - u)\Vert_{ W^{-1,p}(D;\RR^l)} +\Vert  v_k - u\Vert_{W^{-1,p}(B;\RR^d)}  \Big).
\end{aligned}
\end{equation*}
 Consequently,  
we deduce from \eqref{eq:recseq2} that
\begin{equation}
\label{eq:convofw}
\begin{aligned}
\lim_{k\to\infty} \Vert  w_k^m -  u\Vert_{D}^{ \Acal} =0
\quad \text{for every } m\in\NN \text{ with } 0<\tfrac1m<\tfrac\tau2.
\end{aligned}
\end{equation}
%which proves the second part of 
%\eqref{wk A converges to u}.

Finally, for every such  \(m\),  condition (c) and  the second  inequality in (a) of Definition \ref{def:Ffrak} 
together with  \eqref{eq:recseq2}--\eqref{eq:convmeas} yield 
\begin{equation*}
\begin{aligned}
\limsup_{k\to\infty}\  F_k\big(w_k^m,D \big)
&\leq \limsup_{k\to\infty} \big(F_k(u_k, D_2) +F_k(v_k, B)+ F_k(w_k^m,L_m
)\big) \\
&\leq \limsup_{k\to\infty} \big(F_k(u_k, D_2) +F_k(v_k, B) \big)+  c_0\limsup_{k\to\infty} \int_{L_m}
\big(1+ |u_k(x)|^p + |v_k(x)|^p \big) \, dx\\
&\leq \limsup_{k\to\infty} \big(F_k(u_k, D_2) +F_k(v_k, B) \big) +  c_0\,\nu(\overline{L_m}).
\end{aligned}
\end{equation*}
Together with \eqref{eq:convofw}, this implies that for every \(m\in\NN\) with 
\(0<\tfrac1m<\tfrac\tau2\) there exists \(k_m\in\NN\) such that
\begin{equation}\label{Fkwkm}
\Vert  w_k^m -  u\Vert_{D}^{ \Acal} <\tfrac1m \quad\text{and}\quad
F_k(w_k^m,D) < \limsup_{k\to\infty} \big(F_k(u_k, D_2) +F_k(v_k, B)\big) +  c_0\,\nu(\overline{L_m})+\tfrac1m
\end{equation}
for every \(k\ge k_m\). It is not restrictive to assume that \(k_m<k_{m+1}\) for every \(m\).

Define \(w_k:=w_k^m\) for \(k_m\le k <k_{m+1}\). Then, \eqref{Fkwkm} yields
\begin{equation*}
\Vert  w_k -  u\Vert_{D}^{ \Acal} <\tfrac1m \quad\text{and}\quad F_k(w_k,D)
 < \limsup_{k\to\infty} \big(F_k(u_k, D_2) +F_k(v_k, B) \big) +  c_0\,\nu(\overline{L_m})+\tfrac1m
\end{equation*}
for every \(k\ge k_m\). Hence,
\begin{equation*}
\limsup_{k\to\infty} \Vert  w_k -  u\Vert_{D}^{ \Acal} \le \tfrac1m
\quad\text{and}\quad 
\limsup_{k\to\infty} F_k(w_k,D) \le 
\limsup_{k\to\infty} \big(F_k(u_k, D_2) +F_k(v_k, B) \big) +  c_0\,\nu(\overline{L_m})+\tfrac1m.
\end{equation*}
Since \(\lim_m\nu(\overline{L_m})= \nu (\partial D^{\eta})=0\), taking the limit as \(m\to\infty\) in the preceding estimates, we obtain
the second part of \eqref{wk A converges to u} and \eqref{Fkwk leq}.

Recalling (a) of Definition \ref{def:Ffrak}, inequality  \eqref{Fkwk leq}  and \eqref{eq:recseq2}  imply  that 
 \((w_k)_{k\in\NN}\) is bounded in \(L^p (D;\RR^d)\). Therefore, the first part of \eqref{wk A converges to u}
is a consequence of the second one (see Remark~\ref{weak Lp versus W-1p}).
\end{proof}

\begin{proof}[Proof of Theorem \ref{thm:Gamma-unconst}] For each \(D\in\Ocal(\RR^N)\),  we define 
\(F'(\cdot,D),F''(\cdot,D):L^p(D;\RR^d)\to[0,+\infty]\)
by
\begin{equation*}
\begin{aligned}
F'(\cdot,D):=\Gamma(\Vert\cdot\Vert^{\Acal}_D)\text{-}\liminf_{k\to\infty}
F_k(\cdot,D)\quad
\text{and } \quad F''(\cdot,D):=\Gamma(\Vert\cdot\Vert^{\Acal}_D)\text{-}\limsup_{k\to\infty} F_k(\cdot,D),
\end{aligned}
\end{equation*}
where the \(\Gamma\)-limits are taken with
respect to the  topology  induced on
\(L^p(D;\RR^d)\) by the norm \(\Vert\cdot\Vert^{\Acal}_D\). 
If \(u\in  L^p(B;\RR^d)\) for some \(B\in\Ocal(\R^N)\) containing \(D\), we simply write \(F'(u,D)\) and \(F''(u,D)\) instead of
\(F'(u|_D,D)\) and \(F''(u|_D,D)\).

We now proceed in several steps.

\medskip

\textit{Step 1 (Monotonicity of \(F'\) and \(F''\)).} 
We observe that for every \(D_1, D_2\in
\Ocal(\RR^N)\) with \(D_1\subset D_2\),  condition (c) in
Definition~\ref{def:Ffrak}  implies that \(F_k(u,D_1)\le F_k(u,D_2)\) for every 
\(u\in  L^p(D_2;\RR^d)\) and every \(k\in\NN\).
Therefore,  \(F'(u,D_1)\le F'(u,D_2)\) and  \(F''(u,D_1)\le F''(u,D_2)\) for every 
\(u\in  L^p(D_2;\RR^d)\).

\medskip

\textit{Step 2 (Upper bound for \(F''\)).} Let  \(D\in\Ocal(\RR^N)\) and  \(u\in L^p(D;\RR^d)\). Then,
\begin{equation}\label{eq:F''bdd}
\begin{aligned}
F''(u,D)\leq c_0(|D|
+ \Vert u\Vert_{L^p(D;\RR^d)}^p).
\end{aligned}
\end{equation}
Indeed,  the definition of \(F''(\cdot, D)\)
and  the upper bound in (a) of Definition \ref{def:Ffrak} 
%\eqref{eq:pgrowth F}
yield
\begin{equation*}
F''(u,D) \leq \limsup_{k\to\infty} F_k(u,D) \leq c_0(|D|
+ \Vert u\Vert_{L^p(D;\RR^d)}^p),
\end{equation*}
which proves \eqref{eq:F''bdd}.

\medskip

\textit{Step 3 (Nested subadditivity  of \(F''\)).} Let \(D_1,
D_2, B \in \Ocal (\R^N)\), with \(D_1\subset\subset
D_2\), and let  \(u\in L^p (D_2\cup B;\RR^d) \). We want to prove that
\begin{equation}\label{eq:F''sub}
F''(u,D_1 \cup B) \leq F''(u, D_2) + F''(u, B).
\end{equation}
Let   \((u_k)_{k\in\NN}\subset L^p (D_2;\RR^d)\) and 
\((v_k)_{k\in\NN}\subset L^p (B;\RR^d)\) be two sequences such that
\begin{equation*}%\label{eq:recseq}
\begin{aligned}
&\lim_{k\to\infty} \Vert u_k -u\Vert^{\Acal}_{D_2}=0 \quad\text{and}\quad \limsup_{k\to\infty}
F_k(u_k, D_2) = F''(u, D_2)<+\infty,\\
&\lim_{k\to\infty} \Vert v_k -u\Vert^{\Acal}_B=0 \quad\text{and}\quad  \limsup_{k\to\infty}
F_k(v_k, B) = F''(u, B)<+\infty,
\end{aligned}
\end{equation*}
which exist by \cite[Proposition~8.1]{DM93}.
By Lemma~\ref{lemma: construction og wk}, there exists a sequence \((w_k)_{k\in\NN}\subset L^p (D_1\cup B;\RR^d)\) such that 
\eqref{wk A converges to u} and \eqref{Fkwk leq} hold, and so
%Then,%
\begin{equation*}
\begin{gathered}
F''_D(u,D_1 \cup B) \leq \limsup_{k\to\infty} F_k(w_k,D_1
\cup B)
\leq \limsup_{k\to\infty} \big(F_k(u_k, D_2) +F_k(v_k, B)\big)
\le F''(u, D_2) + F''(u, B),
\end{gathered}
\end{equation*}
which proves \eqref{eq:F''sub}.
\medskip

\textit{Step 4 (Compactness property).} Let \(\Dcal\) be the countable collection of the open sets that are 
finite unions of open rectangles with rational vertices. Using the compactness of \(\Gamma\)-convergence
on separable metric spaces (see \cite[Theorem 8.5]{DM93}) and a diagonal argument, we obtain a 
subsequence  of \((F_k)_{k\in\NN}\), which we do not relabel,
for which\begin{equation*}
F'(u,D)=F''(u,D)\quad \text{for every }D\in\Dcal\text{ and every }u\in L^p (D;\RR^d).
\end{equation*}
For every \(D\in\Ocal(\R^N)\) and every \(u\in L^p (D;\RR^d)\), we define
\begin{equation}\label{eq:def-F-}
F(u,D):=\sup_{\substack{D'\in \Dcal\\D'\subset\subset D}} F'(u,D')=\sup_{\substack{D'\in  \Dcal\\D'\subset\subset D}} F''(u,D').
\end{equation}
From the monotonicity of \(F'(u,\cdot)\) and \(F''(u,\cdot)\)
in Step~1, we deduce   for every \(D\in\Ocal(\R^N)\) and every \(u\in L^p (D;\RR^d)\) that 
\begin{equation}\label{inner-regularity-of-F}
F(u,D)=\sup_{\substack{D'\in \Ocal(\R^N)\\D'\subset\subset D}} F'(u,D')=\sup_{\substack{D'\in \Ocal(\R^N)\\D'\subset\subset D}} F''(u,D')
\end{equation}
and
\begin{equation}\label{F le F' le F''}
F(u,D)\le F'(u,D)\le F''(u,D).
\end{equation}
In fact, by \eqref{eq:def-F-}, we clearly have 
\begin{equation}
\label{eq:fleqF'all}
\begin{aligned}
F(u,D) \leq\sup_{\substack{\widetilde D\in \Ocal(\R^N)\\ \widetilde D\subset\subset D}}
F'(u,\widetilde D). 
\end{aligned}
\end{equation}
Conversely, given \(\widetilde D\in \Ocal(\R^N)\) with \(\widetilde
D\subset\subset D\), we can find \(D'\in\Dcal \) such that \(\widetilde
D \subset D' \subset\subset D\). Thus, by the  monotonicity of \(F'(u,\cdot)
\) 
and by \eqref{eq:def-F-}, we conclude that
\begin{equation*}
\begin{aligned}
F'(u,\widetilde D) \leq F'(u,D') \leq F(u, D). 
\end{aligned}
\end{equation*}
Taking the supremum over all sets  \(\widetilde D\in \Ocal(\R^N)\) with \(\widetilde
D\subset\subset D\) yields the converse inequality of \eqref{eq:fleqF'all},
which proves the first identity in \eqref{inner-regularity-of-F}.
The remaining statements in  \eqref{inner-regularity-of-F} and \eqref{F le F' le F''} can be  proven similarly.

Moreover,  by \eqref{eq:F''bdd}, we have 
\begin{equation}\label{eq:Fbdd}
F(u,D)\leq c_0(|D|
+ \Vert u\Vert_{L^p(D;\RR^d)}^p)
\end{equation}
for every \(D\in\Ocal(\RR^N)\) and every  \(u\in L^p(D;\RR^d)\). 

\medskip

\textit{Step 5 (Proof of the \(\Gamma\)-convergence).}  By \eqref{F le F' le F''}, we obtain that \((F_k(\cdot,D))_{k\in\NN}\)
\(\Gamma\)-converges to \(F(\cdot,D)\) once we  prove that
\begin{equation}\label{F'' le F}
F''(u,D)\le F(u,D)
\end{equation}
for every \(D\in\Ocal(\R^N)\) and every \(u\in L^p (D;\RR^d)\). 

Fix any such \(D\) and \(u\),  and fix  \(\eps>0\).  Let \(K\subset D\) be a compact
set  such that
\begin{equation}\label{strip < eps}
c_0(|D\setminus K|
+ \Vert u\Vert_{L^p(D\setminus K;\RR^d)}^p)<\eps.
\end{equation}
Fix  \(D_1, D_2 \in \Ocal (\R^N)\), 
with \(K\subset D_1\subset\subset D_2\subset\subset D\). By \eqref{eq:F''sub} with \(B:=D\setminus K\),  \eqref{eq:F''bdd}, \eqref{inner-regularity-of-F}, and  
\eqref{strip < eps}, we obtain
\begin{equation*}%\label{F'' < F + eps}
F''(u,D)\le F''(u,D_2)+F''(u,D\setminus K)\le F(u,D)+c_0(|D\setminus K|
+ \Vert u\Vert_{L^p(D\setminus K;\RR^d)}^p)\le F(u,D)+\eps.
\end{equation*}
The arbitrariness of \(\eps>0\)  yields \eqref{F'' le F}, completing the proof of \(\Gamma\)-convergence.
In particular, in view of \eqref{inner-regularity-of-F}, \eqref{F le F' le F''}, and \eqref{F'' le F}, we conclude that \(F\) is
inner regular; that is,
  we have for
every \(D\in\Ocal(\R^N)\) and every \(u\in L^p (D;\RR^d)\) that%
\begin{equation}
\label{inner
regularity of F}
\begin{aligned}
F(u,D) =\sup_{\substack{D'\in \Ocal(\R^N)\\D'\subset\subset
D}} F(u,D').
\end{aligned}
\end{equation}
Moreover, by a general property of \(\Gamma\)-limits, we have for every \(D\in\Ocal(\R^N)\) that the functional \(u\mapsto  F(u,D)\)
is lower semicontinuous for the topology induced on
\(L^p(D;\RR^d)\) by the norm \(\Vert\cdot\Vert^{\Acal}_D\). 

\medskip

\textit{Step 6 (Proof of % \eqref{eq:pgrowth 
 condition (a) of Definition~\ref{def:Ffrak}  for \(F\)).} Fix \(D\in\Ocal(\R^N)\) and  \(u\in L^p(D;\R^d)\). By Step~5, there
exists a sequence   \((u_k)_{k\in\NN}\subset L^p (D;\RR^d)\)  such
that
\begin{equation*}%\label{eq:recseq}
\begin{aligned}
&\lim_{k\to\infty} \Vert u_k -u\Vert^{\Acal}_{D}=0 \quad\text{and}\quad
\limsup_{k\to\infty}
F_k(u_k, D) = F(u, D).
\end{aligned}
\end{equation*}
Then, using the lower bound in  %\eqref{eq:pgrowth F} 
 condition (a) of Definition \ref{def:Ffrak}  for \(F_k\), we have for all sufficiently large \(k\in\NN\)  that
\begin{equation*}
\frac1{c_0}\Vert u_k\Vert_{L^p(D;\RR^d)}^p 
- c_0|D| \le F_k(u_k,D) \leq F(u,D) +1. 
\end{equation*}
The preceding estimate and  \eqref{eq:Fbdd} yield that \((u_k)_{k\in\NN}\) is a  bounded sequence in \( L^p(D;\R^d)\).  
Hence, \(u_k \rightharpoonup 
u\) weakly in \( L^p(D;\R^d)\), and so
\begin{equation}\label{eq:lbforF}
\begin{aligned}
F(u, D)= \limsup_{k\to\infty}
F_k(u_k, D)  \geq   \limsup_{k\to\infty} \Big(\frac1{c_0}\Vert u_k\Vert_{L^p(D;\RR^d)}^p 
- c_0|D|\Big) \geq \frac1{c_0}\Vert u\Vert_{L^p(D;\RR^d)}^p 
- c_0|D|.
\end{aligned}
\end{equation}
Finally, \eqref{eq:Fbdd} and \eqref{eq:lbforF}  show that   %\eqref{eq:pgrowth F}
 condition (a) of Definition \ref{def:Ffrak}  holds for \(F\).

\medskip

\textit{Step  7  (Proof of  condition (b) of Definition \ref{def:Ffrak}  for \(F\)).} 
Fix \(D\in \Ocal(\RR^N)\) and  \(u\in L^p(D;\RR^d)\), and
let \(\alpha\colon \Ocal(D) \to [0,+\infty)\) be
the (increasing) set function 
defined by setting  \(\alpha(B):= F( u,B)\) for every \(B\in\Ocal (D)\).
Invoking \cite[Theorem~14.23]{DM93}, 
 condition (b) of Definition \ref{def:Ffrak}  is satisfied  for \(F\) 
provided that \(\alpha\) is subadditive, superadditive, and inner regular
in \(\Ocal(D).\) The inner regularity holds by  \eqref{inner regularity of F}, while the simple proof of the superadditivity can be obtained as in \cite[Proposition~16.12]{DM93}. We are then left to show that \(\alpha\)
is superadditive in \(\Ocal(D)\), which amounts to proving that   
\begin{equation}
\label{eq:alphasuper}
\begin{aligned}
\alpha( B_1 \cup B_2)\leq \alpha (B_1) +\alpha (B_2)  \quad \text{ for all
\( B_1,B_2 \in \Ocal(D)\)}.
\end{aligned}
\end{equation}

Let \( B_1,B_2 \in \Ocal(D)\),  and fix \(\delta>0\). By \eqref{inner regularity of F},
we can find \(B' \subset \subset B_1 \cup B_2\)  such that 
\begin{equation*}
\begin{aligned}
\alpha(B_1\cup B_2) - \delta < F'' (u, B').
\end{aligned}
\end{equation*}
Let   \(B'_1, B''_1, B_2'\in \Ocal(D)\) be such that \(B_1'\subset\subset
B''_1 \subset\subset B_1\),  \(B_2' \subset\subset B_2\), and \(B'
\subset\subset B_1' \cup B'_2\) (see \cite[Lemma~14.20]{DM93}  for instance).
Then, using  \eqref{eq:F''sub}, \eqref{inner regularity of F},  and the monotonicity of \(F''(u,\cdot)\)  proved in Step~1, we obtain
\begin{equation*}
\begin{aligned}
\alpha( B_1 \cup B_2) - \delta < F'' (u, B') \leq F'' (u, B_1'
\cup B'_2) \leq F''(u, B_1'') +F''(u, B_2')\leq \alpha(B_1)
+\alpha(B_2), 
\end{aligned}
\end{equation*}
from which \eqref{eq:alphasuper} follows by letting \(\delta\to0\).

\medskip

\textit{Step  8  (Proof of % \eqref{eq:pLip F} 
 condition (c) of Definition \ref{def:Ffrak}  for \(F\)).}  Assume that \(F_k\in \Ical_{Lip}\) for every \(k\in\NN\)
and fix  \(D\in\Ocal(\R^N)\). By \eqref{eq:pLipF asymmetric} for \(F_k\), we have
\begin{equation}\label{F_k equiLip}
 F_k(u,D) \leq  F_k(v,D) + c_1\big (|D|^{\frac{p-1}p} +
 F_k(v,D)^{\frac{p-1}p} +\Vert u-v\Vert_{L^p(D;\RR^d)}^{p-1}\big)\Vert
u-v\Vert_{L^p(D;\RR^d)}
\end{equation}
 for every \(u, v\in L^p(D;\RR^d)\). We claim that this inequality passes to the  \(\Gamma\)-limit. Indeed, given 
 \(u, v\in L^p(D;\RR^d)\), we can find a sequence \((v_k)_{k\in\NN}\) in  \(L^p(D;\RR^d)\) converging to \(v\)
 in the topology induced on
\(L^p(D;\RR^d)\) by the norm \(\Vert\cdot\Vert^{\Acal}_D\) and such that
\begin{equation}\label{Fkvk to Fv}
 \lim_{k\to\infty} F_k(v_k,D)= F(v,D).
 \end{equation}
For every \(k\in\NN\), let \(u_k:=v_k+u-v\).
By \eqref{F_k equiLip}, we have
\begin{equation}\label{eq:est_Fk}
 F_k(u_k,D) \leq  F_k(v_k,D) + c_1\big (|D|^{\frac{p-1}p}  + F_k(v_k,D)^{\frac{p-1}p} +\Vert u-v\Vert_{L^p(D;\RR^d)}^{p-1}\big)\Vert
u-v\Vert_{L^p(D;\RR^d)}.
\end{equation}
On the other hand, since \((u_k)_{k\in\NN}\) converges to \(u\)
 in the topology induced on
\(L^p(D;\RR^d)\) by the norm \(\Vert\cdot\Vert^{\Acal}_D\), we have
by \(\Gamma\)-convergence that
\begin{equation*}
 F(u,D)\le \liminf_{k\to\infty} F_k(u_k,D).
 \end{equation*}
This inequality, together with   \eqref{Fkvk to Fv} and \eqref{eq:est_Fk}, leads to
\eqref{eq:pLipF asymmetric} for \(F\).  As observed in Remark~\ref{rem: exchangindg pLipF}, we then conclude that
%\eqref{eq:pLip F} 
 condition (c) of Definition \ref{def:Ffrak}  holds  for \(F\).
\end{proof}

We now prove that every functional in \(\Ical_{\lip}\) can be represented by an integral whose integrand belongs to 
\(\Fcal_{\lip}\).
\begin{theorem}\label{thm:unconstrained main} Let \(F\in\Ical_{\lip}\). For every  \(x\in\R^N\) and \(\xi\in\R^d\), we set
\begin{equation}\label{eq:def f(x,xi)}
f(x,\xi):=\limsup_{\rho\to 0^+}\frac{F(\xi,Q_\rho(x))}{\rho^N}. 
\end{equation}
Then, \(f\in\Fcal_{\lip}\) and
\begin{equation}\label{eq:intrep}
F(u,D)= \int_D f(x, u(x))\, dx
\end{equation}
for every  \(D\in\Ocal(\R^N)\) and every \(u\in L^p(D,\RR^d)\). If, in addition,  \(F\in\Ical_{\lip}\cap\Ical_{\lsc}\), then 
the function \(\xi\mapsto f(x,\xi)\) is \(\Acal\)-quasiconvex
for every \(x\in\R^N\); that is,  \(f\in\Fcal_{\lip}\cap\Fcal_{\qc}\).
\end{theorem}

\begin{proof}
By  condition (a) of Definition \ref{def:Ffrak}  and \eqref{eq:def f(x,xi)}, the function \(f\) satisfies \eqref{eq:pgrowth}.  
Since \(F\)  satisfies  \eqref{eq:pLipF asymmetric}  (see Remark \ref{rem: exchangindg pLipF}), we deduce that \eqref{eq:pLip asymmetric} holds for 
\(f\), which is
equivalent to \eqref{eq:pLip} (see  Remark~\ref{rem: exchangind pLip}).

Let us fix \(D\in\Ocal(\R^N)\). By %\eqref{eq:measure on D},
 condition (b) of Definition \ref{def:Ffrak} 
the set function \(B\mapsto F(\xi,B)\) defined on \(\Ocal(D)\)
can be extended to a measure defined on all Borel subsets of \(D\). By   condition (a) of Definition~\ref{def:Ffrak},  this measure is
 absolutely continuous with
respect to the Lebesgue measure. By \eqref{eq:def f(x,xi)} and by the Lebesgue Differentiation Theorem, 
the function \(x\mapsto f(x,\xi)\) is measurable on \(\R^N\) for every \(\xi\in\R^d\), and
\begin{equation}\label{eq:intuk}
F(\xi,B)=\int_B f(x,\xi)\, dx \quad \text{ for every Borel set }
B\subset D.
\end{equation}
The measurability of \(x\mapsto f(x,\xi)\), together with \eqref{eq:pgrowth} and \eqref{eq:pLip},  which encode the continuity of \(\xi\mapsto
f(x,\xi)\),  implies that \(f\in\Fcal\).

Let \(u\colon D\to \R^d\) be a piecewise constant function, that is, there exists a finite family \((B_i)_{i\in I}\) 
of pairwise disjoint sets in \(\Ocal(D)\), covering almost all of \(D\),  and a finite family \((\xi_i)_{i\in I}\) in \(\R^d\) such that for every \(i\in I\) we have \(u(x)=\xi_i\) for every \(x\in B_i\). By  conditions (a) and (b) of Definition \ref{def:Ffrak} 
 %\eqref{eq:pgrowth F} and \eqref{eq:measure on D}, 
 the set function \(B\mapsto F(u,B)\) is a measure that  is absolutely continuous with respect to the Lebesgue measure. By applying \eqref{eq:intuk} to \(\xi_i\) and \(B_i\), we obtain
\begin{equation*}
F(u,D)=\sum_{i\in I}F(\xi_i,B_i)=\sum_{i\in I}\int_{B_i} f(x,\xi_i)\, dx = \int_D f(x,u(x))\, dx.
\end{equation*}

Consider now an arbitrary function \(u\in L^p(D,\RR^d)\). There exists a sequence  \((u_k)_{k\in\NN}\) of piecewise 
constant functions converging to \(u\) in \(L^p(D,\RR^d)\). By the previous step, we have
\begin{equation*}
F(u_k,D) = \int_D f(x,u_k(x))\, dx
\end{equation*}
for every \(k\in\NN\). By   condition (a) of Definition \ref{def:Ffrak}, using  \eqref{eq:pgrowth}, \eqref{eq:pLip22}, and \eqref{eq:pLip220} we can pass to the limit in both terms as \(k\to\infty\) and we obtain \eqref{eq:intrep}.

If \(F\in \Ical_{\lsc}\), then for a.e.~\(x\in\R^N\) the function \(\xi\mapsto f(x,\xi)\) is 
\(\Acal\)-quasiconvex by  \cite[Theorem 3.6]{FoMu99}.
Fix \(x\in\R^N\) and \(\rho>0\), and let \(g\colon\R^{N}\to\R\)
be the function defined by
\begin{equation}\label{new g(xi)}
g(\xi):=\frac{1}{\rho^{N}} \int_{Q_\rho(x)} f(y,\xi)\,dy.
\end{equation}
Let \(Q\subset \RR^N\) be a cube, and let \(w \in C^\infty(\mathbb
R^N;\mathbb R^d)\) be a   \(\) \(Q\)-periodic function, with \(\Acal w=0\)
in \(\R^{N}\) and
\(\int_{Q}
w(y)\,dy = 0\). By \(\Acal\)-quasiconvexity of \(f\), we have  for a.e.\ \(y\in\R^{N}\)
that
\[
f(y,\xi)\le  \frac{1}{|Q|}\int_{Q} f(y,\xi+w(z))\,dz.
\]
Integrating with respect to \(y\) and using Fubini's theorem,
we get
\[
\int_{Q_\rho(x)}f(y,\xi)\,dy \le  \frac{1}{|Q|}\int_{Q} \Big(\int_{Q_\rho(x)}
f(y,\xi+w(z))\,dy\Big)\,dz.
\]
In view of \eqref{new g(xi)}, this gives
\[
g(\xi)\le  \frac{1}{|Q|}\int_{Q}g(\xi+w(z))\,dz.
\]
Hence, \(g\) is \(\Acal\)-quasiconvex.
By  \eqref{eq:intrep}  this    is equivalent to saying that
  the function
\begin{equation*}
\xi\mapsto \frac{F(\xi,Q_\rho(x))}{\rho^N}
\end{equation*}
is  \(\Acal\)-quasiconvex.
 Moreover, by Fatou's lemma for bounded sequences (see \cite[Lemma~1.83~(ii)]{FoLe07}),
it can be similarly checked that the  \(\limsup\)   of  locally equi-bounded
functions preserves \(\Acal\)-quasiconvexity.  
We then 
   deduce from \eqref{eq:def f(x,xi)}   that \(\xi\mapsto
f(x,\xi)\) is 
\(\Acal\)-quasiconvex for every \(x\in\R^N\).
\end{proof} 

We are now in a position to prove a compactness result for the collection of integrands \(\Fcal\).
\begin{corollary}\label{cor:UCmain}
Let  \((f_k)_{k\in\NN}\) be a sequence in  \(\Fcal_{\lip}\)  and let \((F_k)_{k\in\NN}\) be the corresponding sequence of functionals in \(\Ical\) defined by
\eqref{eq:Ffree}.  Then, there exist a subsequence,
which we do not relabel,
and a function  \(f\in \Fcal_{\lip}\cap  \Fcal_{\qc}\) such that for every \(D\in\Ocal(\R^N)\), the sequence
\((F_k(\cdot,D))_{k\in\NN}\)
\(\Gamma\)-converges to \(F(\cdot,D)\) defined by
\eqref{eq:Ffree} with respect to
the topology induced by \(\Vert\cdot\Vert_D^\Acal\) on \(L^p(D;\RR^d)\).
\end{corollary}

\begin{proof}
The result follows from  Remark \ref{rem:Fcal gives Ffrac} and  Theorems \ref{thm:Gamma-unconst}
and \ref{thm:unconstrained main}.
\end{proof}

 The following theorem, which was communicated to us by Jean-Fran\c cois Babadjian, shows that the hypothesis of \(p\)-Lipschitz continuity can be omitted under a condition on the wave cone \(\Lambda\) defined by
\eqref{eq:wave}.
\begin{theorem}\label{th:Babadjian}Assume that the vector space $\mathrm{span}(\Lambda)$ generated by the wave cone $\Lambda$ coincides with $\R^d$. Then, Theorem \ref{thm:unconstrained main} is still satisfied if we replace \(\Ical_{\lip}\) by
\(\Ical_{\lsc}\) and \(\Fcal_{\lip}\) by \(\Fcal_{\lip}\cap\Fcal_{\qc}\). Moreover, the inclusion \(\Ical_{\lsc}\subset \Ical_{\lip}\) holds.
%Moreover \(f\) defined by \eqref{eq:def f(x,xi)} satisfies the Lipschitz condition \eqref{eq:pLip22} for 
%every \(x\in \RR^N\) and \(\xi_1, \xi_2\in\RR^d\), with a constant \(c_2>0\) depending only on \(c_0\), \(p\), and \(d\).
%In other words \(f\in \Fcal_{\lip}\) with \(c_1\) replaced by \(c_2\).
Finally, Corollary  \ref{cor:UCmain} remains valid if we assume only that  \((f_k)_{k\in\NN}\) is a sequence in \(\Fcal\).
\end{theorem}

\begin{proof} Assume that \(F\in\Ical_{\lsc}\) and let \(f\) be defined by \eqref{eq:def f(x,xi)}. 
Arguing as in Theorem \ref{thm:unconstrained main}, we show 
that the function \(x\mapsto f(x,\xi)\) is measurable on \(\R^N\) for every \(\xi\in\R^d\),
that \(f\) satisfies the \(p\)-growth condition \eqref{eq:pgrowth}, and that
\eqref{eq:intrep} holds for every \(D\in\Ocal(\R^N)\) and every piecewise constant 
function \(u \colon D\to \R^d\). 

We claim that
for every \(x\in\R^N\) the function \(\xi\mapsto f(x,\xi)\) is \(\Lambda\)-convex on \(\R^d\), i.e., 
\begin{equation}\label{Lambda-conv}
f(x,\theta \xi_1+ (1-\theta)\xi_2) \le \theta f(x,\xi_1) + (1-\theta)f(x,\xi_2)
\end{equation}
for every \(\xi_1\), \(\xi_2\in \R^d\) with \(\xi_1-\xi_2\in\Lambda\) and every \(\theta\in[0,1]\).
To this aim, we fix \(\xi_1\), \(\xi_2\), and \(\theta\) as required and we consider the periodic extension
\(\chi\colon\R \to [0,1]\) of the characteristic function of \([0,\theta]\). By the definition of wave cone (see \eqref{eq:wave}), there exists \(w\in\R^N\setminus\{0\}\)
such that 
\begin{equation}\label{eq:wave2}
\xi_1-\xi_2\in \ker\left( \sum_{i=1}^N A^i w_i\right).
 \end{equation}
 
We consider the piecewise constant function \(u_\eps\in L^p_{\loc}(\R^{N};\RR^d)\) defined by
\begin{equation*}
u_\eps(x) =  \chi(x\cdot w/\eps)\xi_1 + (1-\chi(x\cdot w/\eps))\xi_2.
\end{equation*}
By the Riemann--Lebesgue Lemma, \(u_\eps\rightharpoonup \theta \xi_1+ (1-\theta)\xi_2\) weakly in \(L^p(D; \R^d)\) for every
\(D\in\Ocal(\R^N)\). Moreover,
\begin{equation}\label{Acal ueps}
\Acal u_\eps =
\frac{1}{\eps}\left( \sum_{i=1}^N A^i w_i\right)(\xi_1-\xi_2)
\sum_{j\in\Z}
(\mathcal H^{N-1}\LLL \{x\cdot w= j\}-\mathcal H^{N-1}\LLL \{x\cdot w= \theta+ j\} )
\end{equation}
in the sense of distributions on \(\R^N\), where \(\mathcal H^{N-1}\) is the \((N-1)\)-dimensional Hausdorff measure and \((H^{N-1}\LLL E)(B)=H^{N-1}(E\cap B)\) for every pair of
Borel sets \(E\), \(B\subset \R^N\).
By \eqref{eq:wave2}, the right-hand side of \eqref{Acal ueps} equals \(0\), hence \(\Acal u_\eps=0\) in the sense of distributions on \(\R^N\).
 
Since \(F\in\Ical_{\lsc}\), for every
\(D\in\Ocal(\R^N)\) the functional \(F(\cdot,D)\) is lower semicontinuous with respect to the norm \(\Vert\cdot\Vert_D^\Acal\) 
on \(L^p(D;\RR^d)\).
Thus, by Remark \ref{weak Lp versus W-1p},
\begin{equation*}
F(\theta \xi_1+ (1-\theta)\xi_2),D) \le \liminf_{\eps\to 0+}
F(u_\eps,D).
 \end{equation*}
 Using the integral representation of \(F\) on piecewise constant functions, from the previous inequality we obtain
 \begin{align*}
F(\theta \xi_1+ (1-\theta)\xi_2),D)  &\le \liminf_{\eps\to 0+}
\int_Df(x,u_\eps(x))\,dx\\
&=\lim_{\eps\to 0+} \int_D\big(\chi(x\cdot w/\eps)f(x,\xi_1) + (1-\chi(x\cdot w/\eps))f(x,\xi_2)\big) \,dx
\\
&=\theta \int_Df(x,\xi_1)\,dx  + (1-\theta) \int_Df(x,\xi_2) \,dx = \theta F(\xi_1,D) + (1-\theta) F(\xi_2,D),
 \end{align*}
 where the second equality follows again from the Riemann--Lebesgue Lemma.
 
 Given \(x\in\R^N\) and \(\rho>0\), we take \(D=Q_\rho(x)\) in the previous inequality. Dividing by \(\rho^N\) and using
 \eqref{eq:def f(x,xi)}, we obtain \eqref{Lambda-conv}, which shows that
the function \(\xi\mapsto f(x,\xi)\) is \(\Lambda\)-convex on \(\R^d\)  for every \(x\in\R^N\).

By \eqref{eq:pgrowth}, we can apply \cite[Lemma 2.3]{KirKris} (see also \cite[Lemma 4.6]{GueRai}) to
obtain that  there exists \(\check c_1>0\), depending only on \(c_0\), \(p\), and \(d\), such that
\eqref{eq:pLip22} holds for  every \(x\in \RR^N\) and every \(\xi_1, \xi_2\in\RR^d\). Since
\(f\) is measurable with respect to \(x\) and satisfies the growth conditions \eqref{eq:pgrowth}, we conclude that
\(f\in \Fcal\).

Thanks to \eqref{eq:pLip22}, the integral representation on piecewise constant functions
leads to the following inequality
\begin{equation*}
F(u,D) \le
\int_Df(x,u(x)) \,dx \quad\hbox{for all }D\in\Ocal(\R^N)\hbox{ and }u\in L^p(D; \R^d)
\end{equation*}
by the lower semicontinuity of \(F(\cdot,D)\) in the strong topology of \(L^p(D; \R^d)\) (see Remark \ref{weak Lp versus W-1p}) and the continuity of \(u\mapsto \int_D f(x,u(x))\,dx\) in the strong topology of \(L^p(D; \R^d)\) due to \eqref{eq:pLip22}. 

To prove the equality, we fix \(u\in  L^p_{\loc}(\R^N; \R^d)\) and 
we use the translation argument introduced in the proof of \cite[Lemma 4.1]{ButDM} (see also \cite[Theorem 21.1]{DM93}). 
We define
\begin{equation*}
G(v,D) = F(u+ v,D) \quad\hbox{for all }D\in\Ocal(\R^N)\hbox{ and }v\in L^p(D; \R^d).
\end{equation*}
We observe that \(G\) satisfies the growth conditions
\begin{itemize}
\item[(a\('\))]\(\displaystyle \frac{1}{C}\Vert v\Vert_{L^p(D;\RR^d)}^p - C\big(|D| + \Vert u\Vert_{L^p(D;\RR^d)}^p\big)\le G(v,D) \leq C\big(|D|+\Vert u\Vert_{L^p(D;\RR^d)}^{p}\big) + C\Vert
v\Vert_{L^p(D;\RR^d)}^{p},\)
\end{itemize}
for a suitable  constant \(C>1\). This is a slight variant of condition (a) of Definition  \ref{def:Ffrak}.

Arguing as before, we obtain a Carath\'eodory function \(g\colon \R^N\times\R^d\to\R\) with \(p\)-growth and \(p\)-Lipschitz
with respect to the second variable such that for every
\(D\in\Ocal(\R^N)\) we have
\begin{equation*}
G(v,D) =
\int_D g(x,v(x)) \,dx  \quad\hbox{if }v\hbox{ is piecewise constant on
}D,
\end{equation*}
from which it follows that
\begin{equation*}
G(v,D) \le 
\int_D g(x,v(x)) \,dx \quad\hbox{for all }v\in L^p(D; \R^d).
\end{equation*}

Let \(D\in\Ocal(\R^N)\) and let \((u_k)_{k\in\NN}\) be a sequence of piecewise constant functions such that \(u_k\to u\) strongly in \(L^p(D; \R^d)\).
Since \(v\mapsto \int_D f(x,v(x))\,dx\) and  \(v\mapsto \int_D g(x,v(x))\,dx\) are continuous in the strong topology of \(L^p(D; \R^d)\), the equalities and inequalities satisfied by \(F(v,D)\), \(G(v,D)\), \(\int_D f(x,v(x))\,dx\) and \(\int_D g(x,v(x))\,dx\) give
\begin{equation*}
\begin{aligned}
F(u,D) &\le
\int_Df(x,u(x)) \,dx
= \lim_{k\to \infty} \int_Df(x,u_k(x)) \,dx
=\lim_{k\to \infty} F(u_k,D)
= \lim_{k\to \infty}
G(u_k-u,D)
\\
&\le \lim_{k\to \infty}
\int_D g(x,u_k(x)-u(x)) \,dx
=\int_Dg(x,0) \,dx = G(0,D) = F(u,D),
\end{aligned}
\end{equation*}
and we obtain
\begin{equation}\label{eq: intrepLp}
F(u,D)= \int_D f(x, u(x))\, dx.
\end{equation}
Since every \(u\in L^p(D; \R^d)\) can be extended to a function of \(L^p_{\loc}(\RR^N; \R^d)\), 
\eqref{eq:intrep} holds for every \(D\in\Ocal(\R^N)\) and \(u\in L^p(D; \R^d)\).

Since \(f\in\Fcal\) and \(F\in\Ical_{\lsc}\), using \cite[Theorem 3.6]{FoMu99} we deduce from \eqref{eq: intrepLp}
that the function \(\xi\mapsto f(x,\xi)\) is 
\(\Acal\)-quasiconvex for a.e.~\(x\in\R^N\). Arguing as in the last part of the proof
of Theorem \ref{thm:unconstrained main}, we obtain that \(\xi\mapsto f(x,\xi)\) is 
\(\Acal\)-quasiconvex for every \(x\in\R^N\), hence \(f\in\Fcal_{\qc}\). By Remark~\ref{rem: qc implies Lip} ,we have
 \(\Fcal_{\qc}\subset \Fcal_{\lip}\), hence \(f\in\Fcal_{\lip}\cap\Fcal_{\qc}\).
This concludes the proof of the
modified version of Theorem \ref{thm:unconstrained main}.

To prove the inclusion \(\Ical_{\lsc}\subset \Ical_{\lip}\), we observe that, by the
modified version of Theorem \ref{thm:unconstrained main} for every \(F\in \Ical_{\lsc}\)
there exists \(f\in \Fcal_{\lip}\cap \Fcal_{\qc}\) such that \eqref{eq:intrep} holds for every  \(D\in\Ocal(\R^N)\) and every \(u\in L^p(D,\RR^d)\). By Remark \ref{rem:Fcal gives Ffrac}, this implies that \(F\in \Ical_{\lip}\).

The modified version of Corollary  \ref{cor:UCmain} follows easily from this modified version of Theorem \ref{thm:unconstrained main}. \BBB
\end{proof}

\BBB
\section{\(\Gamma\)-convergence in the \(\Acal\)-free setting}\label{sect:Gamma}

In this section, we study \(\Gamma\)-convergence in the \(\Acal\)-free setting, i.e., with the constraint \(\Acal u=0\).
We begin with some preliminary lemmas. The following result has been established in \cite[Lemma 2.15]{FoMu99}.

\begin{lemma}\label{lem:equiint} 
Let \(D \in\Ocal(\R^N)\), let \(u\in L^p(D;\RR^d) \), and let 
 \((u_k)_{k\in\NN}\subset L^p(D;\RR^d) \) be a sequence
such that 
\begin{equation*}%\label{eq:initialseq}
u_k \weakly  u  \text{ weakly in \(L^p(D;\RR^d) \)} \enspace
\text{ and } \enspace \Acal_D u_k \to 0 \text{ in \(W^{-1,p}(D;\RR^{l})\).}
\end{equation*}
Then, there exists a \(p\)-equi-integrable sequence \(( v_k)_{k\in\NN}\subset L^p(D;\RR^d)  \) satisfying
\begin{equation}
\label{eq:newseq1}
\begin{gathered}
v_k \weakly  u  \text{ weakly in \(L^p(D;\RR^d) \)}, \quad \Acal_D
v_k =0, \quad \int_D v_k\, dx = \int_D u\,dx,
\\
\lim_{k\to\infty} \Vert u_k - v_k\Vert_{L^r(D;\RR^d)} =
0 \enspace \text{ for all \(1\leq r<p\)}.
\end{gathered}
\end{equation}
If \(D\) is a cube \(Q\), \(u\in L^p_{\loc}(\R^N;\R^d)\) is \(Q\)-periodic, and \(\Acal u=0\) in \(\R^N\), then, in addition to the previous properties, 
we can obtain that \(v_k\in L^p_{\loc}(\R^N;\R^d)\) is \(Q\)-periodic and satisfies \(\Acal v_k=0\) in \(\R^N\).
\end{lemma}

The following result will be used to deduce the \(\Gamma\)-convergence in the \(\Acal\)-free setting from the 
\(\Gamma\)-convergence woth respect to the topology induced by \(\Vert\cdot\Vert_D^\Acal\).
\begin{lemma}\label{lem:pequiint} Let \(D\), \(u\), \((u_k)_{k\in\NN}\), and  \((v_k)_{k\in\NN}\) be as in Lemma \ref{lem:equiint}, and let
\((F_k)_{k\in\NN}\) be a sequence in
\(\Ical_{\lip}\).   
Then,
\begin{equation}
\label{eq:newnewseq2}
\begin{aligned}
&\limsup_{k\to\infty} \big( F_k(v_k,D) - F_k(u_k,D)\big)\le 0.
\end{aligned}
\end{equation}
In particular,
\begin{equation}
\label{eq:newseq2}
\begin{aligned}
&\liminf_{k\to\infty} F_k(v_k,D)  \leq \liminf_{k\to\infty} F_k(u_k,D)
\enspace\hbox{ and }\enspace
 \limsup_{k\to\infty} F_k(v_k,D)  \leq \limsup_{k\to\infty} F_k(u_k,D).
\end{aligned}
\end{equation}
\end{lemma}

\begin{proof}
This result was established
in \cite{BrFoLe00}
when
either \(F_k\) are independent
of \(k\) or \(F_k\) are the functionals associated to functions \(f_k\) with \(f_k(x,\xi)=f(kx,\xi)\) for some
\(f\in \Fcal\) periodic in the first variable.
To prove that \eqref{eq:newnewseq2}  also holds in our setting, we apply Theorem \ref{thm:unconstrained main}
to obtain a sequence \((f_k)_{k\in\NN}\) in
\(\Fcal_{\lip}\)  such that
\begin{equation*}%\label{Fk and fk}
F_k(v,D)=\int_D f_k(x,v(x))\,dx \quad\text{for every }v\in L^p(D;\R^d).
\end{equation*}
We define
\begin{equation}\label{wk and zk}
w_k(x) :=f_k(x,u_k(x)) \quad \text{and} \quad
z_k(x) :=f_k(x, v_k(x)),
\end{equation}
and show that
\begin{equation}
\label{eq:newseq3}
\begin{aligned}
w_k - z_k \to 0 \text{ in measure}.
\end{aligned}
\end{equation}
Indeed, introducing
\begin{equation*}
\begin{aligned}
\alpha:=\frac{q}{q-1}\in (0,1),\quad  s:=\frac1\alpha\in
(1,+\infty), \quad \text{and} \quad t\in (1,+\infty) \quad \text{such that }\frac1s+\frac1t=1,
\end{aligned}
\end{equation*}
we have \(s\alpha=1\) and \(\alpha t=q=\frac{p}{p-1}\); hence,
using \eqref{eq:pLip22}
and H\"older's inequality,
it follows, for some constant \(c\) independent
of \(k\), that
\begin{equation*}
\begin{aligned}
\int_D |w_k - z_k|^\alpha \, dx &\leq c_2^\alpha 
\int_D (1+ |u_k|^{p-1} + | v_k|^{p-1})^\alpha |u_k
- v_k|^\alpha \, dx \\
&\leq c_2^\alpha \Big( \int_D (1+ |u_k|^{p-1} + | v_k|^{p-1})^{\alpha
t}  \, dx \Big)^{\frac1{t}}
\Big( \int_D |u_k - v_k|^{\alpha
s}  \, dx \Big)^{\frac1{s}}\\
&\leq c \Big( \int_D (1+ |u_k|^{p} + |
v_k|^{p})^{}  \, dx \Big)^{\frac\alpha{q}}
\Big( \int_D |u_k - v_k|  \, dx \Big)^{\alpha} \to 0
\end{aligned}
\end{equation*}
by \eqref{eq:newseq1} and the  boundedness of the sequences \((u_k)_{k\in\NN}\)
and \(( v_k)_{k\in\NN}\) in \(L^p(D;\RR^d)\). Thus, \eqref{eq:newseq3}
holds. 

To conclude, we fix \(\delta>0\) and, setting  \(w^+:= \max \{0,w\}\), 
we observe that
\begin{equation*}
\begin{aligned}
0&\leq \int_D (z_k - w_k)^+\, dx = \int_{\{x\in D\colon (z_k - w_k)^+(x)>
\delta\}}( z_k - w_k) \, dx + \int_{\{x\in D\colon(z_k - w_k)^+\leq
\delta\}}(z_k -  w_k)^+\, dx\\
&\leq \int_{\{x\in D\colon| z_k (x)-w_k(x)|>
\delta\}} z_k  \, dx +\delta \Lcal^N(D) , 
\end{aligned}
\end{equation*}
where we used the inequalities  \(-w_k\leq 0\) (by the nonnegativity
of \(f_k\)) and  \((z_k -w_k)^+\leq |z_k -w_k|\). 
By  \eqref{eq:pgrowth} the $p$-equi-integrability of \(( v_k)_{k\in\NN}\) implies the
equi-integrability of \(( z_k)_{k\in\NN}\). Therefore, 
the
preceding estimate and the convergence in measure in \eqref{eq:newseq3}
 yield
\begin{equation*}
\begin{aligned}
\lim_{k\to\infty}  \int_D (z_k - w_k)^+\, dx=0.
\end{aligned}
\end{equation*}
Since \(z_k - w_k\le (z_k - w_k)^+\), we obtain
\begin{equation*}
\begin{aligned}
\limsup_{k\to\infty}
\int_D (z_k - w_k)\, dx \le 0\,,
\end{aligned}
\end{equation*}
which is equivalent to \eqref{eq:newnewseq2} by \eqref{wk and zk}.
The implication \eqref{eq:newnewseq2}\(\Rightarrow\)\eqref{eq:newseq2} is trivial.
\end{proof}

In the proof of Proposition \ref{M and Mceta}, we need the following technical result, which can be obtained from the previous lemmas.

\begin{corollary}\label{cor:pequiint} 
For every \(\eps>0\)  and  \(C>0\)  there exists \(\eta>0\) with the following property: for every open 
cube \(Q\subset\R^{N}\) with side length less than or equal to $1$ \BBB and every \(u\in L^p(Q;\R^d)\), with  \(\|u\|_{L^{p}(Q;\R^d)}^{p}<C|Q|\),
 \(\supp u\subset\subset Q\), and
\(\|\widetilde\Acal_Q u\|_{\widetilde W^{-1,p}(Q;\R^l)}^{p}<\eta|Q|\),
there exists \(v\in L^{p}_{\rm per}(Q;\R^d)\), with 
\(\|v-u\|_{W^{-1,p}(Q;\R^d)}^{p}<\eps|Q|\), \(\Acal v=0\) in \(\R^N\), and \(\int_Q v\,dx=\int_Q u\, dx\),
%\begin{equation*}
%\|v-u\|_{W^{-1,p}(Q;\R^d)}^{p}<\eps|Q|, \quad \Acal_Q v=0\text{ in }\R^N, \quad \int_Q vdx=\int_Q u dx,
%\end{equation*}
such that
\begin{equation}\label{F(v)<F(u)+e}
F(v,Q)<F(u,Q)+\eps|Q|
\end{equation}
for every \(F\in \Ical_{\lip}\).
\end{corollary}
\begin{proof} It is clear that the result does not depend on the center of the cube. We claim that it is enough 
to prove it for the cube \(Q:=Q_{1}(0)\) with center \(0\) and side length \(1\). 
Indeed, if \(Q_\rho=Q_\rho(0)\) is the cube with center \(0\) and side length \(\rho\) and \(u_{\rho}\in L^{p}(Q_{\rho};\RR^{d})\) is
a function with \(\supp u_\rho\subset\subset Q_\rho\), we  consider the rescaled function  \(u\in L^{p}(Q;\RR^{d})\) defined by
\[
u(x):=u_\rho(\rho x)\quad \hbox{for every }x\in Q.
\]

Then, \(\supp u\subset\subset Q\). Moreover, by a change of variables in the integrals, we see that
\[
\|u\|_{L^{p}(Q;\R^d)}^{p}<C|Q|=C \quad\Longleftrightarrow \quad
\|u_\rho\|_{L^{p}(Q_{\rho};\R^d)}^{p}<C|Q_{\rho}|=C\rho^{N}.
\]
Furthermore, as we prove next, if
\begin{equation}\label{norm Aurho}
\|\widetilde\Acal_{Q_{\rho}} u_{\rho}\|_{\widetilde W^{-1,p}(Q_\rho;\R^l)}^{p}<\eta|Q_\rho|=\eta\rho^N,
\end{equation}
then
\begin{equation}\label{norm Au}
\|\widetilde \Acal_Q u\|_{\widetilde W^{-1,p}(Q;\R^l)}^{p}<\eta=\eta|Q|.
\end{equation}
In fact, for every 
\(\psi\in W^{1,q}(Q;\R^{l})\), we consider the function \( \psi_\rho\in W^{1,q}(Q_\rho;\RR^{l})\) defined by
\begin{equation*}%\label{psi psirho}
\psi_\rho(x):=\psi(\tfrac{x}{\rho})\quad \hbox{for every }x\in Q_\rho.
\end{equation*}
We first observe that  
\begin{equation}\label{norm psi in W1q}
\begin{aligned}
\|\psi\|^{q}_{W^{1,q}(Q;\RR^{l})}&=\int_{Q}| \psi(x)|^q\,dx + \int_{Q}|\nabla \psi(x)|^q\,dx=\int_{Q}| \psi_{\rho}(\rho x)|^q\,dx+\rho^{q}\int_{Q}|\nabla \psi_{\rho}(\rho x)|^q\,dx
\\
&=\rho^{-N}\int_{Q_{\rho}}|\psi_{\rho}(x)|^q\,dx+\rho^{q-N}\int_{Q_{\rho}}|\nabla \psi_{\rho}(x)|^q\,dx\ge \rho^{q-N}\|\psi_\rho\|^{q}_{W^{1,q}(Q_{\rho};\RR^{l})},
\end{aligned}
\end{equation}
where  we used the fact that \(0<\rho\le 1\) in the last inequality.
Moreover,
\begin{equation}\label{eq:dualityrel}
\begin{aligned}
\langle\widetilde\Acal_{Q}u,\psi\rangle&=-\sum_{i=1}^{N}\int_{Q}(A^iu(x))\cdot \partial_i\psi(x)\,dx=
-\rho \sum_{i=1}^{N}\int_{Q}(A^iu_{\rho}(\rho x))\cdot \partial_i\psi_{\rho}(\rho x)\,dx
\\
&=-\rho^{1-N} \sum_{i=1}^{N}\int_{Q_{\rho}}(A^iu_{\rho}(x))\cdot \partial_i\psi_{\rho}(x)\,dx
=\rho^{1-N} \langle\widetilde\Acal_{Q_{\rho}}u_{\rho},\psi_{\rho}\rangle.
\end{aligned}
\end{equation}
Thus, \eqref{norm psi in W1q} and \eqref{eq:dualityrel} yield
\[
\frac{\langle\widetilde\Acal_{Q}u,\psi\rangle}{\|\psi\|_{W^{1,q}(Q;\RR^{l})}}
 \le  \frac{\rho^{1-N} \langle\widetilde\Acal_{Q_{\rho}}u_{\rho},
 \psi_{\rho}\rangle}{\rho^{1-\frac{N}{q}}\|\psi_{\rho}\|_{ W^{1,q}(Q_{\rho};\RR^{l})}}
=\frac{ \langle\widetilde\Acal_{Q_{\rho}}u_{\rho},
\psi_{\rho}\rangle}{\rho^{\frac{N}{p}}\|\psi_{\rho}\|_{ W^{1,q}(Q_{\rho};\RR^{l})}},
\]
from which we conclude that\[
\|\widetilde\Acal_{Q}u\|_{\widetilde W^{-1,p}(Q,\RR^{l})}^{p}\le \frac{1}{\rho^{N}}\|\widetilde\Acal_{Q_{\rho}}u\|_{\widetilde W^{-1,p}(Q_{\rho},\RR^{l})}^{p}.
\]
Therefore, \eqref{norm Aurho} implies \eqref{norm Au}. Consequently, \(u\) fulfils all hypotheses required on \(Q\). 

Assuming that the result is proved for \(Q\), let 
\(\rho\in(0,1]\), \(Q_\rho\), and \(u_\rho\) be given
with  \(u_\rho\in L^p(Q_\rho;\R^d)\),
  \(\|u_\rho\|_{L^{p}(Q_\rho;\R^d)}^{p}<C|Q_\rho|\),
 \(\supp u_\rho\subset\subset Q_\rho\), and
\(\|\widetilde\Acal_{Q_\rho} u_\rho\|_{\widetilde W^{-1,p}(Q_\rho;\R^l)}^{p}<\eta|Q_\rho|\).
As above, set \(u(x):=u_\rho (\rho x)\) for \(x\in Q\). Then,
there exists  \(v\in L^{p}_{\rm per}(Q;\R^d)\) satisfying the properties considered in
the statement for \(Q\) relative to \(u\). Let   \(v_\rho \in L^{p}_{\rm per}(Q_\rho;\R^d)\) be defined by
\[
v_\rho(x):=v(\tfrac{x}{\rho})\quad \hbox{for every }x\in Q_\rho.
\]
 Given \(\psi_\rho\in W^{1,q}_{0}(Q_\rho;\RR^{d})\), we consider
\(\psi\in W^{1,q}(Q;\R^{d})\) defined by \(\psi(x):=\psi_\rho(\rho x)\), \(x\in Q\), and observe that
\begin{equation*}%\label{norm psi in W1q0}
\begin{aligned}
\|\psi\|^{q}_{W^{1,q}_{0}(Q;\RR^{d})}&=\int_{Q}|\nabla \psi(x)|^q\,dx=\rho^{q}\int_{Q}|\nabla \psi_{\rho}(\rho x)|^q\,dx
\\
&=\rho^{q-N}\int_{Q_{\rho}}|\nabla \psi_{\rho}(x)|^q\,dx= \rho^{q-N}\|\psi_\rho\|^{q}_{W^{1,q}(Q_{\rho};\RR^{d})}.
\end{aligned}
\end{equation*}
Moreover,
\begin{equation*}
\begin{aligned}
\langle v - u,\psi\rangle&=\int_{Q}(v(x)-u(x))\cdot\psi(x)\,dx=
\int_{Q}(v_{\rho}(\rho x)-u_{\rho}(\rho x))\cdot \psi_{\rho}(\rho x)\,dx
\\
&=\rho^{-N} \int_{Q_{\rho}}(v_{\rho}(x)-u_{\rho}(x))\cdot\psi_{\rho}(x)\,dx
=\rho^{-N} \langle v_{\rho} - u_{\rho},\psi_{\rho}\rangle.
\end{aligned}
\end{equation*}
From the two preceding chains of equalities, we conclude that
\[
\frac{\langle v-u,\psi\rangle}{\|\psi\|_{W^{1,q}_{0}(Q;\RR^{d})}}
=\frac{\rho^{-N} \langle v_{\rho} - u_{\rho},\psi_{\rho}\rangle}{\rho^{1-\frac{N}{q}}\|\psi_{\rho}\|_{W^{1,q}_{0}(Q_{\rho};\RR^{d})}}
=\frac{ \langle v_{\rho}-u_{\rho},\psi_{\rho}\rangle}{\rho^{1+\frac{N}{p}}\|\psi_{\rho}\|_{W^{1,q}_{0}(Q_{\rho};\RR^{d})}}.
\]
Thus,
\[\|v_{\rho}-u_{\rho}\|^{p}_{W^{-1,p}(Q_{\rho};\RR^{d})}\leq \rho^{p+N} \|v-u\|^{p}_{W^{-1,p}(Q;\RR^{d})}.
\]
Because \(\|v-u\|_{W^{-1,p}(Q;\R^d)}^{p}<\eps|Q|=\eps\), we obtain
 for \(0<\rho<1\) that
 \[
\|v_{\rho}-u_{\rho}\|_{W^{-1,p}(Q_{\rho};\R^d)}^{p}<\eps\rho^{p+N}<\eps\rho^{N}=\eps|Q_{\rho}|.
\]

The equalities \(\Acal v_{\rho}=0\) in \(\R^N\) and \(\int_{Q_{\rho}} v_{\rho}\,dx=\int_{Q_{\rho}} u_{\rho} \,dx\) can be obtained from the corresponding properties of \(v\) and \(u\) by a change of variables. As for \eqref{F(v)<F(u)+e} for \(v_{\rho}\) and \(u_{\rho}\),
given  \(F_\rho\in\Ical_{\lip}\), let \(f_\rho\in\Fcal_{\lip}\) be the corresponding integrand (see~Theorem~\ref{thm:unconstrained main}).
Let \(f\in \Fcal_{\lip}\)  be the function defined by
\[
f(x,\xi):= f_{\rho}(\rho x,\xi)\quad\hbox{for every }x\in \R^{N}\hbox{
and }\xi\in\R^{d},
\]
and denote by \(F\in\Ical_{\lip}\)  the corresponding functional (see~Remark~\ref{rem:Fcal gives Ffrac}).
Then, by \eqref{F(v)<F(u)+e} for \(Q\), we have
\begin{equation}
\label{eq:FuFvQ}
\begin{aligned}
F(v,Q)<F(u,Q)+\eps|Q|=F(u,Q)+\eps.
\end{aligned}
\end{equation}
On the other hand, 
\begin{equation*}
\begin{aligned}
&F_\rho(v_{\rho}, Q_\rho) = \int_{Q_\rho} f_\rho(x, v_{\rho}(x))\,dx = \rho^N \int_{Q} f_\rho(\rho x, v_{\rho}(\rho x))\,dx=\rho^N \int_{Q} f( x, v( x))\,dx = \rho^N F(v,Q), \\
& F_\rho(u_{\rho}, Q_\rho) = \int_{Q_\rho} f_\rho(x, u_{\rho}(x))\,dx
= \rho^N \int_{Q} f_\rho(\rho x, u_{\rho}(\rho x))\,dx=\rho^N
\int_{Q} f( x, u( x))\,dx = \rho^N F(u,Q),
\end{aligned}
\end{equation*}
which, together with \eqref{eq:FuFvQ}, yield
\[
F_{\rho}(v_{\rho},Q_{\rho})<F_{\rho}(u_{\rho},Q_{\rho})+\eps\rho^N=
F_{\rho}(u_{\rho},Q_{\rho}) +\eps|Q_{\rho}|.
\]
Since \(F_{\rho}\) is an arbitrary element of \(\Ical_{\lip}\), we obtain \eqref{F(v)<F(u)+e} for \(v_{\rho}\) and \(u_{\rho}\).
This concludes the proof of the claim that it is enough to prove the corollary for 
\(Q:=Q_{1}(0)\), which we establish next.

We argue by contradiction. Assume
that the statement for \(Q\) is false. Then, there  exist
\(\eps>0\) and \(C>0\)  such that for every \(k\in\NN\) there
exists  \(u_k\in L^p(Q;\R^d)\), with  \(\|u_k\|_{L^{p}(Q;\R^d)}^{p}<C\),
 \(\supp u_k\subset\subset Q\), and
\(\|\widetilde\Acal_Q u_k \|_{\widetilde W^{-1,p}(Q;\R^l)}^{p}<1/k\),
such that for every
\(v\in L^{p}_{\rm per}(Q;\R^d)\), with \(\|v-u_k\|_{W^{-1,p}(Q;\R^d)}^{p}<\eps\),
 \(\Acal v=0\) in \(\R^N\), and \(\int_Qv\,dx=\int_Qu_{k}\,dx\),
there exists \(F_{k,v}\in \Ical_{\lip}\) such that
\begin{equation}\label{eq:F(v) ge F(u)+e}
F_{v,k}(v,Q)\ge F_{v,k}(u_k,Q)+\eps.
\end{equation}
 Since \((u_k)_{k\in \NN}\) is bounded in \(L^{p}(Q;\R^d)\),
 a subsequence of \((u_k)_{k\in \NN}\), not relabeled, converges
weakly 
in \(L^p(Q;\R^d)\) to a function \(u\in L^p(Q;\R^d)\).
We extend each \(u_k\) to a \(Q\)-periodic function, still denoted
\(u_k\). 
Then, for every \(D\in\Ocal(\R^N)\), \(u_k\) converges weakly in
\(L^p(D;\R^d)\) 
to the periodic extension of \(u\), still denoted by \(u \). 
To prove that  \(\Acal u=0\) in \(\R^N\), we start by setting
 \(Q_m:=Q_m(0)\) for every \(m\in\N\), and we show that
\begin{equation}\label{AQm}
\|\widetilde\Acal_{Q_{m}} u_k \|_{\widetilde W^{-1,p}(Q_{m};\R^l)}<\frac{m^N}{k^{1/p}}.
\end{equation}
To prove this inequality, we define \(A_{m}:=\{1,\dots,m\}^{N}\)
and for every  \(\alpha=(\alpha_{1},\dots,\alpha_{N})\in A_{m}\),
we set
\[x(\alpha):=\Big(-\frac{m}{2}-\frac{1}{2}+\alpha_{1},\dots,
-\frac{m}{2}-\frac{1}{2}+\alpha_{N}\Big).
\]
We observe that
\[
\overline{Q}_m=\bigcup_{\alpha\in A_{m}}\overline{Q}_{1}(x(\alpha)).
\]
Let \(\psi\in W^{1,q}(Q_{m};\R^l)\).  For every \(\alpha\in
A_{m}\) , we define \(\psi_{\alpha}\in W^{1,q}(Q;\R^l)\) by
\(
\psi_{\alpha}(x):= \psi(x+x(\alpha))
\).
Using the \(Q\)-periodicity of \(u_k\), we get
from \eqref{eq:widetildeAu} that
\[
\langle\widetilde\Acal_{Q_{m}} u_k,\psi\rangle=-\sum_{i=1}^{N}\int_{Q_{m}}
(A^{i}u_{k})\cdot\partial_{i}\psi\,dx
=-\sum_{i=1}^{N}\sum_{\alpha\in A_{m}}\int_{Q} (A^{i}u_{k})\cdot\partial_{i}\psi_{\alpha}\,dx
=\sum_{\alpha\in A_{m}}\langle\widetilde\Acal_{Q} u_k,\psi_{\alpha}\rangle.
\]
Since \(\|\widetilde\Acal_Q u_k \|_{\widetilde W^{-1,p}(Q;\R^l)}^{p}<1/k\),
we deduce that
\[
\big|\langle\widetilde\Acal_{Q_{m}} u_k,\psi\rangle\big| \leq
\frac{1}{k^{1/p}}\sum_{\alpha\in A_{m}}\|\psi_{\alpha}\|_{W^{1,q}(Q;\R^l)}
\leq \frac{m^{N}}{k^{1/p}}\|\psi\|_{W^{1,q}(Q_m;\R^l)},
\]
which concludes the proof of \eqref{AQm}.
Because  \(u_k\) converges to \(u\) weakly in
\(L^p(Q_{m};\R^{d})\), 
we obtain from  \eqref{AQm} that
\(\widetilde\Acal_{Q_{m}} u=0\) for every  \(m\in\N\). We then
conclude that \(\Acal u=0\) in \(\R^N\).

By Lemma \ref{lem:equiint}, there exists a sequence \((v_k)_{k\in
\NN}\) in  
\(L^p_{\rm per}(Q;\R^d)\) satisfying \eqref{eq:newseq1} with
\(D=Q\).
Since the embedding of \(L^p(Q;\R^d)\) into \(W^{-1,p}(Q;\R^d)\)
 is compact, 
the sequences \((v_k)_{k\in \NN}\) and \((u_k)_{k\in \NN}\) converge to \(u\) strongly
in
\(W^{-1,p}(Q;\R^d)\), and so
 \(\|v_k-u_k\|_{W^{-1,p}(Q;\R^d)}^{p}<\eps\) for \(k\) large
enough.
 
For every \(k\), let \(F_k:=F_{v_k,k}\). By Lemma~\ref{lem:pequiint},
we have
\begin{equation*}
\limsup_{k\to\infty}\big(F_k(v_k,Q)-F_k(u_k,Q)\big)\le 0,
\end{equation*}
while \eqref{eq:F(v) ge F(u)+e} gives
\(F_k(v_k,Q)-F_k(u_k,Q)\ge \eps\)
for \(k\) large enough. This contradiction concludes the proof.
\end{proof}

\begin{definition}\label{FAcal}
For every \(F\in\Ical\) and every \(D\in\Ocal(\R^N)\), let \(F^{\Acal}(\cdot,D)\)
be the restriction of \(F(\cdot,D)\) to \(\ker\Acal_D\).
\end{definition}

\begin{remark}\label{rem:Gammaseq}
Let \((F_k)_{k\in\NN}\) be a sequence of functionals in \(\Ical\),
let \(F\in \Ical\), and let \(D\in\Ocal(\R^N)\). If we extend
\(F^{\Acal}(\cdot,D)\) to \(L^p(D;\RR^d)\) by setting \(F^{\Acal}(u,D)=+\infty\)
for every \(u\in L^p(D;\RR^d)\setminus \ker\Acal_D\), by 
 the lower bound in (a) of Definition \ref{def:Ffrak} 
%\eqref{eq:pgrowth F} 
we can apply \cite[Proposition 8.16]{DM93}
to conclude that the sequence \((F_k^{\Acal}(\cdot,D))_{k\in\NN}\)
\(\Gamma\)-converges to \(F^{\Acal}(\cdot,D)\) in \(\ker\Acal_D\)
with respect to
the weak topology of \(L^p(D;\RR^d)\) if and only if
\begin{itemize}
\item[(LI)]for every \(u\in \ker\Acal_D\) and every sequence
\((u_k)_{k\in\NN}\) in \(\ker\Acal_D\) converging to \(u\) weakly
in
\(L^p(D;\RR^d)\), we have
\(\displaystyle
F^{\Acal}(u,D)\le \liminf_{k\to\infty}F_k^{\Acal}(u_k,D)\),
\item[(LS)]for every \(u\in \ker\Acal_D\) there exists a sequence
\((v_k)_{k\in\NN}\) in \(\ker\Acal_D\) converging to \(u\) weakly
in
\(L^p(D;\RR^d)\) such that
\(\displaystyle
F^{\Acal}(v,D)\ge \limsup_{k\to\infty}F_k^{\Acal}(v_k,D)\).
\end{itemize}
\end{remark}

We now prove the equivalence between  \(\Gamma\)-convergence
with respect to
the topology induced by \(\Vert\cdot\Vert_D^\Acal\) and \(\Gamma\)-convergence
in the \(\Acal\)-free setting.
\begin{theorem}\label{main constrained}
Let \((F_k)_{k\in\NN}\) be a sequence of functionals in \(\Ical_{\lip}\),
let \(F\in \Ical\), and let \(F_k^{\Acal}\) and  \(F^{\Acal}\)
be as in Definition \ref{FAcal}. 
The the following conditions are
equivalent:
\begin{itemize}
\item[(a)]
For every \(D\in\Ocal(\R^N)\), the sequence \(F_k(\cdot,D)\) \(\Gamma\)-converges
to \(F(\cdot,D)\)  with respect to
the topology induced by \(\Vert\cdot\Vert_D^\Acal\)  on \(L^p(D;\RR^d)\);
\item[(b)]
for every \(D\in\Ocal(\R^N)\), the sequence \((F_k^{\Acal}(\cdot,D))_{k\in\NN}\)
\(\Gamma\)-converges to \(F^{\Acal}(\cdot,D)\) in \(\ker\Acal_D\)
with respect to
the weak topology of \(L^p(D;\RR^d)\).
\end{itemize}
\end{theorem}

\begin{proof}\
Assume (a). Let  \(u\in \ker\Acal_D\) and let \((u_k)_{k\in\NN}\)
be a sequence in \(\ker\Acal_D\) converging to \(u\) weakly in
\(L^p(D;\RR^d)\). Since the embedding of \(L^p(D;\RR^d)\) into
\(W^{-1,p}(D;\RR^d)\) is compact and \(\Acal u_k=\Acal u=0\),
the  sequence \((u_k)_{k\in\NN}\) converges to  \(u\) strongly in
the topology induced by the norm \(\Vert\cdot\Vert_D^\Acal\)
 on \(L^p(D;\RR^d)\). By (a), this implies that
 \begin{equation*}
 F^{\Acal}(u,D)\le \liminf_{k\to\infty}F_k^{\Acal}(u_k,D).
 \end{equation*}
Hence, condition (LI) of Remark \ref{rem:Gammaseq} is satisfied.

To prove (LS), we fix \(u\in \ker\Acal_D\) and  a sequence \((u_k)_{k\in\NN}\subset
L^p (D;\R^d)\)  such that
\begin{equation*}%\label{eq:recseq22}
\lim_{k\to\infty} \Vert u_k -u\Vert^{\Acal}_D=0 \quad\text{and}\quad
\limsup_{k\to\infty}
F_k(u_k, D) = F(u, D)=F^{\Acal}(u, D)<+\infty,
\end{equation*}
which exists by \cite[Proposition~8.1]{DM93}. We observe further
that \(u_k \weakly u\) weakly in \(L^p(D;\RR^d)\) (see Remark~\ref{weak Lp versus W-1p}). Let \((v_k)_{k\in\NN}\)
be the sequence
 in \(\ker\Acal_D\) provided by Lemma \ref{lem:equiint}.
 By Lemma \ref{lem:pequiint}, we have
 \begin{equation*}
\limsup_{k\to\infty} F^{\Acal}_k(v_k, D)= \limsup_{k\to\infty}
F_k(v_k, D)\le   \limsup_{k\to\infty}
F_k(u_k, D) = F^{\Acal}(u, D),
\end{equation*}
which proves condition (LS) of Remark \ref{rem:Gammaseq}. This
concludes the proof of (b).

Assume (b). By   Theorem \ref{thm:Gamma-unconst},
there exist a subsequence, which we do not relabel,
and a functional  \(\hat F\in\Ical_{\lsc}\) such that for every
\(D\in\Ocal(\R^N)\), the sequence
\((F_k(\cdot,D))_{k\in\NN}\)
\(\Gamma\)-converges to \(\hat F(\cdot,D)\)  with respect to
the
topology induced by \(\Vert\cdot\Vert_D^\Acal\) on \(L^p(D;\RR^d)\).
Since (a) implies (b), we have \(\hat F(u,D)=F(u,D)\) for every
\(D\in \Ocal(\R^N)\) and every
\(u\in \ker\Acal_D\). Let  \(\hat f\)  and \(f\) be the integrands
corresponding to \(\hat F\)  and \(F\), respectively,
defined by \eqref{eq:def f(x,xi)}. Since every constant function
belongs to \(\ker\Acal_D\), we deduce that  \(\hat f=f\); 
 by Theorems
\ref{thm:unconstrained main} and \ref{th:Babadjian}
this implies \(\hat F=F\). Since the \(\Gamma\)-limit does not
depend on the subsequence, we obtain (a) 
by the Urysohn property of \(\Gamma\)-convergence (see \cite[Proposition
8.3]{DM93}).
\end{proof}

From Theorems \ref{thm:Gamma-unconst} and  \ref{main constrained},
we deduce the following compactness result in the
\(\Acal\)-free setting.
\begin{corollary}\label{cor:Afree compactness}
Let \((F_k)_{k\in\NN}\) be a sequence of functionals in \(\Ical_{\lip}\)
and let \(F_k^{\Acal}\) be as in Definition \ref{FAcal}. 
Then, there exist a subsequence, which we do not relabel,
and a functional  \(F\in\Ical_{\lsc}\) such that
for every \(D\in\Ocal(\R^N)\), the sequence \((F_k^{\Acal}(\cdot,D))_{k\in\NN}\)
\(\Gamma\)-converges to \(F^{\Acal}(\cdot,D)\) in \(\ker\Acal_D\)
with respect to
the weak topology of \(L^p(D;\RR^d)\).
\end{corollary}

\section{The integrand obtained from minimum values on small
cubes}

In this section, given an integrand \(f\in\Fcal_{\qc}\), we reconstruct
its values \(f(x,\xi)\) at a pair \((x,\xi)\in \R^N\times\R^d\)
by taking the limit, as \(\rho\to0^+\), of the infima of some
minumum problems related to \(f\) and \(\xi\) in the cubes \(Q_\rho(x)\).
\begin{definition}\label{def:U0} For every cube \(Q\subset\RR^N\),
we consider the sets
\begin{equation*}%\label{ccharacterization}
\begin{aligned}
&\Ucal(Q):=\bigg\{u\in L^{p}_{\loc}(\RR^N;\RR^{d})\colon u \text{
is
\(Q\)-periodic} ,\int_{Q}u\,dx=0,\text{and }\Acal u=0 \text{ in
$\RR^N$ in the sense of \eqref{eq:Au on RN}}\bigg\},\\
& \Ucal_c(Q):=\bigg\{u\in L^{p}(Q;\RR^{d})\colon \supp u \subset\subset
Q  ,\int_{Q}u\,dx=0,\text{and }\Acal_Q u=0 \text{ in \(Q\) in the
sense of \eqref{eq:Au}}\bigg\}.
\end{aligned}
\end{equation*}
For every \(f\in\Fcal\)  and \(\xi\in
\RR^d\), we set
\begin{equation}
\label{eq:Mcset}
\begin{aligned}
&M(f,\xi,Q):= \inf \big\{ F(\xi + u,Q): u\in \Ucal(Q)\big\},\\
&M_c(f,\xi,Q):= \inf \big\{ F(\xi + u,Q): u\in \Ucal_c(Q)\big\},
\end{aligned}
\end{equation}
where $F$ is defined by \eqref{eq:Ffree}.
\end{definition}

\begin{remark}\label{rmk:ext}
If \(u\in \Ucal_c(Q)\), then its \(Q\)-periodic extension belongs
to \(\Ucal(Q)\).
In fact, denoting by \(\tilde u\)  the  \(Q\)-periodic extension
of \(u\), let us prove that \(\Acal\tilde  u=0\) in
$\RR^N$.  Given \(\psi \in C^\infty_c(\RR^N;\RR^l)\), we can
find a finite family of mutually disjoint translations of \(Q\),
\(\{Q_j\}_{j=1}^M\), such that
\begin{equation*}
\begin{aligned}
\supp \psi \subset\subset {\rm int }\bigg(\bigcup_{j=1}^M \overline{Q_j}\bigg).
\end{aligned}
\end{equation*}
Moreover, setting \(K:=\supp u \subset\subset Q\), let \(K_j\subset\subset
Q_j\) be the corresponding translations of \(K\). In other words,
\(K_j=\supp \tilde u_{|Q_j}\). Finally, let \(\theta_j\in C^\infty_c(Q_j;[0,1])\)
be a cut-off function with \(\theta_j=1\) on \(K_j\). Then,
\begin{equation*}
\begin{aligned}
\langle\Acal \tilde u,\psi\rangle&=-\sum_{i=1}^{N}\sum_{j=1}^M\int_{Q_j}(A^i\tilde
u)\cdot
\partial_i\psi\,
dx = -\sum_{i=1}^{N}\sum_{j=1}^M\int_{Q_j}(A^i\tilde u)\cdot
\partial_i(\theta_j\psi)\,
dx \\&= -\sum_{i=1}^{N}\sum_{j=1}^M\int_{Q}(A^iu)\cdot
\partial_i \tilde \psi_j\,
dx=0,
\end{aligned}
\end{equation*}
where \(\tilde \psi_j\in C^\infty_c(Q;\RR^l)\) is the translation
of \((\theta_j\psi)_{|Q_j}\) from \(Q_j\) to \(Q\), and where we used
the fact that \(\Acal_Q u=0\) in the last equality.

Hence, \(M(f,\xi,Q) \leq M_c(f,\xi,Q)\) for every cube \(Q\subset\R^N\),
every \(f\in\Fcal\), and every \(\xi\in \R^d\).
\end{remark}

We now prove that the value \(f(x,\xi)\) of an integrand \(f\in
\Fcal_{\qc}\) can be reconstructed using the minimum values
introduced in
\eqref{eq:Mcset} on cubes shrinking to \(x\). 

\begin{theorem}\label{thm:formula}
Let  \(f\in\Fcal_{\lip}\cap\Fcal_{\qc}\).
Then, for a.e.~ \(x\in\RR^N\) and every \(\xi\in\RR^d\),  we
have that
\begin{equation*}
%\label{eq:}
\begin{aligned}
f(x,\xi) = \lim_{\rho\to 0^+} \frac{M(f,\xi, Q_{\rho}(x))}{\rho^N}=\lim_{\rho\to0^+}
\frac{M_c(f,\xi, Q_{\rho}(x))}{\rho^N}.
\end{aligned}
\end{equation*}
\end{theorem}

\begin{proof}
For \(\xi\in\RR^d\) fixed, we  have for every Lebesgue point
for \(f(\cdot,\xi)\)
 that
\begin{equation}\label{eq:lebpoint}
\begin{aligned}
f(x,\xi) = \lim_{\rho\to0^+} \frac{1}{\rho^N} \int_{Q_\rho(x)}
f(y,\xi)\,dy.
\end{aligned}
\end{equation}
Thus, for a.e.~\(x\in\RR^N\) and for all \(\xi\in\Q^d\), \eqref{eq:lebpoint}
holds. Consequently, using the Lipschitz continuity in \eqref{eq:pLip},
we conclude that \eqref{eq:lebpoint}
holds  for a.e.~\(x\in\RR^N\) and for all \(\xi\in\R^d\).
Then,    for a.e.~\(x\in\RR^N\) and for all \(\xi\in\R^d\), we
have that
\begin{equation*}
\begin{aligned}
f(x,\xi) = \lim_{\rho\to0^+} \frac{1}{\rho^N} \int_{Q_\rho(x)}
f(y,\xi)\,dy \geq \limsup_{\rho\to0^+} \frac{M_c(f,\xi, Q_\rho(x))}{\rho^N}.
\end{aligned}
\end{equation*}

By Remark~\ref{rmk:ext}, it remains to prove that
\begin{equation}
\label{eq:leqineq}
\begin{aligned}
f(x,\xi) \leq \liminf_{\rho\to0^+} \frac{M(f,\xi, Q_{\rho}(x))}{\rho^N}
\end{aligned}
\end{equation}
holds   for a.e.~\(x\in\RR^N\) and for all \(\xi\in\R^d\).

Fix \(x\in\R^N\) and \(\xi\in\RR^d\). Given \(\rho>0\),   \eqref{eq:Mcset}
yields
a function  \(u_\rho \in \Ucal(Q_{\rho}(x))\) satisfying
\begin{equation*}
\begin{aligned}
\int_{Q_{\rho}(x)} f(y, \xi + u_\rho(y))\, dy < M(f,\xi, Q_{\rho}(x)
) + \rho^{N+1}.
\end{aligned}
\end{equation*}
Then, to prove \eqref{eq:leqineq},  it suffices to show that
\begin{equation}\label{eq:suffcondr}
\begin{aligned}
f(x,\xi) \leq \liminf_{\rho\to0^+} \frac{1}{\rho^N}\int_{Q_{\rho}(x)}
f(y, \xi + u_\rho(y))\, dy.  
\end{aligned}
\end{equation}

To prove this inequality, we have to compare the values of \(f\)
at two different points, \(x\), \(y\in\R^N\). For this reason,
for \(m\in\NN\) and  
 for \(x,y \in\RR^N\), we define%
\begin{equation}\label{eq:def omegam}
\begin{aligned}
\omega_m(x,y):= \sup_{|\eta|\leq m} |f(y,\eta) - f(x,\eta)|.
\end{aligned}
\end{equation}

We claim that there exists a set \(E\subset \R^N\) with measure
\(0\) such that
for every \(x\in\RR^N\setminus E\) and every \(m\in\NN\), we have
\begin{equation}
\label{eq:omegatoz}
\begin{aligned}
\lim_{\rho\to 0^+} \frac{1}{\rho^N} \int_{Q_\rho(x)
} \omega_m(x,y)\, dy =0.
\end{aligned}
\end{equation}
To prove this claim, we fix \(m\in\NN\) and \(k\in \NN\) and
 find \(n_{m,k}\in\NN\)
and a finite family \((\eta_i^{m,k})_{i=1}^{n_{m,k}}\) in \(
B_m(0)\)
such that
\begin{equation*}
\begin{aligned}
\overline {B_m(0)} \subset \bigcup_{i=1}^{n_{m,k}} B_{\frac1k}(\eta_i^{m,k}).
\end{aligned}
\end{equation*}
Then, for \(\eta\in B_{\frac1k}(\eta_i^{m,k})\), we have
\begin{equation*}
\begin{aligned}
|f(y,\eta) - f(x,\eta) & \leq |f(y, \eta_i^{m,k}) - f (x , \eta_i^{m,k})|
+ |f(y, \eta) - f(y, \eta_i^{m,k})| + | f(x, \eta) - f(x, \eta_i^{m,k})|\\
& \leq  |f(y, \eta_i^{m,k}) - f (x , \eta_i^{m,k})| + C(1+
m^{p-1})\tfrac1k,
\end{aligned}
\end{equation*}
where we used \eqref{eq:pgrowth} and \eqref{eq:pLip}, and \(C\) depends
on \(c_0\) and \(c_1\). Consequently,
\begin{equation*}
\begin{aligned}
\omega_m(x,y) \leq \sup_{1\leq i\leq n_{m,k}}   |f(y, \eta_i^{m,k})
- f (x , \eta_i^{m,k})| + C(1+
m^{p-1})\tfrac1k, 
\end{aligned}
\end{equation*}
which gives
\begin{equation*}
\begin{aligned}
 \frac{1}{\rho^N} \int_{Q_\rho(x)
} \omega_m(x,y)\, dy \leq \sum_{i=1}^{n_{m,k}}  \frac{1}{\rho^N}
\int_{Q_\rho(x)
} |f(y, \eta_i^{m,k})
- f (x , \eta_i^{m,k})|\, dy + C(1+
m^{p-1})\frac1k.
\end{aligned}
\end{equation*}
Hence, considering the Lebesgue points \(x\) for all functions
\(f(\cdot, \eta_i^{m,k})\) with \(m\in\NN\), \(k\in\NN\), and
\(i\in\{1,...,n_{m,k}\}\), we find a set 
\(E\subset\R^N\) with measure \(0\) such that for every \(x\in\R^N\setminus
E\),  \(m\in\NN\), and \(k\in\NN\) we have
\begin{equation*}
\begin{aligned}
\lim_{\rho\to 0^+} \frac{1}{\rho^N} \int_{Q_\rho(x)
} \omega_m(x,y)\, dy \leq C(1+
m^{p-1})\frac1k,
\end{aligned}
\end{equation*}
from which \eqref{eq:omegatoz} follows by taking the limit \(k\to\infty\).
Since \(f\in \Fcal_{\qc}\) 
%we may assume, changing possibly \(E\),
%that 
for every \(x\in \R^N\setminus E\) the function \(\xi\mapsto
f(x,\xi)\) is \(\Acal\)-quasiconvex.

We now fix \(x\in \R^N\setminus E\), \(\xi \in \R^d\), and a
sequence \(\rho_j\to 0^+\) such that
\begin{equation}\label{lim j = liminf rho}
\lim_{j\to\infty} \frac{1}{\rho_j^N}\int_{Q_{\rho_j}(x)}f(y,
\xi + u_{\rho_j}(y))\, dy = \liminf_{\rho\to0^+} \frac{1}{\rho^N}\int_{Q_{\rho}(x)}f(y,
\xi + u_\rho(y))\, dy.
\end{equation}
By \eqref{eq:pgrowth},
\begin{equation*}
\begin{aligned}
\int_{Q_{\rho_j}(x)} \Big(\frac1c_0 |\xi + u_{\rho_j}(y)|^p -
c_0 \Big)\,
dy &\leq
\int_{Q_{\rho_j}(x)} f(y, \xi + u_{\rho_j}(y))\, dy < M(f,\xi,
{Q_{\rho_j}(x)})
+ \rho_j^{N+1}\\
&\leq c_0 (1+|\xi|^p)\rho_j^N + \rho_j^{N+1},
\end{aligned}
\end{equation*}
which gives
\begin{equation}\label{eq:unifrhoc}
\begin{aligned}
\frac{1}{\rho_j^N}\int_{Q_{\rho_j}(x)} |\xi + u_{\rho_j}(y)|^p\,
dy \leq
C_\xi
\end{aligned}
\end{equation}
for some positive constant \(C_\xi\), independent of  \(j
\). Then, setting
\begin{equation*}
\begin{aligned}
v_j(z) := u_{\rho_j}(x+\rho_j z) \quad \text{ for } z\in Q,
\end{aligned}
\end{equation*}
with \(Q\) the unit cube centered at the origin, 
 we have that \(v_j \in \Ucal(Q)\) with \(\int_{Q} | \xi+
v_j(z)|^p\, dz \leq C_\xi\) by \eqref{eq:unifrhoc}. Then, we
can use
\cite[Lemma~3.1]{BrFoLe00} to find a \(p\)-equi-integrable
sequence \((\tilde v_j)_{j\in\NN}\subset \Ucal(Q)\)
 such
that
\begin{equation*}
\begin{aligned}
\limsup_{j\to\infty} \int_Q f(x+\rho_j z, \xi + \tilde v_j(z))\,
dz \leq \limsup_{j\to\infty} \int_Q f(x+\rho_j z, \xi +  v_j(z))\,
dz.
\end{aligned}
\end{equation*}

Next, we set \(\tilde u_j(y):= \tilde v_j(\frac{y-x}{\rho_j})\),
and
we observe that the preceding inequality becomes 
\begin{equation*}%\label{eq:byequi}
\limsup_{j\to\infty} \frac{1}{\rho_j^N} \int_{Q_{\rho_j}(x)}
f(y, \xi + \tilde u_j(y)) dy 
\leq \lim_{j\to\infty} \frac{1}{\rho_j^N} \int_{Q_{\rho_j}(x)}
f(y, \xi +  u_{\rho_j}(y)) dy.
\end{equation*}
This inequality, together with \eqref{lim j = liminf rho}, implies
that
\eqref{eq:suffcondr} is a consequence of the  inequality
\begin{equation}\label{eq:suffcondrrr}
\begin{aligned}
f(x,\xi) \leq \liminf_{j\to\infty} \frac{1}{\rho_j^N}\int_{Q_{\rho_j}(x)}
f(y, \xi
+ \tilde u_j(y))\, dy,  
\end{aligned}
\end{equation}
which we establish next.
Observing that  \((\tilde u_j)_{j\in\NN}\subset \Ucal(Q_{\rho_j}(x))
\), we have
 by the definition of \(\Acal\)-quasiconvexity
(Definition \ref{def:Aqcx}) 
that%
\begin{equation}
\label{eq:byAqcxty}
\begin{aligned}
f(x,\xi) \leq \frac{1}{\rho_j^N} \int_{Q_{\rho_j}(x) } f(x, \xi
+
\tilde u_j(y))\,
dy \quad \text{ for all  } j\in\NN. 
\end{aligned}
\end{equation}
To prove \eqref{eq:suffcondrrr}, 
 we compare  the  integrals
\begin{equation*}
\begin{aligned}
\frac{1}{\rho_j^N} \int_{Q_{\rho_j}(x) } f(x, \xi + \tilde u_j(y))\,
dy \quad \text{ and } \quad \frac{1}{\rho_j^N} \int_{Q_{\rho_j}(x)
} f(y, \xi + \tilde u_j(y))\,
dy.
\end{aligned}
\end{equation*}

To this aim, for \(m\in \NN\), we set
\begin{equation*}
\begin{aligned}
\hat Q_j^m(x):= \big\{ y \in Q_{\rho_j}(x): |\xi+\tilde u_j(y)|\leq
m\big \} \quad
\text{ and } \quad \check Q_j^m(x):= \big\{ y \in Q_{\rho_j}(x):
|\xi+\tilde
u_j(y)|>
m\big \} ,
\end{aligned}
\end{equation*}
and observe that \eqref{eq:byAqcxty} and
  \eqref{eq:def omegam}  yield
\begin{equation*}
\begin{aligned}
f(x,\xi)&\leq  \frac{1}{\rho_j^N}\int_{\hat Q_j^m(x)}
f(x, \xi
+ \tilde u_j(y))\, dy + \frac{1}{\rho_j^N}\int_{\check Q_j^m(x)}
f(x, \xi
+ \tilde u_j(y))\, dy \\
&\leq \frac{1}{\rho_j^N}\int_{Q_{\rho_j}(x)}
f(y, \xi
+ \tilde u_j(y))\, dy +\frac{1}{\rho_j^N}\int_{Q_{\rho_j}(x)}
\omega_m(x,y) \, dy  + \frac{1}{\rho_j^N}\int_{\check Q_j^m(x)}
f(x, \xi
+ \tilde u_j(y))\, dy.
\end{aligned}
\end{equation*}
Passing to the limit as \(j\to\infty\), 
we obtain
from \eqref{eq:omegatoz} that\begin{equation*}
f(x,\xi)\leq \liminf_{j\to\infty}  \frac{1}{\rho_j^N}\int_{Q_j(x)}
f(y, \xi
+ \tilde u_j(y))\, dy + \limsup_{j\to\infty} \frac{1}{\rho_j^N}\int_{
\check Q_j^m(x)}
f(x, \xi
+ \tilde u_j(y))\, dy.
\end{equation*}

To conclude the proof of \eqref{eq:suffcondrrr}, it is enough
to show that
\begin{equation*}
%\label{eq:limsupz}
\begin{aligned}
\limsup_{j\to\infty} \frac{1}{\rho_j^N}\int_{
\check Q_j^m(x)}
f(x, \xi
+ \tilde u_j(y))\, dy \leq  \lambda_m\quad \text{with } \lim_{m\to\infty}
\lambda_m
=0.
\end{aligned}
\end{equation*}
The preceding estimate is a consequence of
the  \(p\)-equi-integrability of the sequence \((\tilde v_j)_{j\in\NN}\)
in \(Q\) together with the inequality
\begin{equation*}
\begin{aligned}
\frac{1}{\rho_j^N}\int_{\check Q_j^m(x)}
f(x, \xi
+ \tilde u_j(y))\, dy \leq \frac{c_0}{\rho_j^N} \int_{\check
Q_j^m(x)}
(1+|\xi
+ \tilde u_j(y)|^p)\, dy = c_0\int_{\check Q_j^m} \big( 1+ |\xi
+ \tilde v_j(z)|^p\big)\, dz,
\end{aligned}
\end{equation*}
where \(\check Q_j^m:= \{ z\in Q: |\tilde v_j(z)|>m\} \), with \(|\check Q_j^m|\to0\) as \(m\to\infty\).    
\end{proof}

\section{The \(\Gamma\)-limit obtained from minimum values on
small
cubes}

In this section, we prove that the integrand of the \(\Gamma\)-limit
of a sequence \((F_k)_{k\in\NN}\) of functionals in \(\Ical\)
can be obtained by talking the limit, first as \(k\to\infty\)
and then as \(\rho\to 0^+\), of the infima of some minimum problems
for \(F_k\) on the cubes \(Q_\rho(x)\), see \eqref{eq:f from
M}.

We begin by proving that if  \(F_k(\cdot,D)\) \(\Gamma\)-converges
to  \(F(\cdot,D)\), then the corresponding infima introduced
in 
\eqref{eq:Mcset} satisfy some inequalities.

\begin{proposition}\label{M and Mc}
Let \((f_k)_{k\in\NN}\) be a sequence in \(\Fcal_{\lip}\), let \(f\in
\Fcal\), and let \(F_k\) and \(F\) be the corresponding functionals
in \(\Ical\) defined by
\eqref{eq:Ffree}.  Assume that for every \(D\in\Ocal(\RR^N)\),
the sequence  \((F_k(\cdot,D))_{k\in\NN}\)
\(\Gamma\)-converges to
 \(F(\cdot,D)\) with respect to the topology induced by \(\Vert\cdot\Vert_{D}^{\Acal}\)
on \(L^p(D;\RR^d)\).  Let \(Q\subset\RR^N\)
be an open cube and let \(\xi\in\RR^d\). 
Then, 
\begin{align}
&\displaystyle \limsup_{k\to\infty}M(f_{k},\xi,Q)\leq M_c(f,\xi,Q),
\label{limsup cube}
\\
&\displaystyle \liminf_{k\to\infty}M(f_{k},\xi,Q)\geq M(f,\xi,Q).
\label{liminf cube}
\end{align}
\end{proposition}

\begin{proof}
Let \(\delta>0\). By \eqref{eq:Mcset}, there exists \(u\in L^{p}(Q;\R^{d})\),
with  \(\supp u \subset\subset Q\), \(\int_{Q}u\,dx=0\), and \(\Acal_Q
u=0\), such that
\begin{equation*}%\label{almost Mc}
F(\xi+u,Q)<M_c(f,\xi,Q)+\delta.
\end{equation*}
By \(\Gamma\)-convergence, there exists a sequence \((u_k)_{k\in\NN}\) in
\(L^{p}(Q,\R^{d})\) such that  \(u_k\to u\) in \(W^{-1,p}(Q,\R^{d})\),
 \(\Acal_Q u_k\to \Acal_Q u=0\) in \(W^{-1,p}(Q,\R^{l})\),
and
\begin{equation}\label{lim from Gamma}
\lim_{k\to\infty}F_k(\xi+u_k,Q)=F(\xi+u,Q)<M_c(f,\xi,Q)+\delta<+\infty.
\end{equation}
By \eqref{eq:pgrowth}, this implies that  \((u_k)_{k\in\NN}\) is bounded in
\(L^{p}(Q,\R^{d})\); hence,  \(u_k\rightharpoonup u\) weakly
in \(L^{p}(Q,\R^{d})\). 
Using Lemmas~\ref{lem:equiint} and~\ref{lem:pequiint}, we can find a \(p\)-equi-integrable
sequence  \(( v_k)_{k\in\NN}\subset L^p_{per}(Q;\RR^d)  \) satisfying
\begin{equation*}%\label{eq:newseq1per1}
 v_k \weakly  u  \text{ in \(L^p(Q;\RR^d) \)}, \quad \Acal
v_k =0 \hbox{ in }\R^{N}, \quad \int_Q v_k\, dx = \int_Q
u\,dx=0,
\end{equation*}
and
\begin{equation}
\label{eq:newseq2per1}
\limsup_{k\to\infty} F_k(\xi+v_k,Q) \leq \limsup_{k\to\infty}
F_k(\xi+u_k,Q).
\end{equation}
 By \eqref{eq:Mcset}, we have \(M(f_k,\xi,Q)\le F_k(\xi+v_k,Q)\).
This inequality, together with 
 \eqref{lim from Gamma} and \eqref{eq:newseq2per1}, yields
\begin{equation*}
\limsup_{k\to\infty}M(f_{k},\xi,Q)\leq M_c(f,\xi,Q) + \delta.
\end{equation*}
Since $\delta>0$ is arbitrary, we obtain \eqref{limsup cube}.

To prove \eqref{liminf cube}, we choose \(u_k
\in L^{p}_{\loc}(\RR^N;\RR^{d})\) for every \(k\), with \(u_k\)
\(Q\)-periodic, \(\int_{Q}u_k \,dx=0\), and \(\Acal u_k=0\) in
\(\RR^N\), such that
\begin{equation}\label{Fkuk<M+1/k}
F_k(\xi+u_k,Q)<M(f_{k},\xi,Q) + \tfrac1k.
\end{equation}
By \eqref{limsup cube}, the right-hand side of the previous inequality
is bounded; 
hence, a subsequence of \((u_k)_{k\in\NN}\), 
not relabeled, converges to some function \(u\) weakly in \(L^{p}(Q;\RR^{d})\).

Since all functions \(u_k\) are \(Q\)-periodic, the function
\(u\) can be extended to a 
\(Q\)-periodic function, still denoted by \(u\). 
Because the embedding of \(L^{p}(Q;\RR^{d})\) into \(W^{-1,p}(Q;\RR^{d})\)
is compact 
and \(\Acal u_k=0\) for every \(k\), we deduce that \((u_k)_{k\in\NN}\) converges
to \(u\) in
the
topology induced by \(\Vert\cdot\Vert_D^\Acal\) on \(L^p(D;\RR^d)\)
and that \(\Acal u=0\) in \(\R^{N}\). 
Therefore, by \(\Gamma\)-convergence and by \eqref{eq:Mcset}, we
have
\begin{equation*}
M(f,\xi,Q)\le F(\xi+u,Q)\le \liminf_{k\to\infty}F_k(\xi+u_k,Q).
\end{equation*}
Together with \eqref{Fkuk<M+1/k}, the preceding estimate gives \eqref{liminf cube}.
\end{proof}

\begin{theorem}\label{thm:f from M}
Let \((f_{k})_{k\in\NN}\) be a sequence in \(\Fcal_{\lip}\), let \((F_k)_{k\in\NN}\)
 be the corresponding sequence of 
functionals in \(\Ical\) defined by
\eqref{eq:Ffree}, and let \((F_k^{\Acal})_{k\in\NN}\)  be the
sequence obtained as in Definition \ref{FAcal}.  Suppose that
\begin{itemize}
\item[(a)] there exists  \(f\colon\R^N\times\R^d\to[0,+\infty)\)
such that
%!TEX encoding = UTF-8 Unicode
\begin{equation}\label{eq:f from M}
f(x,\xi)=\limsup_{\rho\to 0^+} \, \liminf_{k\to\infty} \frac{M(f_k,\xi,Q_\rho(x))}{\rho^N}
=\limsup_{\rho\to 0^+} \, \limsup_{k\to\infty} \frac{M(f_k,\xi,Q_\rho(x))}{\rho^N}
\end{equation}
for a.e.\ \(x\in \R^N\) and every \(\xi\in \R^d\).
\end{itemize}
Then,  there exists \(\hat f\in\Fcal_{\lip}\cap\Fcal_{\qc}\) such that \(f(x,\xi)=\hat f(x,\xi)\)  for a.e.\ \(x\in \R^N\)
and every \(\xi\in \R^d\),  and the functionals  \(F\) and  \(F^{\Acal}\)
introduced in \eqref{eq:Ffree} and 
 Definition \ref{FAcal} satisfy the following properties:
\begin{itemize}
\item[(b)] for every \(D\in\Ocal(\RR^N)\), the sequence  \((F_k(\cdot,D))_{k\in
\NN}\)
\(\Gamma\)-converges to
 \(F(\cdot,D)\) with respect to the topology induced by \(\Vert\cdot\Vert_{D}^{\Acal}\)
on \(L^p(D;\RR^d)\);
\item[(c)]
for every \(D\in\Ocal(\R^N)\), the sequence \((F_k^{\Acal}(\cdot,D))_{k\in\NN}\)
\(\Gamma\)-converges to \(F^{\Acal}(\cdot,D)\) in \(\ker\Acal_D\)
with respect to
the weak topology of \(L^p(D;\RR^d)\).
\end{itemize}

Conversely, if \(f\in\Fcal\) and the functionals  \(F\) and 
\(F^{\Acal}\) introduced in \eqref{eq:Ffree} and 
Definition \ref{FAcal} satisfy (b) or (c), then \(f\) satisfies
(a).
\end{theorem}

\begin{proof}  Since \(F_k\in\Fcal_{\lip}\) for every \(k\in\NN\) by Remark \ref{rem:Fcal gives Ffrac},  the equivalence between (b) and (c) is proved in
Theorem \ref{main constrained}.
If these conditions are satisfied, then by Corollary \ref{cor:UCmain}
 there exists \(\hat f\in\Fcal_{\lip}\cap \Fcal_{\qc}\) such that 
\begin{equation*}
F(u,D):=\int_D \hat f(x, u(x))\, dx, \quad\text{for every }D\in\Ocal(\R^N)\text{ and }u\in L^p(D;\R^d),
\end{equation*}
which implies
\(f(x,\xi)=\hat f(x,\xi)\)  for a.e.\ \(x\in \R^N\)
and every \(\xi\in \R^d\). 

Next, we assume that (b) holds, and we show that
\(\hat f\) satisfies \eqref{eq:f from M}. Fix \(x\in\R^N\) and \(\xi\in \R^d\). By Proposition
\ref{M and Mc}   for every \(\rho>0\)  we have  that
\begin{align*}
&\displaystyle \limsup_{k\to\infty}M(f_{k},\xi,Q_\rho(x))\leq
M_c(f,\xi,Q_\rho(x)),
\\
&\displaystyle M(f,\xi,Q_\rho(x))\le \liminf_{k\to\infty}M(f_{k},\xi,Q_\rho(x)).
\end{align*}
By Theorem \ref{thm:formula}, we obtain \eqref{eq:f from M}  with \(\hat f\)  for
a.e.\ \(x\in \R^N\) and every \(\xi\in \R^d\), 
concluding the proof of (a).

Finally, we assume that (a) holds, and we prove that  (b) is also satisfied. By Corollary \ref{cor:UCmain},  there exists a subsequence
\((f_{k_j})_{j\in\NN}\) and a function \(\hat f\in\Fcal_{\lip}\cap\Fcal_{\qc}\)
 such that  for every \(D\in\Ocal(\RR^N)\) the sequence  \((F_{k_j}(\cdot,D))_{j\in\NN}\)
\(\Gamma\)-converges to
 \(\widehat F(\cdot,D)\) with respect to the topology induced
by \(\Vert\cdot\Vert_{D}^{\Acal}\) on \(L^p(D;\RR^d)\),
 where \(\widehat F\) is the functional associated with \(\hat
f\) as in \eqref{eq:Ffree}.
Since (b)\(\Rightarrow\)(a) by the preceding part of the proof,
we have  for a.e.\ \(x\in \R^N\)
and every \(\xi\in \R^d\) that 
\begin{equation*}
\hat f(x,\xi)=\limsup_{\rho\to 0^+} \, \liminf_{j\to\infty} \frac{M(f_{k_j},\xi,Q_\rho(x))}{\rho^N}
=\limsup_{\rho\to 0^+} \, \limsup_{j\to\infty} \frac{M(f_{k_j},\xi,Q_\rho(x))}{\rho^N}.
\end{equation*}
By \eqref{eq:f from M}, this implies that \(\hat f(x,\xi)= f(x,\xi)\)
for a.e.\ \(x\in \R^N\) and every \(\xi\in \R^d\); hence, 
 \(\widehat F=F\).
Since the  \(\Gamma\)-limit does not depend on the subsequence,
we obtain (b) by the Urysohn property of \(\Gamma\)-convergence (see \cite[Proposition
8.3]{DM93}).
\end{proof}

The result of the previous theorem cannot be applied directly
to the study of stochastic homogenization by means of the subadditive
ergodic theorem \cite[Theorem~2.7]{AkKr81} (see also \cite{DaMo86} and \cite{LiMi02})   because the term \(M(f,\xi,Q)\) is not subadditive;
that is, we do not know if 
\begin{equation}\label{subadd cubes}
M(f,\xi,Q)\le \sum_{i\in I}M(f,\xi,Q_i)
\end{equation}
when \((Q_i)_{i\in I}\) is a finite decomposition of \(Q\) into
disjoint cubes. We now introduce a variant of \(M(f,\xi,Q)\)
that satisfies this property. The idea is to relax the constraint
\(\Acal u=0\) that was used in the definition of \(M(f,\xi,Q)\).
We begin with a technical result that will be useful to impose
a constraint on the norm \(\|\Acal_D u\|_{W^{-1,p}(D;\R^l)}\)
depending additively on \(D\).

\begin{remark}\label{rem:Ueta}
For every \(D\in\Ocal(\R^{N})\) and  \(u\in L^{p}(D;\RR^d)\),
there exists 
 \( V \in L^{p}(D;\RR^{l\times N})\) such that
 \begin{equation}\label{Au divU}
 \langle \widetilde \Acal_{D} u,\psi \rangle = 
 \int_{D}  V\cdot \nabla \psi\,dx
\quad\text{for every }\psi\in W^{1,q}(D;\RR^{l}),
 \end{equation}
 where \(\cdot\) denotes the Euclidean scalar product between
matrices.  By \eqref{eq:widetildeAu}, this equality is satisfied,
for instance, when for every \(x\in D\) and \(i=1,\dots,N\), the
\(i\)-th column of the matrix 
\( V(x)\) is given by the vector \(A^{i}u(x)\). If \eqref{Au divU}
holds, then
\begin{equation}\label{easy tilde W estimate}
 \Vert \widetilde\Acal_{D} u\Vert_{\widetilde W^{-1,p}(D;\RR^{l})}
  \leq \Vert  V\Vert_{L^p(D;\RR^{l\times N})}.
\end{equation}
\end{remark}

The following lemma provides a partial converse to inequality
\eqref{easy tilde W estimate}, which will be used later in the proof of Proposition \ref{M and Mceta}.
\begin{lemma}\label{lem:existence U}
Let  \(D\subset\R^{N}\) be a bounded connected open set with Lipschitz boundary,  and let \(K\subset D\) be a compact set. Then, there exists a constant \(C_{K,D}>0\) 
such that for every \(U \in L^{p}(D;\RR^{l\times N})\), with \(\supp U\subset K\), there exists  \(V \in L^{p}(D;\RR^{l\times N})\) satisfying
\begin{align}
&\int_{D}V\cdot\nabla\psi\,dx=\int_{D}U\cdot\nabla\psi\,dx \quad\hbox{for every }\psi\in W^{1,q}(D,\R^{l}),
\label{U and V}\\
&\Vert V\Vert_{L^p(D;\RR^{l\times N})}\leq C_{K,D} \Vert  \diverg U \Vert_{W^{-1,p}(D;\RR^{l})}.
\label{estimate U in terms of V}
\end{align}
\end{lemma}

\begin{proof} 
        Fix  \(D\subset\R^{N}\), \(K\subset D\), and \(U \in L^{p}(D;\RR^{l\times N})\) as in the statement. 
        Let \(W^{1,q}_{m}(D;\R^{l}):=\{v\in W^{1,q}(D;\R^{l})\colon  \int_{D}v\,dx=0\}\), with the norm induced by \(W^{1,q}(D;\R^{l})\), let
\(\widetilde W^{1,p}_{m}(D;\R^{l})\) be the dual space of  \(W^{1,q}_{m}(D;\R^{l})\), endowed with the dual norm, let \(T\colon 
W^{1,q}_{m}(D;\R^{l})\to \widetilde W^{1,p}_{m}(D;\R^{l})\) be the monotone operator defined by
\[
\langle T(v),\psi\rangle:=\int_{D}|\nabla v|^{q-2}\nabla v\cdot\nabla \psi\,dx\quad\hbox{for every }v,\psi\in W^{1,q}_{m}(D;\R^{l}),
\]
and let \(\Gcal\in \widetilde W^{1,p}_{m}(D;\R^{l})\) be defined by
\[
\langle \Gcal,\psi\rangle:= \int_{D}U\cdot\nabla\psi\,dx \quad\hbox{for every }\psi\in W^{1,q}_{m}(D;\R^{l}).
\]

By the Hartman--Stampacchia Theorem \cite[Lemma~3.1]{HarSta66}, there exists a unique function \(v\in W^{1,q}_{m}(D;\R^{l})\) such that
\(T(v)=\Gcal\); i.e., 
\begin{equation}\label{q-Laplacian}
\int_{D}|\nabla v|^{q-2}\nabla v\cdot\nabla \psi\,dx=\int_{D}U\cdot\nabla\psi\,dx \quad\hbox{for every }\psi\in W^{1,q}(D;\R^{l}).
\end{equation}
Taking \(\psi=v\) in \eqref{q-Laplacian}, we obtain
\begin{equation}\label{estimate nabla u q}
\int_{D}|\nabla v|^{q}\,dx=\int_{D}U\cdot\nabla v\,dx.
\end{equation}

Let \(\omega\in C^{\infty}_{c}(D)\) with \(\omega=1\) on \(K\). Since  \(\supp U\subset K\),  we get from \eqref{estimate nabla u q} that
\begin{equation}\label{second estimate nabla u q}
\int_{D}|\nabla v|^{q}\,dx=\int_{D}U\cdot\nabla (\omega v)\,dx=-\langle\diverg U, \omega v\rangle
\leq  \Vert  \diverg U \Vert_{W^{-1,p}(D;\RR^{l})}  \Vert \omega v \Vert_{W^{1,q}_{0}(D;\RR^{l})} .
\end{equation}
Recalling the definition of \(\Vert \cdot \Vert_{W^{1,q}_{0}(D;\RR^{l})}\) given at the beginning of Section \ref{sect:notation},  the
Poincar\'e--Wirtinger Inequality yields a constant \(C_{D,\omega}>0\) such that
\[
 \Vert \omega v \Vert_{W^{1,q}_{0}(D;\RR^{l})}\leq  \Vert \omega \Vert_{L^{\infty}(D)}  \Vert \nabla v \Vert_{L^{q}(D;\RR^{l\times N})}
 +\Vert \nabla \omega \Vert_{L^{\infty}(D;\R^{N})}  \Vert v \Vert_{L^{q}(D;\RR^{l})} 
 \leq
 C_{D,\omega} \Vert \nabla v \Vert_{L^{q}(D;\RR^{l\times N})}.
\]
From this inequality and from \eqref{second estimate nabla u q}, we get
\[
 \Vert \nabla v \Vert^{q}_{L^{q}(D;\RR^{l\times N})} = \int_{D}|\nabla v|^{q}dx\leq C_{D,\omega} \Vert  \diverg U \Vert_{W^{-1,p}(D;\RR^{l})}  \Vert \nabla v \Vert_{L^{q}(D;\RR^{l\times N})}.
\]
Hence,
\begin{equation}\label{third estimate nabla u q}
 \Vert \nabla v \Vert^{q-1}_{L^{q}(D;\RR^{l\times N})}\leq C_{D,\omega} \Vert  \diverg U \Vert_{W^{-1,p}(D;\RR^{l})}.
\end{equation}

Let \(V:=|\nabla v|^{q-2}\nabla v\). Equality \eqref{U and V} follows from \eqref{q-Laplacian}. Since  \(|V|^{p}=|\nabla v|^{p(q-1)}=|\nabla v|^{q}\), we have
\(\Vert V \Vert_{L^p(D;\RR^{l\times N})}=\big(\int_{D}|\nabla v|^{q}dx\big)^{1/p}=
 \Vert \nabla v \Vert^{q/p}_{L^{q}(D;\RR^{l\times N})}= \Vert \nabla v \Vert^{q-1}_{L^{q}(D;\RR^{l\times N})}\).
 By \eqref{third estimate nabla u q}, this gives \eqref{estimate U in terms of V}.
\end{proof}

\begin{definition}\label{def:Ueta}   Given \(D\in\Ocal(\R^{N})\)  and  \(\eta>0\), we set 
\begin{align*}
&\Vcal^\eta(D):=\big\{  V \in
L^{p}( D;\RR^{l\times
N})\colon  \Vert  V\Vert^{p}_{L^p(D;\RR^{l\times
N})} <\eta|D| \big\},
\\
&
\Ucal^\eta_c( D ):=\bigg\{ u\in L^{p}( D ;\RR^d)\colon
\supp u
\subset \subset  D ,
\int_{ D } u \, dx = 0, \text{ and \eqref{Au divU} holds
for some
} V
\in \Vcal^\eta( D )\bigg\}.
\end{align*}
For every \(f\in \Fcal\) and every \(\xi\in \R^d\), we set
\begin{equation}\label{eq:Metac}
M^\eta_c(f,\xi, D ):=\inf \big\{ F(\xi + u, D )\colon u\in \Ucal^\eta_c( D )\big\},
\end{equation}
where $F$ is defined by \eqref{eq:Ffree}.
\end{definition}

We will see in Lemma \ref{lemma:stofg} that \(M^\eta_c\) satisfies
the subadditivity property mentioned in \eqref{subadd cubes}.

\begin{remark}\label{rmk:onUetaper}
If   \(R\subset\R^N\) is an open rectangle,  \(u\in\Ucal^\eta_c(R)
\), and \(  U \in\Vcal^\eta(R)
\) satisfies \eqref{Au divU}, we can extend   \(u   \) and
\(U\)  by \(R\)-periodicity, which extensions we do not relabel.
Then, 
using the fact that \(\supp u
\subset \subset R\), we
obtain that \(\Acal u = -\diverg U \) in \(\RR^N\)
in the sense of distributions. \end{remark}

\begin{remark}\label{rmk:pLip} For every   \(D\in\Ocal(\R^{N})\)
and every \(\eta>0\), we observe that \(0\in \Ucal^\eta_c(
D )\).
Therefore, we have \(M^\eta_c(f,\xi, D )\le F(\xi,
D )\) for every
\(\xi\in\R^d\). By \eqref{eq:pgrowth}, this implies
\begin{equation}\label{eq:pgrowth Mceta}
M^\eta_c(f,\xi, D )\le c_0(1+|\xi|^p)| D |.
\end{equation}
 If \(f\in \Fcal_{\lip}\),  it follows immediately from the definition  of \(M^\eta_c\) and
from \eqref{eq:pLip asymmetric} that for every \(\xi_1\) and \(\xi_2\in\R^d\),
we have
\begin{equation*}
\frac{M^\eta_c(f,\xi_1, D )}{| D |}\le \frac{M^\eta_c(f,\xi_2,
D )}{| D |}+
c_1\Big(1+ \Big( \frac{M^\eta_c(f,\xi_2, D )}{| D
|}\Big)^{\frac{p-1}p}
+|\xi_1-\xi_2|^{p-1}\Big)|\xi_2-\xi_1|.
\end{equation*}
Exchanging the roles of \(\xi_1\) and \(\xi_2\) and using \eqref{eq:pgrowth
Mceta}, we obtain
\begin{equation}\label{eq:pLip Mceta}
\Big|\frac{M^\eta_c(f,\xi_1, D )}{| D |} - \frac{M^\eta_c(f,\xi_2,
D )}{| D |}\Big|
\le c_5\big(1+ |\xi_1| +|\xi_2|\big)^{p-1}|\xi_2-\xi_1|
\end{equation}
for a suitable constant \(c_5\) depending only on \(c_1\) and
\(p\).
\end{remark}

The following result concerns the behavior of \(M^\eta_c(f_{k},\xi,Q)\)
on a cube \(Q\) when the functionals corresponding to \(f_k\)
\(\Gamma\)-converge.

\begin{proposition}\label{M and Mceta}
Let \((f_{k})_{k\in\NN}\) be a sequence in \(\Fcal_{\lip}\), let \(f\in
\Fcal\), and let \(F_k\) and \(F\) be the corresponding functionals
in \(\Ical\) defined by
\eqref{eq:Ffree}. Assume that for every \(D\in\Ocal(\RR^N)\),
the sequence  \((F_k(\cdot,D))_{k\in \NN}\)
\(\Gamma\)-converges to
 \(F(\cdot,D)\) with respect to the topology induced by \(\Vert\cdot\Vert_{D}^{\Acal}\)
on \(L^p(D;\RR^d)\). Then, for every  \(\eta>0\), every cube \(Q\subset\RR^N\),
and every \(\xi\in\RR^d\), we have
\begin{equation}
\limsup_{k\to\infty}M^\eta_c(f_{k},\xi,Q)\leq M_c(f,\xi,Q).
\label{limsup cube eta}
\end{equation}
Moreover, for every
\(\eps>0\), there exists \(\eta>0\) such that for every cube \(Q\subset\RR^N\)
 with side length  less than or equal to \( 1\) 
and  every \(\xi\in\RR^d\), we have 
\begin{equation}
M(f,\xi,Q)\le \liminf_{k\to\infty}M^\eta_c(f_{k},\xi,Q)+\eps|Q|.
\label{liminf cube eta}
\end{equation}
Consequently,
\begin{align}
&\displaystyle \sup_{\eta>0}\,\limsup_{k\to\infty}M^\eta_c(f_{k},\xi,Q)\leq
M_c(f,\xi,Q),
\label{sup eta limsup}
\\
&\displaystyle M(f,\xi,Q)\le \sup_{\eta>0}\, \liminf_{k\to\infty}M^\eta_c(f_{k},\xi,Q),
\label{sup eta liminf}
\end{align}
for every cube \(Q\subset\RR^N\)  with side length less than
or equal to \( 1\)  and every \(\xi\in\RR^d\).
\end{proposition}

\begin{proof}
Fix   \(\xi\in\RR^d\) and   \(\delta>0\). By \eqref{eq:Mcset},
there exists \(u\in L^{p}(Q;\R^{d})\), with  \(\supp u \subset\subset
Q\), \(\int_{Q}u\,dx=0\), and \(\Acal_Q u=0\), such that
\begin{equation*}%\label{almost Mc}
F(\xi+u,Q)<M_c(f,\xi,Q)+\delta.
\end{equation*}
By \(\Gamma\)-convergence, there exists a sequence \((u_k)_{k\in\NN}\)
in \(L^{p}(Q;\R^{d})\) such that 
 \(u_k\to u\) in \(W^{-1,p}(Q;\R^{d})\),
 \(\Acal_Q u_k\to \Acal_Q u\) in \(W^{-1,p}(Q;\R^{l})\), and
\begin{equation}\label{lim from Gamma2}
\lim_{k\to\infty}F_k(\xi+u_k,Q)=F(\xi+u,Q)<M_c(f,\xi,Q)+\delta<+\infty.
\end{equation}
By \eqref{eq:pgrowth}, this inequality implies that  \((u_k)_{k\in\NN}\)
is bounded in \(L^{p}(Q;\R^{d})\); hence,  \(u_k\rightharpoonup
u\) weakly in
\(L^{p}(Q;\R^{d})\).

Fix a compact set \(K\), with \(\supp u\subset K\subset\subset
Q\), such that
\(c_0(1+|\xi|^p)|Q\setminus K|<\delta\), 
and let
\(D_1\), \(D_2\in \Ocal(\R^N)\) be such that \(K\subset D_1\subset\subset
D_2\subset\subset Q\).
We apply
Lemma \ref{lemma: construction og wk} with \(v_k:=u\) for every
\(k\in\NN\) and \(B:=Q\setminus K\) to obtain a sequence
\(w_k\in L^p(Q;\R^d)\) such that
\begin{align}
& 
w_k=u=0\ \hbox{ in }Q\setminus D_{2},\quad
w_k\rightharpoonup u\quad\text{weakly in }L^{p}(Q;\R^{d}),
\quad  w_k\to u \text{ in }W^{-1,p}(Q;\R^{d}),
\label{lim wk}
\\
& \Acal_Q w_k\to \Acal_Q u=0\ \text{ in } W^{-1,p}(Q;\R^{l}),
\nonumber %\label{lim Awk}
\\
& \limsup_{k\to\infty}F_k(\xi+w_k,Q)\le \limsup_{k\to\infty}\big(F_k(\xi+u_k,D_2)+F_k(\xi,
Q\setminus K)\big), \label{lim Fkwk}
\end{align}
where we used in  \eqref{lim Fkwk} the fact that
\begin{equation*}
\begin{aligned}
\limsup_{k\to\infty}\big(F_k(\xi+u_k,D_2)+F_k(\xi+u,
Q\setminus K)\big)=\limsup_{k\to\infty}\big(F_k(\xi+u_k,D_2)+F_k(\xi,
Q\setminus K)\big)
\end{aligned}
\end{equation*}
because \(\supp u \subset K\).
By \eqref{eq:pgrowth} and by our choice of \(K\), we have \(F_k(\xi,
Q\setminus K)\le c_0(1+|\xi|^p)|Q\setminus K|<\delta\), and so
\eqref{lim Fkwk} gives
\begin{equation}\label{lim Fkwk2}
\limsup_{k\to\infty}F_k(\xi+w_k,Q)\le \lim_{k\to\infty} F_k(\xi+u_k,Q)+\delta.
\end{equation}
By \eqref{lim wk}, we have \(\int_Q w_k\,dx\to \int_{Q}u\,dx=0\).
Fix \(\varphi\in C^\infty_c(Q)\), with \(\int_Q \varphi\, dx=1\)  and \(\supp \varphi \subset\subset D_{2}\),
and set \(z_k:=w_k-\varphi\int_Qw_k \,dx\). 
 By the first formula in \eqref{lim wk}, we have
\begin{equation*}%\label{supp zk}
z_k=0\ \hbox{ in }Q\setminus D_{2}.
\end{equation*}
Moreover,
\(\Acal_Q z_k\to0\)
in \(W^{-1,p}(Q;\R^{l})\), and  we have by \eqref{eq:pLip} that
\begin{equation*}
 \limsup_{k\to\infty}F_k(\xi+z_k,Q)=  \limsup_{k\to\infty}F_k(\xi+w_k,Q).
\end{equation*}
This inequality, together with \eqref{lim from Gamma2} and \eqref{lim
Fkwk2}, gives
 \begin{equation}\label{ls Fk < McQ}
 \limsup_{k\to\infty}F_k(\xi+z_k,Q)<  M_c(f,\xi,Q)+2\delta.
\end{equation}
Since the supports of the functions \(z_k\) are contained in \(\overline D_{2}\),  recalling \eqref{eq:normDiv}, Remark~\ref{rem:Ueta} and Lemma~\ref{lem:existence U} yield a constant \(C=C_{\overline D_{2},Q}>0\) such that for every \(k\in\N\), there exists  \(V_k \in L^{p}(Q;\RR^{l\times N})\) satisfying
\begin{align}
&\int_{Q}V_k\cdot\nabla\psi\,dx=\langle\widetilde\Acal_Q z_k, \psi\rangle\quad\hbox{for every }\psi\in W^{1,q}(Q;\R^{l}),
\nonumber
\\
&\Vert V_k\Vert_{L^p(Q;\RR^{l\times N})}\leq C \Vert  \Acal_Q z_k \Vert_{W^{-1,p}(Q;\RR^{l})}. \label{estimate U_k in terms of Acal}
\end{align}

Fix \(\eta>0\). Since \(\Acal_Q z_k\to0\) in \(W^{-1,p}(Q;\R^{l})\), we obtain from 
\eqref{estimate U_k in terms of Acal}  that, for \(k\) large enough, the functions
\(V_k\) belong to the set \(\Vcal^{\eta}(Q)\) introduced in Definition~\ref{def:Ueta}.
Moreover, since \(\supp z_k\subset\subset Q\) and \(\int_Qz_k
\, dx=0\), we have for
\(k\) large enough that the functions
\(z_k\) belong to the set \(\Ucal^\eta_{c}(Q)\) introduced in Definition~\ref{def:Ueta}.
By \eqref{eq:Metac}, this implies that
\(M^\eta_c(f_k,\xi,Q)\le F_k(\xi+z_k,Q)\). Together with 
 \eqref{ls Fk < McQ}, this  yields
\begin{equation*}
 \limsup_{k\to\infty}M^\eta_c(f_k,\xi,Q) <  M_c(f,\xi,Q)+2\delta.
\end{equation*}
Given the arbitrariness of \(\delta>0\), we obtain \eqref{limsup cube
eta}, from  which  \eqref{sup eta limsup} follows.

To prove \eqref{liminf cube eta}, we fix \(\eps>0\) and 
set \(C:=2^{p-1}(c_{0}+1)|\xi|^{p}+2^{p-1}c_{0}(1+2c_{0})\). Let
 \(\eta>0\) be as in Corollary~\ref{cor:pequiint}. 
 Using the definition of \(M^\eta_c\) (see \eqref{eq:Metac})
and \eqref{easy tilde W estimate}, 
we choose \(u_k \in L^{p}(Q;\RR^{d})\) for every \(k\in \NN\)   ,
with \(\supp u_k\subset\subset Q\), \(\int_{Q}u_k \,dx=0\), and
\(\|\widetilde\Acal u_k\|_{\widetilde W^{-1,p}(Q;\R^l)}<\eta|Q|\),
such that
\begin{equation}\label{Fkuk<M+1/k 2}
 \frac{1}{c_{0}}\int_{Q}|\xi+u_k|^{p\,}dx - c_{0}|Q|\le 
F_k(\xi+u_k,Q)<M^\eta_c(f_{k},\xi,Q) + \tfrac1k|Q| \le \Big(c_{0}(1+|\xi|^{p})+1\Big)|Q|,
\end{equation}
 where the first and last inequality follow from
  (a) in Definition~\ref{def:Ffrak}.  %\eqref{eq:pgrowth F} 
 These inequalities imply that
 \(\|u_k\|_{L^{p}(Q;\R^d)}^{p}<C|Q|\) for every \(k\in\N\),
%
%By \eqref{limsup cube eta} the right-hand side of the previous
which allows us to extract a subsequence of \((u_k)_{k\in\NN}\), not relabeled, that converges to some function \(u\) weakly in \(L^{p}(Q;\RR^{d})\).

We extend each \(u_k\) to a \(Q\) periodic function, still denoted
\(u_k\). 
Then, for every \(D\in\Ocal(\R^N)\), the sequence \((u_k)_{k\in\NN}\) converges
weakly in \(L^p(D;\R^d)\) 
to the periodic extension of \(u\), still denoted by \(u\).
Since the embedding of \(L^{p}(Q;\RR^{d})\) into \(W^{-1,p}(Q;\RR^{d})\)
is compact,
\((u_k)_{k\in\NN}\) converges to \(u\) in \(W^{-1,p}(Q;\R^l)\); hence,
 \(\|u_k-u\|_{W^{-1,p}(Q;\RR^{d})}\le \eps |Q|\) for \(k\) large
enough.
 
Therefore, by Corollary \ref{cor:pequiint}, there exists
 \(v_k\in L^{p}_{\rm per}(Q;\R^d)\), with  \(\|v_k-u_k\|_{W^{-1,p}(Q;\R^d)}^{p}<\eps|Q|\),
\(\Acal v_k=0\) in \(\R^N\), and 
  \(\int_{Q}v_k \,dx=\int_{Q}u_k \,dx=0\),
 such that
\begin{equation}\label{Fk(v)<Fk(u)+e}
F_k(\xi+v_k,Q)<F_k(\xi+u_k,Q)+\eps|Q|.
\end{equation}
Since the right-hand side of \eqref{Fkuk<M+1/k 2} is bounded,
the previous inequality and \eqref{eq:pgrowth} imply 
that  \((v_k)_{k\in\NN}\) is bounded in \(L^{p}(Q;\RR^{d})\);
hence, a subsequence of \((v_k)_{k\in\NN}\), 
not relabeled, converges to some function \(v\) weakly in \(L^{p}(Q;\RR^{d})\).

By periodicity, we have  for every \(D\in \Ocal(\R^N)\) that  the
sequence  \((v_k)_{k\in\NN}\) converges  weakly in \(L^{p}(D;\RR^{d})\)
to the periodic extension of \(v\), still denoted by \(v\).
Since the embedding of \(L^{p}(D;\RR^{d})\) into \(W^{-1,p}(D;\RR^{d})\)
is compact 
and \(\Acal v_k=0\) in \(\R^N\) for every \(k\), we deduce 
that \(\Acal v=0\) in \(\R^{N}\) and  that \((v_k)_{k\in\NN}\)
converges
to \(v\) in
the topology induced by \(\Vert\cdot\Vert_D^\Acal\) on \(L^p(D;\RR^d)\).

Moreover, since  \(\int_{Q}v_k \,dx=0\) for every \(k\), we have
also \(\int_{Q}v\, dx=0\).
Therefore, by \eqref{eq:Mcset} and by \(\Gamma\)-convergence, we
have
\begin{equation*}
M(f,\xi,Q)\le F(\xi+v,Q)\le \liminf_{k\to\infty}F_k(\xi+v_k,Q).
\end{equation*}
Together with \eqref{Fkuk<M+1/k 2} and \eqref{Fk(v)<Fk(u)+e},
the preceding estimate gives \eqref{liminf cube eta}. Since \(\eps>0\) is arbitrary,
 we obtain \eqref{sup eta liminf} from \eqref{liminf cube eta}.
\end{proof}

By analogy with Theorem~\ref{thm:f from M}, we are now ready to present the characterization of the \(\Gamma\)-convergence
of the functionals associated with \(f_k\) by means of the behavior
of \(M^\eta_c(f_k,\xi,Q)\) on small cubes \(Q\).

\begin{theorem}\label{thm:f from M eta}
Let \((f_{k})_{k\in\NN}\) be a sequence in \(\Fcal_{\lip}\), let \((F_k)_{k\in\NN}\)
 be the corresponding sequence of 
functionals in \(\Ical\) defined by
\eqref{eq:Ffree}, and let \((F_k^{\Acal})_{k\in\NN}\)  be the
sequence obtained as in Definition \ref{FAcal}.  Suppose that
\begin{itemize}
\item[(a)] there exists  \(f\colon\R^N\times\R^d\to[0,+\infty)\)
such that
%!TEX encoding = UTF-8 Unicode
\begin{equation}\label{eq:f from M eta}
f(x,\xi)=\limsup_{\rho\to 0^+} \,\sup_{\eta>0}\, \liminf_{k\to\infty}
\frac{M^\eta_c(f_k,\xi,Q_\rho(x))}{\rho^N}
=\limsup_{\rho\to 0^+} \,\sup_{\eta>0}\, \limsup_{k\to\infty}
\frac{M^\eta_c(f_k,\xi,Q_\rho(x))}{\rho^N}
\end{equation}
for a.e.\ \(x\in \R^N\) and every \(\xi\in \R^d\).
\end{itemize}
Then,  there exists \(\hat f\in\Fcal_{\lip}\cap\Fcal_{\qc}\) such that \(f(x,\xi)=\hat f(x,\xi)\)  for a.e.\ \(x\in \R^N\)
and every \(\xi\in \R^d\),   and the functionals  \(F\) and  \(F^{\Acal}\)
introduced in \eqref{eq:Ffree} and 
 Definition \ref{FAcal} satisfy the following properties:
\begin{itemize}
\item[(b)] for every \(D\in\Ocal(\RR^N)\), the sequence  \((F_k(\cdot,D))_{k\in
\NN}\)
\(\Gamma\)-converges to
 \(F(\cdot,D)\) with respect to the topology induced by \(\Vert\cdot\Vert_{D}^{\Acal}\)
on \(L^p(D;\RR^d)\);
\item[(c)]
for every \(D\in\Ocal(\R^N)\), the sequence \((F_k^{\Acal}(\cdot,D))_{k\in\NN}\)
\(\Gamma\)-converges to \(F^{\Acal}(\cdot,D)\) in \(\ker\Acal_D\)
with respect to
the weak topology of \(L^p(D;\RR^d)\).
\end{itemize}

Conversely, if \(f\in\Fcal\) and the functionals  \(F\) and 
\(F^{\Acal}\) introduced in \eqref{eq:Ffree} and 
Definition \ref{FAcal} satisfy (b) or (c), then \(f\) satisfies
(a).
\end{theorem}

\begin{proof}  Since \(F_k\in\Fcal_{\lip}\) for every \(k\in\NN\) by Remark \ref{rem:Fcal gives Ffrac},  the equivalence between (b) and (c) is proved in
Theorem \ref{main constrained}.
If these conditions are satisfied, then by Corollary~\ref{cor:UCmain}
 there exists \(\hat f\in\Fcal_{\lip}\cap \Fcal_{\qc}\) such that 
\begin{equation*}
F(u,D):=\int_D \hat f(x, u(x))\, dx, \quad\text{for every }D\in\Ocal(\R^N)\text{ and }u\in L^p(D;\R^d),
\end{equation*}
which implies that
\(f(x,\xi)=\hat f(x,\xi)\)  for a.e.\ \(x\in \R^N\)
and every \(\xi\in \R^d\). 

Assume (b). Fix \(x\in\R^N\) and \(\xi\in \R^d\). By Proposition
\ref{M and Mceta}, we have for every \(0<\rho\leq 1\) that
\begin{align}
&\displaystyle \sup_{\eta>0}\,\limsup_{k\to\infty}M^\eta_c(f_{k},\xi,Q_\rho(x))\leq
M_c(f,\xi,Q_\rho(x)),\label{eq:from1}
\\
&\displaystyle M(f,\xi,Q_\rho(x))\le 
\sup_{\eta>0}\,\liminf_{k\to\infty}M^\eta_c(f_{k},\xi,Q_\rho(x)).\label{eq:from2}
\end{align}
Using Theorem~\ref{thm:formula}, \eqref{eq:from1}, and \eqref{eq:from2}, we conclude that
\begin{align*}
&\hat f(x,\xi) = \lim_{\rho\to 0^+} \frac{M(\hat f,\xi, Q_{\rho}(x))}{\rho^N}  = \lim_{\rho\to 0^+} \frac{M(f,\xi, Q_{\rho}(x))}{\rho^N}
\leq \limsup_{\rho\to 0^+} \sup_{\eta>0} \liminf_{k\to\infty}
\frac{M^\eta_c(f_k,\xi, Q_{\rho}(x))}{\rho^N}
\\
&\leq \limsup_{\rho\to 0^+} \sup_{\eta>0} \limsup_{k\to\infty}
\frac{M^\eta_c(f_k,\xi, Q_{\rho}(x))}{\rho^N} \leq \limsup_{\rho\to 0^+} \frac{M_c(f,\xi, Q_{\rho}(x))}{\rho^N}=
\limsup_{\rho\to 0^+} \frac{M_c(\hat f,\xi, Q_{\rho}(x))}{\rho^N}=\hat f(x,\xi).
\end{align*}
Thus, \eqref{eq:f from M eta} holds for a.e.~\(x\in \R^N\) and every \(\xi\in \R^d\), 
concluding the proof of (a).

Assume (a). By Corollary \ref{cor:UCmain},  there exists a subsequence
\((f_{k_j})_{j\in\NN}\) and a function \(\hat f\in\Fcal_{\lip}\cap\Fcal_{\qc}\)
 such that  for every \(D\in\Ocal(\RR^N)\), the sequence  \((F_{k_j}(\cdot,D))_{j\in\NN}\)
\(\Gamma\)-converges to
 \(\widehat F(\cdot,D)\) with respect to the topology induced
by \(\Vert\cdot\Vert_{D}^{\Acal}\) on \(L^p(D;\RR^d)\),
 where \(\widehat F\) is the functional associated with \(\hat
f\) by \eqref{eq:Ffree}.
Since (b)\(\Rightarrow\)(a), as proved above, we have for a.e.\ \(x\in \R^N\)
and every \(\xi\in \R^d\) that
\begin{equation*}
\hat f(x,\xi)=\limsup_{\rho\to 0^+} \,\sup_{\eta>0}\, \liminf_{j\to\infty}
\frac{M^\eta_c(f_{k_j},\xi,Q_\rho(x))}{\rho^N}
=\limsup_{\rho\to 0^+} \,\sup_{\eta>0}\, \limsup_{j\to\infty}
\frac{M^\eta_c(f_{k_j},\xi,Q_\rho(x))}{\rho^N}.
\end{equation*}
By \eqref{eq:f from M eta}, this implies that \(\hat f(x,\xi)=
f(x,\xi)\) for a.e.\ \(x\in \R^N\) and every \(\xi\in \R^d\);
hence,
\(f\in\Fcal_{\qc}\) and \(\widehat F=F\).
Since the  \(\Gamma\)-limit does not depend on the subsequence,
we obtain (b) from the Urysohn property of \(\Gamma\)-convergence (see \cite[Proposition
8.3]{DM93}).
\end{proof}

\section{Homogenization without periodicity assumptions}

Throughout this section, we fix a function \(f\in\Fcal_{\lip}\).
 We observe that if the vector space \(\mathrm{span}(\Lambda)\) generated by the wave cone \(\Lambda\) coincides with \(\R^d\)
then every \(f\in\Fcal_{\qc}\) belongs to \(\Fcal_{\lip}\) by Remark~\ref{rem: qc implies Lip}. 
 For
every \(\eps>0\), we consider the functions \(f_\eps\in\Fcal\)
defined by
\begin{equation*}%\label{f x/eps}
f_\eps(x,\xi):= f(\tfrac x\eps,\xi)
\end{equation*}
for every \(x\in\R^N\) and \(\xi\in\R^d\), the functionals \(F_\eps\in\Ical\)
associated with \(f_\eps\) by \eqref{eq:Ffree}, and the corresponding
\(\Acal\)-free functionals \(F_\eps^{\Acal}\) introduced in Definition
\ref{FAcal}.  Note that \(f_\eps\in\Fcal_{\lip}\) for every \(\eps>0\). 

The following theorem provides very general conditions on the
function \(f\) which ensure that there exists a function
\(f_{\ho}\in\Fcal\), independent of \(x\), such that for every
\(D\in\Ocal(\R^N)\), the  family of functionals \((F_\eps(\cdot,D))_{\eps>0}\)
 \(\Gamma\)-converges as \(\eps\to 0^+\) to the functional \(F_{\ho}(\cdot,D)\)
corresponding to \(f_{\ho}\).  By this we mean 
that,  for every sequence \((\eps_k)_{k\in\NN}\) of positive numbers
converging to \(0\), the sequence 
\((F_{\eps_k}(\cdot,D))_{k\in \NN}\) \(\Gamma\)-converges to
\(F_{\ho}(\cdot,D)\).

\begin{theorem}\label{thm:homogenization eta}
Suppose that for every \(x\in\R^N\), \(\xi\in \Q^d\), and \(k\in\NN\),
the limit
\begin{equation}\label{lim large cubes eta}
f_{\ho}^k(\xi):=\lim_{r\to+\infty} \frac{M^{1/k}_c(f,\xi,Q_r(rx))}{r^N}
\end{equation}
exists and is independent of \(x\)  (see Definition \ref{def:Ueta}).
Then, \(f_{\ho}^k\) can be extended as a continuous function on
\(\R^d\), which we still denote by
\(f_{\ho}^k\), and \eqref{lim large cubes eta} holds for every
\(\xi\in\R^d\).
Let \(f_{\ho}\colon \R^d\to [0,+\infty)\) be the function defined
by
\begin{equation}\label{eq:def fhom}
f_{\ho}(\xi):=\sup_{k\in\NN}f_{\ho}^k(\xi)= \lim_{k\to \infty}
f_{\ho}^k(\xi)
\end{equation}
for every \(\xi\in\R^d\).
Then, \(f_{\ho}\in\Fcal_{\lip}\cap\Fcal_{\qc}\) and the following properties hold:
\begin{itemize}
\item[(a)] for every \(D\in\Ocal(\RR^N)\), the family  \((F_\eps(\cdot,D))_{\eps>0}\)
\(\Gamma\)-converges as \(\eps\to 0^+\) to
 \(F_{\ho}(\cdot,D)\) with respect to the topology induced by
\(\Vert\cdot\Vert_{D}^{\Acal}\)
on \(L^p(D;\RR^d)\);
\item[(b)]
for every \(D\in\Ocal(\R^N)\),  the family  \((F_\eps^{\Acal}(\cdot,D))_{\eps>0}\)
\(\Gamma\)-converges as \(\eps\to 0^+\) to
 \(F_{\ho}^{\Acal}(\cdot,D)\) in \(\ker\Acal_D\) with respect
to
the weak topology of \(L^p(D;\RR^d)\).
\end{itemize}
\end{theorem}

We will see in Section \ref{sect:stochastic} that \eqref{lim
large cubes eta}
is satisfied almost surely under the standard hypotheses of
stochastic homogenization. In particular, it is  satisfied
when \(x\mapsto f(x,\xi)\) is \(Q_{1}(0)\)-periodic for every \(\xi\in\R^{d}\).
The following proposition examines another simple case in which \eqref{lim
large cubes eta} holds: \(f\) is the sum
of a periodic function with respect to \(x\) and  a function
whose support has compact projection onto \(\R^{N}\).

\begin{proposition}\label{periodic + compact support} Assume
that \(f\) can be written as
\begin{equation}\label{fper + fc} 
f=f_{per}+f_{comp},
\end{equation}
where \(f_{per},\,f_{comp}\in\Fcal_{\lip}\) satisfy  the two following
properties:
\begin{align}
(a)\enspace &x\mapsto f_{per}(x,\xi)\hbox{ is }Q_{1}(0)\hbox{-periodic for
every }\xi\in\R^{d};
\label{fper}
\\
(b)\enspace&\hbox{there exists }R>0 \hbox{ such that }f_{comp}(x,\xi)=0\hbox{
for every }x\in \R^{N}\setminus Q_R(0)\hbox{ and every }\xi\in
\R^{d}.
\label{compact support}
\end{align}
Then, for every \(x\in\R^N\), \(\xi\in \Q^d\), and \(k\in\NN\),
the limit
\begin{equation}\label{lim large cubes eta 2}
f_{\ho}^k(\xi):=\lim_{r\to+\infty} \frac{M^{1/k}_c(f,\xi,Q_r(rx))}{r^N}
\end{equation}
exists and is independent of \(x\).
\end{proposition}

\begin{proof} By \eqref{fper}, we can apply Theorem \ref{prop:
stocf} to  \(f_{per}\), considered as a 
stochastically periodic random integrand (independent of \(\omega\)),
and we obtain
 for every \(k\in\N\), \(x\in\R^{N}\), and \(\xi\in\Q^{d}\)
that the limit
\begin{equation}\label{lim large cubes eta per}
f_{\ho}^k(\xi):=\lim_{r\to+\infty} \frac{M^{1/k}_c(f_{\per},\xi,Q_r(rx))}{r^N}
\end{equation}
exists and is independent of \(x\). We remark that  since probability
is not involved here, this result can be obtained
directly by  adapting some arguments of  \cite{BrFoLe00}. 

We claim that
\begin{equation}\label{lim large cubes eta per + comp}
f_{\ho}^k(\xi)=\lim_{r\to+\infty} \frac{M^{1/k}_c(f,\xi,Q_r(rx))}{r^N}
\end{equation}
for every \(k\in\N\), \(x\in\R^{N}\), and \(\xi\in\Q^{d}\).
Since \(f_{comp}\in\Fcal\), we have
\(0\le f_{comp}\) by Definition \ref{def:Frak}. Hence, \(f_{per}\le f\) by \eqref{fper + fc}.
This implies that
\(M^{1/k}_c(f_{per},\xi,Q_r(rx))\le M^{1/k}_c(f,\xi,Q_r(rx))\),
which, together with 
\eqref{lim large cubes eta per}, yields
\begin{equation}\label{liminf large cubes eta per + comp}
f_{\ho}^k(\xi)\le \liminf_{r\to+\infty} \frac{M^{1/k}_c(f,\xi,Q_r(rx))}{r^N}.
\end{equation}

In order to prove that
\begin{equation}\label{limsup large cubes eta per + comp}
\limsup_{r\to+\infty} \frac{M^{1/k}_c(f,\xi,Q_{r}(rx))}{r^N}\le
f_{\ho}^k(\xi),
\end{equation}
we fix  \(k\in\N\), \(x=(x_{1},\dots,x_{N})\in\R^{N}\), and \(\xi\in\Q^{d}\).
Similarly to the proof of Corollary~\ref{cor:pequiint}, given \(m\in\N\), we set
\(A_{m}:= \{1,\dots, m\}^{N}\) and for every \(\alpha=(\alpha_{1},\dots,\alpha_{N})\in
A_{m}\), we consider the point
\begin{equation}\label{x(alpha)}
x(\alpha):=(x_{1},\dots,x_{N}) + \Big(-\frac{1}{2} - \frac{1}{2m} + \frac{\alpha_1}{m},
\dots, -\frac{1}{2} - \frac{1}{2m} + \frac{\alpha_N}{m}\Big).
\end{equation}
We observe that
\[%\begin{equation}\label{subcubes}
\overline{Q_1(x)}=\bigcup_{\alpha\in A_{m}} \overline{Q_{1/m}(x(\alpha))}\quad\hbox{and}\quad
Q_{1/m}(x(\alpha))\cap Q_{1/m}(x(\beta))=\emptyset\hbox{ for
}\alpha\neq \beta.
\]%\end{equation}
Hence, for every \(r>0\),
\begin{equation}\label{subcubes Qr}
\overline{Q_r(rx)}=\bigcup_{\alpha\in A_{m}} \overline{Q_{r/m}(r
x(\alpha))}\quad\hbox{and}\quad Q_{r/m}(r x(\alpha))\cap 
Q_{r/m}(r(x(\beta))=\emptyset\hbox{ for }\alpha\neq \beta.
\end{equation}

Given \(r\ge m R\), we set
\begin{equation}\label{intersecting subcubes}
A^{r,R}_{m}:=\big\{\alpha \in A_{m}\colon Q_{r/m}(r x(\alpha))\cap Q_{R}(0)\neq
\emptyset\big\}.
\end{equation}
From \eqref{x(alpha)}, we deduce that \(\alpha\in A^{r,R}_{m}\)
if and only if for every \(j=1,\dots,N\), we have
\begin{equation}\label{intersecting intervals}
\Big(r x_j-\frac{r}{2} - \frac{r}{m} + \frac{r\alpha_j}{m}, 
r x_j-\frac{r}{2} + \frac{r\alpha_j}{m}\Big)
\cap
\Big(-\frac{R}{2}, \frac{R}{2}\Big)
\neq\emptyset.
\end{equation}
Since  the intervals depending on \(\alpha_j\)
in the previous formula are pairwise disjoint and their length is \(\frac{r}{m}\ge R\),
for every \(j=1,\dots,N\), there are at most two elements \(\alpha_{j}\in
\{1,\dots,m\}\) such that \eqref{intersecting intervals} holds.
This implies that the number \(\#A^{r,R}_{m}\) of elements of
\(A^{r,R}_{m}\) satisfies
\begin{equation}\label{number J}
\#A^{r,R}_{m}\le 2^{N}.
\end{equation}

Given \(\delta>0\), we use Definition \ref{def:Ueta} to find
for every \(\alpha\in
A_{m}\) a function  \(u_{\alpha}\in \Ucal_{c}^{1/k}(Q_{r/m}(rx(\alpha)))\)
such that
\begin{equation}\label{ui quasi minimizer}
\int_{Q_{r/m}(r x(\alpha))}f_{per}(y, \xi+u_{\alpha}(y))\,dy< M^{1/k}_c(f_{\per},\xi,Q_{r/m}(rx(\alpha)))+\delta
\frac{r^{N}}{m^{N}}.
\end{equation}
By \eqref{subcubes Qr}, we can define  \(u\colon Q_{r}(rx)\to
\R^{d}\) by setting \(u(y):=0\) for every \(\alpha \in A^{r,R}_{m}\)
and  \(u(y):=u_{\alpha}(y)\) for every \(\alpha\in A_{m}\setminus
A^{r,R}_{m}\).
Recalling Definition \ref{def:Ueta} and \eqref{subcubes Qr}, we
see that  \(u \in \Ucal_{c}^{1/k}(Q_{r}(rx))\), and so \[
M^{1/k}_c(f,\xi,Q_{r}(rx))\le \int_{Q_{r}(r x)}f(y, \xi+u(y))\,dy.
\]
By the definition of \(u\) and by \eqref{eq:pgrowth}, \eqref{fper
+ fc}, \eqref{compact support},  \eqref{intersecting subcubes},
\eqref{number J}, and \eqref{ui quasi minimizer}, we have 
\begin{equation}\label{74rtUV}
\begin{aligned}
M^{1/k}_c(f,\xi,Q_{r}(rx))&\le \sum_{\alpha\in A^{r,R}_{m}} \int_{Q_{r/m}(r
x(\alpha))}f(y, \xi)\,dy +
\sum_{\alpha\in A_{m}\setminus A^{r,R}_{m}} \int_{Q_{r/m}(r x(\alpha))}f_{per}(y,
\xi+u_{\alpha}(y))\,dy
\\
&\le 2^{N}c_{0}(1+|\xi|^{p})\frac{r^{N}}{m^{N}}+\delta r^{N}+
\sum_{\alpha\in A_{m}} M^{1/k}_c(f_{\per},\xi,Q_{r/m}(rx(\alpha))).
\end{aligned}
\end{equation}
Since  \(Q_{r/m}(rx(\alpha))=Q_{r/m}((r/m)(mx(\alpha))\),
the equality  
\[M^{1/k}_c(f_{\per},\xi,Q_{r/m}(rx(\alpha)))=M^{1/k}_c(f_{\per},
\xi,Q_{r/m}((r/m)(mx(\alpha))),\]
 together with \eqref{lim large cubes eta per}, leads to
\[
\lim_{r\to+\infty}\frac{M^{1/k}_c(f_{\per},\xi,Q_{r/m}(rx(\alpha)))}{r^{N}}
=\frac{1}{m^{N}}f^{1/k}_{hom}(\xi).
\]
From this equality and from \eqref{74rtUV}, we get
\[
\limsup_{r\to+\infty}\frac{M^{1/k}_c(f,\xi,Q_{r}(rx)}{r^{N}}\le
c_{0}(1+|\xi|^{p})\frac{2^{N}}{m^{N}}  + \delta + f^{1/k}_{hom}(\xi).
\]
Taking the limit as \(m\to+\infty\) and \(\delta\to0\), we obtain
\eqref{limsup large cubes eta per + comp}, which, together with
\eqref{liminf large cubes eta per + comp}, yields \eqref{lim large
cubes eta per + comp}, concluding the proof of
\eqref{lim large cubes eta 2}.
\end{proof}

\begin{proof}[Proof of Theorem \ref{thm:homogenization eta}]
By \eqref{eq:pLip Mceta}, we have
  for every \(\xi_1\), \(\xi_2\in \R^d\)
and every \(k\in\NN\) that%
\begin{equation*}%\label{eq:pLip fhom eta}
 \Big|\frac{M^{1/k}_c(f,\xi_1,Q_r(rx))}{r^N}-  \frac{M^{1/k}_c(f,\xi_2,Q_r(rx))}{r^N}\Big|
\le  c_5\big(1+ |\xi_1| +|\xi_2|\big)^{p-1}|\xi_2-\xi_1|.
\end{equation*}
Hence,
\begin{equation*}%\label{eq:pLip fhom eta2}
|f_{\ho}^k(\xi_1)- f_{\ho}^k(\xi_2)| 
\le c_5\big(1+ |\xi_1| +|\xi_2|\big)^{p-1}|\xi_2-\xi_1|
\end{equation*}
for every \(\xi_1\) and \(\xi_2\in \Q^d\). This implies that
\(f_{\ho}^k\) can be extended as a continuous function on \(\R^d\),
which we still denote by
\(f_{\ho}^k\), and that \eqref{lim large cubes eta} holds for
every \(\xi\in\R^d\).

As we show next, by Definition \ref{def:Ueta}   and a change of variables,
we have  for
every \(x\in\R^N\), \(\xi\in \R^d\), \(\rho>0\), and \(k\in\NN\),
 that
\begin{equation}\label{Mc homogenization}
M^{1/k}_c\big(f_\eps,\xi,Q_\rho(x)\big)=\eps^N M^{1/k}_c\big(f,\xi,
Q_{\sfrac{\rho}{\eps}}(\tfrac{x}{\eps})\big)=
\rho^N\Big(\frac\eps\rho\Big)^N M^{1/k}_c\big(f,\xi,
Q_{\sfrac{\rho}{\eps}}\big(\tfrac{\rho}{\eps}(\tfrac{x}{\rho})\big)\big).
\end{equation}
In fact, let \(u_\epsi \in \Ucal^{1/k}_c  (Q_{\sfrac{\rho}{\eps}}(\tfrac{x}{\eps}))\)
and \(V_\epsi \in \Vcal^{1/k}  (Q_{\sfrac{\rho}{\eps}}(\tfrac{x}{\eps}))\)
be such that
\begin{equation*}
\begin{aligned}
& \int_{Q_{\sfrac{\rho}{\eps}}(\tfrac{x}{\eps})}|V_\epsi(y)|^p\,dy
\leq \tfrac1k|Q_{\sfrac{\rho}{\eps}}(\tfrac{x}{\eps})| =\tfrac1k(\tfrac{\rho}{\epsi})^N , \\
&-\sum_{i=1}^N \int_{Q_{\sfrac{\rho}{\eps}}(\tfrac{x}{\eps})}
A^i u_\epsi(y) \cdot \partial_i\psi(y) \, dy = \int_{Q_{\sfrac{\rho}{\eps}}(\tfrac{x}{\eps})}
V_\epsi(y) \cdot \nabla\psi(y)\, dy \quad\text{for every }\psi\in W^{1,q}(Q_{\sfrac{\rho}{\eps}}(\tfrac{x}{\eps});\RR^{l}).
\end{aligned}
\end{equation*}
Define
\begin{equation*}
\begin{aligned}
w_\epsi(z):= u_\epsi(\tfrac{z}{\epsi}) \quad \text{ and } \quad
 W_\epsi(z):= V_\epsi(\tfrac{z}{\epsi}) \quad \text{ for } z\in
 Q_\rho(x).
 \end{aligned}
\end{equation*}
Then, using a change of variables and the definition of \(\Ucal^{1/k}_c (\cdot) \) and \(\Vcal^{1/k}
(\cdot) \), it can be checked that \(w_\epsi
\in L^p(Q_\rho(x);\RR^d)\), \(\supp w_\epsi \subset\subset Q_\rho(x) \), and   \(W_\epsi
\in L^p(Q_\rho(x);\RR^{l\times N})\). Moreover,
\begin{equation*}
\begin{aligned}
&\int_{Q_\rho(x)} w_\epsi(z)\, dz  = \int_{Q_\rho(x)} u_\epsi(\tfrac{z}{\epsi})\, dz =\epsi^N \int_{Q_{\sfrac{\rho}{\eps}}(\tfrac{x}{\eps})} u_\epsi(y)\,
dy =0,\\
&\int_{Q_\rho(x)} |W_\epsi(z)|^p\,dz = \int_{Q_\rho(x)} |V_\epsi(\tfrac{z}{\epsi})|^p\,dz
= \epsi^N \int_{Q_{\sfrac{\rho}{\eps}}(\tfrac{x}{\eps})} |V_\epsi(y)|^p\,\,
dy\leq \tfrac1k \rho^N=\tfrac1k|Q_\rho(x)|,\\
& \int_{Q_\rho(x)} f(\tfrac{z}{\epsi}, \xi+w_\epsi(z))\,dz =
\int_{Q_\rho(x)} f(\tfrac{z}{\epsi}, \xi+u_\epsi(\tfrac{z}{\epsi}))\,dz =\epsi^N \int_{Q_{\sfrac{\rho}{\eps}}(\tfrac{x}{\eps})} f(y,\xi+u_\epsi(y)\,dy.
\end{aligned}
\end{equation*}
Furthermore, given \(\theta\in W^{1,q}(Q_{\rho}(x);\RR^{l})\),
we define 
\(\psi\in
W^{1,q}(Q_{\sfrac{\rho}{\eps}}(\tfrac{x}{\eps});\RR^{l})\) by
setting \(\psi(y):=\tfrac1\epsi \theta(\epsi y)\) for \(y\in Q_{\sfrac{\rho}{\eps}}(\tfrac{x}{\eps})
\), and observe that
\begin{equation*}
\begin{aligned}
-\sum_{i=1}^N
\int_{Q_\rho(x)} A^i w_\epsi(z) \cdot \partial_i\theta(z) \, dz &= -\sum_{i=1}^N
\int_{Q_\rho(x)} A^i u_\epsi(\tfrac{z}{\epsi}) \cdot \partial_i\theta(z) \,dz =-\epsi^N\sum_{i=1}^N
\int_{Q_{\sfrac{\rho}{\eps}}(\tfrac{x}{\eps})} A^i u_\epsi(y) \cdot \partial_i\theta(\epsi y)\,dz\\
&= 
  -\epsi^N\sum_{i=1}^N
\int_{Q_{\sfrac{\rho}{\eps}}(\tfrac{x}{\eps})} A^i u_\epsi(y)
\cdot \partial_i\psi( y)= \epsi^N  \int_{Q_{\sfrac{\rho}{\eps}}(\tfrac{x}{\eps})}
V_\epsi(y) \cdot \nabla\psi(y)\, dy\\
&= \int_{Q_\rho(x)} V_\epsi(\tfrac{z}{\epsi}) \cdot \nabla\theta(z)\,
dz = \int_{Q_\rho(x)} W_\epsi(z) \cdot \nabla\theta(z)\,
dz.
\end{aligned}
\end{equation*}
Hence, recalling  Definition \ref{def:Ueta}, we conclude that  \(M^{1/k}_c\big(f_\eps,\xi,Q_\rho(x)\big)\leq\eps^N M^{1/k}_c\big(f,\xi,
Q_{\sfrac{\rho}{\eps}}(\tfrac{x}{\eps})\big)\). The converse
inequality can be proved similarly, from which  \eqref{Mc homogenization} follows.

Combining \eqref{lim large cubes eta} and \eqref{Mc homogenization},
 recalling that the limit in the former is assumed to exist
and to be independent of \(x\),  we get  for every \(\xi\in\R^d\)
and \(\rho>0\) that\begin{equation*}
f_{\ho}^k(\xi) =\lim_{\eps\to 0^+} \frac{M^{1/k}_c\big(f,\xi,Q_{\sfrac{\rho}{\eps}}(\tfrac{\rho}{\eps}(\tfrac{x}{\rho})\big)}{(\rho/\epsi)^N}
=\lim_{\eps\to 0^+} \frac{M^{1/k}_c\big(f_\eps,\xi,Q_\rho(x)\big)}{\rho^N}.
\end{equation*}
Thus, \eqref{eq:def fhom} gives for every \(\rho>0\) that%
\begin{equation*}
f_{\ho}(\xi) =\sup_{k\in\NN} f_{\ho}^k(\xi)=\sup_{k\in\NN} \lim_{\eps\to 0^+} \frac{M^{1/k}_c\big(f_\eps,\xi,Q_\rho(x)\big)}{\rho^N}=\sup_{\eta>0}
\lim_{\eps\to 0^+} \frac{M^\eta_c\big(f_\eps,\xi,Q_\rho(x)\big)}{\rho^N},
\end{equation*}
where in the last equality we used the monotonicity of \(M^\eta_c\)
with respect to \(\eta\).
 Then, by Theorem \ref{thm:f from M eta}, the function \(f_{\ho}\) belongs
to \(\Fcal_{\lip}\cap\Fcal_{\qc}\)  and (a) and (b) are satisfied.
\end{proof}

\section{Stochastic homogenization}\label{sect:stochastic}
In this section, we study stochastic homogenization problems in
the \(\Acal\)-free setting. 
To this aim, we fix a probability space \((\Omega,\mathcal T,
P)\) 
and a group \((\tau_z)_{z\in\Z^N}\) of \(P\)-preserving transformations
on \((\Omega,\mathcal T, P)\), i.e., 
 a family \((\tau_z)_{z\in\Z^N}\) of \(\mathcal T\)-measurable
bijective maps \(\tau_z\colon\Omega\to\Omega\), with 
 \(P(\tau_z^{-1}(E))=P(E)\) for every \(E\in\mathcal T\) and
every \(z\in\Z^N\), satisfying the
 group property: \(\tau_0={\mathrm {id}}_\Omega\) (the identity
map on \(\Omega\)) and \(\tau_{z+z'}=\tau_z\circ
  \tau_{z'}\) for every \(z,z'\in\Z^N\).
We recall that the group is called ergodic if every set \(E\in
\mathcal T\) with \(\tau_z(E)=E\) for every \(z\in\Z^N\) has
probability \(0\) or \(1\). 

We introduce now the classes of random integrands that we are going
to consider.
\begin{definition}[Stochastically periodic random integrands]\label{sf}
Let \(\SF_{\lip}\) be the collection of functions $f\colon\Omega\times\R^N\times\R^d\to[0,+\infty)$
satisfying the following properties:
\begin{itemize}
\item[(a)] \(f\) is \(\mathcal T\times \mathcal L\times\mathcal
B\)-measurable, where \(\mathcal L\) is the \(\sigma\)-algebra
of Lebesgue measurable subsets of \(\R^N\) and  \(\mathcal B\)
is the Borel \(\sigma\)-algebra of \(\R^d\);
\item[(b)] setting  \(f(\om):=f(\omega,\cdot,\cdot)\), we have
\(f(\om)\in  \Fcal_{\lip}\)  for every \(\omega\in \Omega\);
\item[(c)]  \(f\) is stochastically periodic with respect to
\((\tau_z)_{z\in\Z^N}\), i.e., 
\begin{equation*}%\label{stof}
f(\omega,x+z,\xi)=f(\tau_z(\omega),x,\xi)
\end{equation*}
for every \(\omega\in\Omega\), \(x\in\R^N\), \(z\in\Z^N\), and
\(\xi\in\R^d\).
\end{itemize}
Finally, let  \(\SF_{\qc}\)  be the collection of all functions
\(f\in \SF\) such that
\( f(\omega)\in \Fcal_{\qc}\) for \(P\)-a.e.\ \(\omega\in\Omega\).

 We recall that if the vector space \(\mathrm{span}(\Lambda)\) generated by the wave cone \(\Lambda\) coincides with \(\R^d\) then \(\SF_{\qc}\subset \SF_{\lip}\) by Remark~\ref{rem: qc implies Lip} .
\end{definition}

We now introduce the notion of subadditive process. Let \(\mathcal
R\) be the collection of all rectangles of the form
\begin{equation*}
\begin{aligned}%
&[a,b) :=\{x\in\R^N\colon a_i\le x_i<b_i \text{ for }i=1,\cdots,d\}\quad\text{with
}a,b\in\R^N.
\end{aligned}
\end{equation*}

\begin{definition}[Covariant subadditive process] A covariant
subadditive process with respect to \((\tau_z)_{z\in\Z^N}\) is
a function \(\Phi\colon \Om\times\mathcal R\to[0,+\infty)\) with
the following properties:
 \begin{itemize}
 \item[(a)] (measurability) for every \(R\in \mathcal R\), the
function \(\omega\mapsto\Phi(\omega, R)\) is \(\mathcal T\)-measurable
on \(\Omega\);
 \item[(b)] (covariance) $\Phi(\omega,R+z)=\Phi(\tau_z(\omega),R)$
for every $\omega\in\Omega$, $R\in \mathcal R$, and $z\in\Z^N$;
 \item[(c)] (subadditivity) if $R\in \mathcal R$ and  $(R_i)_{i\in
I}\subset \mathcal R$ is a finite partition of $R$, then
\begin{equation*}
\begin{aligned}
 \Phi(\omega,R)\le \sum_{i\in I}\Phi(\omega,R_i)\quad\text{for
every }\omega\in\Om\,;
\end{aligned}
\end{equation*}
 \item[(d)] (boundedness) there exists $c>0$ such that $0\le\Phi(\omega,R)\le
c|R|$ for every $\omega\in\Om$ and $R\in \mathcal R$.
\end{itemize} 
 \end{definition}
 
We will use the following variant of the Subadditive Ergodic
Theorem \cite[Theorem 2.7]{AkKr81}, see also \cite{DaMo86} and
\cite{LiMi02}. 
 
 \begin{theorem}\label{ergodic}
 Let \(\Phi\) be a  covariant  subadditive process with respect to \((\tau_z)_{z\in\Z^N}\).
Then, there exist a \(\mathcal T\)-measurable 
 function \(\varphi\colon\Omega\to[0,+\infty)\) and a set \(\Om'\in\mathcal
T\), with \(P(\Omega')=1\),   such that 
 \begin{equation*}%\label{muld}
 \lim_{r\to+\infty}\frac{\Phi(\omega,Q(rx,r))}{r^N}=\varphi(\omega)
 \end{equation*}
 for every \(x\in\R^N\) and every \(\omega\in\Omega'\). If, in
addition, 
\((\tau_z)_{z\in\Z^N}\) is ergodic, then \(\varphi\) is constant
\(P\)-a.e.\ in \(\Omega\).
 \end{theorem}

\begin{lemma}\label{lemma:stofg}
Let  \(f\in \SF\), \(\xi\in\R^d\), and \(\eta>0\). Then, the function
 \(\Phi_\xi^\eta\colon\Omega\times\mathcal R\to[0,+\infty)\)
 defined by
\begin{equation*}
\Phi_\xi^\eta(\om, R):=M_c^\eta(f(\omega),\xi, \interiorR)
\end{equation*}
is a  covariant  subadditive process. 
\end{lemma}

\begin{proof}
Using  \eqref{eq:pLip} and  Remark~\ref{rem:Ueta},  we
see
that the infimum
  in \eqref{eq:Metac} defining \(M_c^\eta(f(\omega),\xi, \interiorR)\)  can be
obtained 
using a suitable dense sequence of functions \(u\). Since for
every \(u\in L^p(R;\R^d)\), 
the function \(\omega\mapsto \int_R f(\omega, x, \xi+u(x))\,dx\)
 is  \(\mathcal T\)-measurable by (a) of Definition \ref{sf},
we conclude that  \(\omega\mapsto\Phi_\xi^\eta(\omega,R)\) is
 \(\mathcal T\)-measurable.
The boundedness and  the covariance property are clear.

We now prove subadditivity.
Fix \(\omega\in\Omega\), \(R\in\mathcal R\), a finite
partition \((R_i)_{i\in I}\) of \(R\), and \(\delta>0\). For
every \(i\in I\), there  exist
\(u_i\in L^{p}(R_i;\R^d)\) and \(V_{i}\in  L^{p}(R_i;\RR^{l\times
N}) \)  such
that \(\supp u_i \subset \subset \interiorR{}_i\), \(\int_{R_i}
u_i \,dx = 0\), 
\begin{align*}
&\displaystyle  \langle \widetilde \Acal_{\interiorR{}_i} u_i,\psi
\rangle 
=  \int_{R_{1,i}}  V_{i}\cdot
\nabla \psi\,dx
\quad\text{for every }\psi\in W^{1,q}(\interiorR{}_i;\R^{l}),
%\label{A=divU2}
\\
&\displaystyle  \int_{R_i} |V_{i}|^pdx
< \eta|R_i|,
%\label{int Ri Ui} 
\\
&\displaystyle \int_{R_i}f(\omega, x, \xi +u_i(x))\,dx <M_c^\eta(f(\omega),\xi,
\interiorR{}_i)+ \delta|R_i|.
%\label{ui quasi min}
\end{align*}
We define \(u\in L^{p}(R;\R^d)\)
and \( V\in  L^{p}(R;\RR^{l\times N}) \) by setting 
\(u:=u_i\)
and \( V:=V_{i}\) on
 \(R_i\) for every \(i\in I\). Then, % recalling that
% \(\supp u_i \subset \subset \interiorR{}_i\), using a partition
%of unity we obtain
%
\begin{equation*}
\langle \widetilde\Acal_{\interiorR{}} u,\psi \rangle = 
\int_{R}  V\cdot
\nabla \psi\,dx
\quad\text{for every }\psi\in W^{1,q}(\interiorR{};\R^{l}),
\end{equation*}
Moreover, by additivity, we have    
\begin{align*}
&\int_{R} u\, dx=0,\quad \int_{R} |V |^pdx <
\eta|R|,\quad \text{and}
%\label{int Ri Ui} 
\\
&\displaystyle \int_{R}f(\omega, x, \xi +u(x))\,dx <\sum_{i\in
I}M_c^\eta(f(\omega),\xi, \interiorR{}_i)+ \delta|R|.
%\label{ui quasi min}
\end{align*}
By \eqref{eq:Metac}, we obtain \(M_c^\eta(f(\omega),\xi, \interiorR{})\le
\sum_{i\in I}M_c^\eta(f(\omega),\xi, \interiorR{}_i)+ \delta|R|\).
Hence, due to the arbitrariness of \(\delta>0\), we conclude that \(M_c^\eta(f(\omega),\xi,
\interiorR{})\le \sum_{i\in I}M_c^\eta(f(\omega),\xi, \interiorR{}_i)\),
which proves the subadditivity of \(\Phi_\xi^\eta(\om, R)\).
\end{proof}

We are now in a position to state our main result on stochastic
homogenization of \(\Acal\)-free integral functionals.
Given a stochastic integrand \(f\in\SF_{\lip}\) and  \(\eps>0\),
we consider the stochastic integrands
\(f_\eps\colon\Omega\times\R^N\times\R^d\to [0,+\infty)\) defined
by
\begin{equation*}%\label{f x/eps om}
f_\eps(\omega, x,\xi):= f(\omega, x/\eps,\xi)
\end{equation*}
for every \(x\in\R^N\) and \(\xi\in\R^d\).  Setting  \(f_\eps(\om):=f_\eps(\omega,\cdot,\cdot)\),
we have that \(f_\eps(\om)\in \Fcal_{\lip}\)  for every 
\(\omega\in \Omega\), so we can consider 
 the functionals \(F_\eps(\omega)\in \Ical_{\lip}\) associated with
\(f_\eps(\omega)\) by \eqref{eq:Ffree}, and the corresponding
  \(\Acal\)-free functionals \(F_\eps^{\Acal}(\omega)\) introduced
in Definition \ref{FAcal}.

\begin{theorem}\label{prop: stocf}
Let \(f\in  \SF_{\lip}\)  and, for every \(\epsi>0\) and 
for every \(\omega\in\Omega\), let    \(F_\eps(\omega)\in  \Ical_{\lip}\) and   \(F_\eps^{\Acal}(\omega)\)
be the corresponding functionals. 
Then, there exists
a set \(\Om'\in {\mathcal T}\), with \(P(\Om')=1\), such that
for every  \(k\in\NN\),  \(\omega\in\Omega'\), and \(\xi\in\Q^d\), the limit
\begin{equation}\label{lim large cubes eta omega}
f_{\ho}^k(\omega,\xi):=\lim_{r\to+\infty} \frac{M^{1/k}_c(f(\omega),\xi,Q_r(rx))}{r^N}
\end{equation}
exists and is independent of \(x\). Moreover, for every \(\omega\in\Omega'\),
the function
\(f_{\ho}^k(\omega,\cdot)\) can be extended as a continuous function
on \(\R^d\), which we still denote by
\(f_{\ho}^k(\omega,\cdot)\), and \eqref{lim large cubes eta omega}
holds for every \(\xi\in\R^d\).

Let \(f_{\ho}\colon\Omega\times \R^d\to [0,+\infty)\) be the
function defined for every \(\xi\in\R^d\) by
\begin{equation*}%\label{eq:def fhom om}
f_{\ho}(\omega,\xi):=\begin{cases}
\displaystyle \sup_{k\in\NN}f_{\ho}^k(\omega,\xi)= \lim_{k\to
\infty} f_{\ho}^k(\omega,\xi)&\text{if }\omega\in\Omega'\text{
and }\xi\in\R^d,
\\
\displaystyle  \int_{\Omega'}f_{\ho}(\omega',\xi)dP(\omega')
 &\text{if }\omega\in\Omega\setminus \Omega'\text{ and }\xi\in\R^d.
\end{cases}
\end{equation*}
Then, \(f_{\ho}\in \SF_{\lip}\cap\SF_{\qc}\) and, setting  \(f_{\ho}(\om):=f_{\ho}(\omega,\cdot)\)
 for every \(\omega\in\Omega\),  the corresponding functionals
\(F_{\ho}(\omega)\in \Ical_{\lip}\cap \Ical_{\qc}\) defined by \eqref{eq:Ffree} and
\(F_{\ho}^{\Acal}(\omega)\) introduced in Definition \ref{FAcal}
satisfy
the following properties:
\begin{itemize}
\item[(a)] for every \(\om\in\Omega'\) and every \(D\in\Ocal(\RR^N)\),
the family  \((F_\eps(\omega,\cdot,D))_{\eps>0}\)
\(\Gamma\)-converges as \(\eps\to 0^+\) to
 \(F_{\ho}(\omega,\cdot,D)\) with respect to the topology induced
by \(\Vert\cdot\Vert_{D}^{\Acal}\)
on \(L^p(D;\RR^d)\);
\item[(b)]
for every \(\om\in\Omega'\) and \(D\in\Ocal(\R^N)\),
 the family  \((F_\eps^{\Acal}(\omega,\cdot,D))_{\eps>0}\)
\(\Gamma\)-converges as \(\eps\to 0^+\) to
 \(F_{\ho}^{\Acal}(\omega,\cdot,D)\) in \(\ker\Acal_D\) with
respect to
the weak topology of \(L^p(D;\RR^d)\).
\end{itemize}

 If, in addition, 
\((\tau_z)_{z\in\Z^N}\) is ergodic, then we can select  \(\Omega'\) so that  
\(f_{\ho}^k\) and \(f_{\ho}\) do not depend on \(\omega\).
\end{theorem}
\begin{proof}
The existence of \(\Om'\in {\mathcal T}\), with \(P(\Om')=1\),
such that \eqref{lim large cubes eta omega} holds follows from
Theorem \ref{ergodic} and Lemma \ref{lemma:stofg}. The properties
of \(f_{\ho}^k(\omega,\xi)\) and \(f_{\ho}(\omega,\xi)\) for
\(\omega\in\Omega'\) are given by Theorem \ref{thm:homogenization
eta}, which implies also that (a) and (b) hold and that  \(f_{\ho}(\omega)\in
\Fcal_{\lip}\cap\Fcal_{\qc}\) for every \(\omega\in\Omega\); hence,  \(f_{\ho}\in
\SF_{\lip}\cap
\SF_{\qc}\).

 If
\((\tau_z)_{z\in\Z^N}\) is ergodic, then Theorem \ref{ergodic}
implies that for every \(\xi\in \Q^d\), the function   
\(\omega\mapsto f_{\ho}^k(\omega,\xi)\) is constant \(P\)-a.e.\
in \(\Omega\). Therefore, we can select  \(\Omega'\)
so that  
\(f_{\ho}^k\)  does not depend on \(\omega\) for \(\xi\in\Q^d\),
and so does its continuous extension to \(\R^d\) and its limit
\(f_{\ho}\).
\end{proof}

\section*{Acknowledgements}
 The authors thank Jean-Fran\c cois Babadjian for the proof of Theorem~\ref{th:Babadjian}, which provides the integral representation of \(F\) without assuming a \(p\)-Lipschitz continuity. 
The authors thank the hospitality of the Center of Nonlinear
Analysis (CNA) at Carnegie Mellon University (Pittsburgh, USA) and of Erwin Schr\"odinger
International Institute for Mathematics and Physics (Vienna,
Austria).

 The research of Gianni Dal Maso was partially supported by the
National Research Project PRIN 2022J4FYNJ ``Variational methods
for stationary and evolution problems
with singularities and interfaces'' funded by the Italian Ministry
of University and Research.  When this paper was written, 
Gianni Dal Maso  was  member of GNAMPA of INdAM.

The research of Rita Ferreira was partially supported by King Abdullah University of Science and Technology (KAUST) baseline funds and KAUST
OSR-CRG2021-4674.

The research of Irene Fonseca was partially supported by the
National Science Foundation (NSF) under grants DMS-2108784, DMS-2205627, and DMS-2342349.

%%%%%%%%%%%%%%%%%%%%%%%% BIBLIOGRAPHY %%%%%%%%%%%%%%%%%%%%%%%%%%%

\bibliographystyle{abbrv}

\end{document}